\documentclass[11pt]{amsart}

\usepackage[a4paper, margin=2.5cm]{geometry}
\usepackage{setspace}		
\usepackage{parskip}
\usepackage[inline,final]{showlabels}		% SHOW LABELS

\usepackage{hyperref}
\usepackage{xcolor}
\definecolor{linksc}{RGB}{0,90,120}
\definecolor{citec}{RGB}{40,120,20}
\definecolor{urlc}{RGB}{100,150,80}
\hypersetup{
    colorlinks,
    linkcolor={linksc},
    filecolor=black,
    citecolor={citec},
    urlcolor={urlc}
}
\usepackage[style=alphabetic,backend=bibtex]{biblatex}	%style=verbose-ibid,
\usepackage{amssymb}				% Includes \mathbb
\usepackage{upgreek}
\usepackage{mathrsfs}				% Includes \mathscr
\usepackage{dsfont}					% For double-line numbers
\newcommand{\mc}[1]{\mathcal{ #1 }}

\newcommand{\mb}[1]{\mathbb{ #1 }}

\newcommand{\mbf}[1]{\mathbf{ #1 }}
\newcommand{\RR}{\mathbb{R}}

\usepackage{tikz}
\usetikzlibrary{matrix,arrows,graphs}
\usepackage{tikz-cd} 	% Package for commutative diagrams
\usepackage{graphicx}
\usepackage{subcaption}
\usepackage{adjustbox}
\usetikzlibrary{decorations.pathmorphing,decorations.markings,arrows.meta}
\usetikzlibrary{calc,fit,backgrounds}

\usepackage{amsmath} 	% AMS Math Package
\usepackage{amsthm} 	% Theorem Formatting
\usepackage{mathtools}	% For arrow
\usepackage{array}
\usepackage{cleveref}
\usepackage{enumitem}
\def\beq{\begin{equation}}
\def\eeq{\end{equation}}
\newtheorem{thm}{Theorem}[section]
\newtheorem{cor}[thm]{Corollary}
\newtheorem{lem}[thm]{Lemma}
\newtheorem{prop}[thm]{Proposition}

\newtheorem*{thm*}{Theorem}
\newtheorem*{prop*}{Proposition}

\theoremstyle{definition}
\newtheorem{ex}[thm]{Example}
\newtheorem{dfn}[thm]{Definition}
\newtheorem*{dfn*}{Definition}

\theoremstyle{remark}
\newtheorem{rmk}[thm]{Remark}

\newcommand{\p}{\partial}

\newcommand{\Res}{\operatorname*{Res}}

\newcommand{\cone}{\text{Span}_{\mathbb{R}_+}}
\usepackage{shuffle}
\newcommand{\dstream}[1]{\lceil #1 \rceil}
\newcommand{\ustream}[1]{\lfloor #1 \rfloor}

\newcommand{\dbl}[1]{2#1}
\newcommand{\dblfan}[1]{\Sigma^{\dbl{#1}}}

\title{The Universe Fan}

\author[H. Frost]{Hadleigh Frost}
\address[Institute for Advanced Study, Princeton, 08540 NJ, U.S.A.]
{H. Frost}
\email{frost@ias.edu}
\urladdr{}

\author[F. Lotter]{Felix Lotter}
\address[Max Planck Institute for Mathematics in the Sciences, Leipzig, Germany]
{F. Lotter}
\email{felix.lotter@mis.mpg.de}
\urladdr{}

\begin{document}
\begin{abstract}
\noindent
The wavefunction of the universe, as studied in perturbative quantum field theory, is a rational function whose singularities and factorization properties encode a rich underlying combinatorial structure. We define and study a broad generalization of such wavefunctions that can be associated to any lattice. We obtain these wavefunctions as the Laplace transform of a polyhedral fan, the \emph{universe fan}, whose cones are defined by positivity conditions reflecting a notion of causality in the lattice, and we describe its face lattice.

In the matroid case, the universe fan projects to the nested set fan, and the wavefunctions we define recover the matroid amplitudes introduced by Lam as residues. Moreover, in the case relevant for physics, the positivity conditions give a novel way to study the wavefunction, and we show how it is related to the cosmological polytopes of Arkani-Hamed, Benincasa, Postnikov.

Finally, we study refinements of the universe fan induced by piecewise linear (tropical) functions. The resulting subdivisions project to refinements of the nested set fan and correspond dually to blow-ups of matroid polytopes, generalizing the `cosmohedron' polytope.
\end{abstract}

\maketitle
\setcounter{tocdepth}{1}
\tableofcontents
%%%%%%%%%%%%%%%%%%%%%%
%%%%%%%%%%%%%%%%%%%%%%
%%%%%%%%%%%%%%%%%%%%%%
%%%%%%%%%%%%%%%%%%%%%%
\section{Introduction}
The Bergman fan (or tropical linear space) associated to a matroid $M$ encodes the combinatorics of its flats via the tropicalization of a linear space \cite{FeiStu04,Spe08}.  More prosaically, the Bergman fan of a matroid with ground set $E$ is a fan in $\mathbb{R}^E$. Building sets in the lattice of flats define nested set fans, which are simplicial refinements of the Bergman fan \cite{Pos09,FeiMue05}. The Bergman fan and its refinements play an important role in tropical and discrete geometry. In this paper, we introduce a new family of fans that realize higher-order combinatorial information in the complex of nested sets: the universe fan.

Our motivation to study these fans is the cosmological wavefunction from perturbative quantum field theory. The tree-level cosmological wavefunction $\Psi_n$ is a rational function that describes the density fluctuations produced by quantum fields in the very early universe during the big bang \cite{ArkaniHamed17,Arkani18}. From a combinatorial perspective, the singularity structure of the wavefunction is controlled by partially ordered collections of subgraphs. These features place the cosmological wavefunction naturally alongside objects from tropical and matroidal geometry, including Bergman fans and their refinements. However, while the relation of (tree-level $\phi^3$) amplitudes to matroids and the combinatorics of nested sets has long been understood (see \cite{ABHY} and more recently \cite{Lam24}), a similar description of the cosmological wavefunction has still been missing. This motivates our viewpoint in this paper. Our results are also of interest for the original physical problem. We identify a remarkably simple system of linear positivity conditions whose solution set cuts out the full wavefunction. In particular, our results give partial answers to questions raised in \cite[Section 8]{ArkaniHamed17} and \cite[Section 11]{ArkFigVaz24}.

\subsection{Nested sets, nestable sets, and causal regions}

Fix some lattice $L$ with maximal element $\hat{1}$ and minimal element $\hat{0}$. We consider sets $\mc{G} \subseteq L_{> \hat{0}}$ that satisfy
\[
\mathcal G\text{ contains }\hat 1,\ \text{and for incomparable }f,g\in\mathcal G,\ (f\vee g\notin\mathcal G)\Rightarrow(f\wedge g=\hat 0).
\]
We call $\mc{G}$ a \emph{nestable set}, and study subsets $N \subset \mc{G}$ called \emph{nested sets}. These subsets form a simplicial complex, called the \textit{nested set complex}. We are interested in the cases where this complex (and its restrictions to intervals $[\hat{0},f]$) are pure. When this holds, we call the pair $(L,\mathcal G)$ a \textit{nestoid} (\Cref{dfn:nestoid}).

Nestoids are a relaxation of the notion of building sets. Building sets arose in the study of compactifications of subspace arrangements where they provide a combinatorial framework for organizing iterated blow-ups \cite{conciniprocesi95}.  In a purely combinatorial formulation due to Feichtner and Yuzvinsky, a building set $\mc{G} \subseteq L_{\geq \hat{0}}$ is defined for an arbitrary finite lattice $L$ \cite{feichtner2004chow}. Building sets give rise to nested set complexes and nested set fans, which refine Bergman fans in the case of matroid lattices and sometimes admit polytopal realizations \cite{Pos09,FeiMue05}.

Our main motivation for the introduction of nestoids is that, unlike building sets, this class is `stable' under taking links of the associated nested set complexes, a property made precise in \Cref{lem:nest-factors}.

In this paper, we introduce two rational functions for every nestoid: an \textit{amplitude} and a \textit{cosmological wavefunction}. These functions are obtained as the Laplace transform of unimodular polyhedral fans: the \emph{free nested set fan} and the \emph{universe fan}. Lam \cite{Lam24} defined a class of amplitudes arising from oriented matroids in a similar way, and this can be viewed as a special case of our definition (\Cref{prop:nestoids and matroids}). See also \cite{telen2026toric}.

The \emph{nested sets} $N \subset \mc{G}$ of a nestable set $\mc{G}$ are defined as for building sets. Namely, for any incomparable $f_1,f_2,\ldots \in N$, one requires that their join is not in $\mc{G}$: $f_1 \vee f_2 \vee \cdots \not\in\mc{G}$.
Then the set of nested sets of $\mc{G}$ defines a simplicial complex called the \emph{nested set complex}.
This complex is realized by the \textit{free nested set fan} $\Sigma^\mc{G}$ in $\mathbb{R}^\mc{G}$, the fan with cones $\langle f\,|\, f\in N\rangle$ for each nested set $N$ (Definition \ref{dfn:free nested fans}). 
These cones span unimodular subspaces and so determine a rational function via a Laplace transform (as explained in e.g.\ \cite[Section 10.2]{Lam24}). We define the \emph{amplitude} $A(L,\mc{G})$ associated to $L$ and $\mc{G}$ as the Laplace transform of $\Sigma^\mc{G} / \langle \hat{1} \rangle$.

\begin{figure}
  \centering
  \begin{subfigure}{0.45\textwidth}
    \centering
  	\begin{tikzpicture}[scale=0.9]
  \node (v4) at (0,0) {$\bullet$};
  \node (v1) at (90:2)  {$\bullet$};  % angle:distance
  \node (v2) at (210:2) {$\bullet$};
  \node (v3) at (330:2) {$\bullet$};
  \draw (v4) -- node[midway,left]{$a$} (v1);
  \draw (v4) -- node[midway,above left]{$b$} (v2);
  \draw (v4) -- node[midway,above right]{$c$} (v3);
\end{tikzpicture}
    \caption{}
    \label{intro:star}
  \end{subfigure}
  \begin{subfigure}{0.45\textwidth}
    \centering
    	\begin{tikzpicture}[scale=0.8,node distance=20mm]
  	\node (a) at (0,0)   {a};
  	\node (b) at (2.5,0) {b};
  	\node (c) at (5,0)   {c};
  	\node (d) at (1.25,1.8)   {ab};
  	\node (e) at (3.75,1.8) {bc};
  	\node (f) at (6.25,1.8)   {ac};
  	\node (t) at (2.5,3.4) {$\hat{1}$};
	  \draw (d) -- (a);
	  \draw (d) -- (b);
	  \draw (e) -- (b);
	  \draw (e) -- (c);
	  \draw (f) -- (c);
	  \draw[black!40] (f) -- (a);
	  \draw (t) -- (d);
 	  \draw (t) -- (e);
  	  \draw (t) -- (f);
	\end{tikzpicture}
    \caption{}
    \label{intro:starbuild}
  \end{subfigure}
\caption{The star graph (A), and the associated building set (B) given by the graph's connected subgraphs, where $\hat{1} = 1234$ denotes the whole graph.}
\label{intro:starfig}
\end{figure}
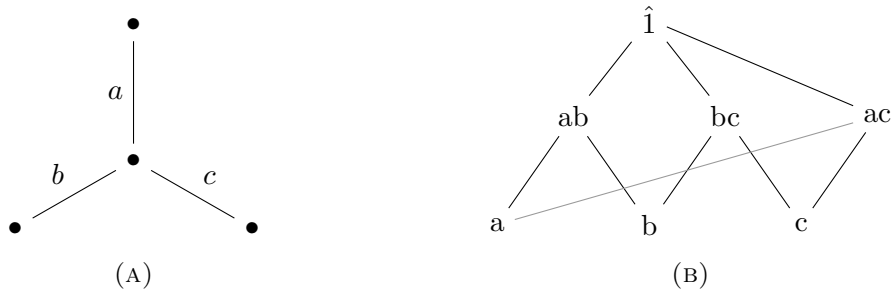

\begin{ex}
Consider the star graph $G$ (Figure \ref{intro:star}) with three edges, $a$, $b$, and $c$. We take the lattice $L$ to be the \emph{lattice of flats} of the associated matroid, which has one flat $f$ for each subgraph of $G$ (see Figure \ref{intro:starbuild}), with maximal element $\hat{1}$ for the whole graph. Write $\mc{G} = L \backslash \hat{0}$ for the \emph{maximal building set}, containing every non-empty flat. The 6 maximal nested sets of $\mc{G}$ take the form $N = \{f,g,\hat{1}\}$, with $f < g < \hat{1}$. In fact, $\mc G$ is an example of a \textit{graphical building set}, and its nested sets are equivalently \emph{tubings} of $G$ that do not use vertex tubes (Figure \ref{intro:start})\footnote{Throughout the paper, by a \textit{tubing} we mean a laminar family of vertex sets of connected subgraphs; in particular, there is no non-adjacency condition for tubes in a tubing.}.
The six maximal cones of $\Sigma^\mc{G}$ are
\[
\langle a, ab, \hat{1} \rangle,~~~ \langle b, ab, \hat{1} \rangle,~~~\langle b, bc, \hat{1} \rangle,~~~ \langle c, bc, \hat{1} \rangle,~~~\langle c,ac, \hat{1} \rangle,~~~ \langle a,ac, \hat{1} \rangle,
\]
and the associated amplitude is
\[
A(L,\mc{G}) = \frac{1}{s_a s_{ab} } + \frac{1}{s_b s_{ab}} + \frac{1}{s_b s_{bc} } + \frac{1}{s_c s_{bc}} + \frac{1}{s_c s_{ac} } + \frac{1}{s_a s_{ac}}.
\]
This is a matroid amplitude in the sense of \cite{Lam24} under the substitution $s_f := \sum_{e \in f} a_e$. This substitution corresponds to a projection $\mathbb R^{\mathcal G} \to \mathbb R^E$ which identifies $\Sigma^{\mc{G}}$ with the nested set fan $\Sigma_{\mathcal G}$ from \cite{feichtner2004chow} (\Cref{prop:nestoids and matroids}).
\end{ex}

\begin{figure}
  \centering
  \begin{subfigure}{0.45\textwidth}
    \centering
\includegraphics[width=0.5\textwidth]{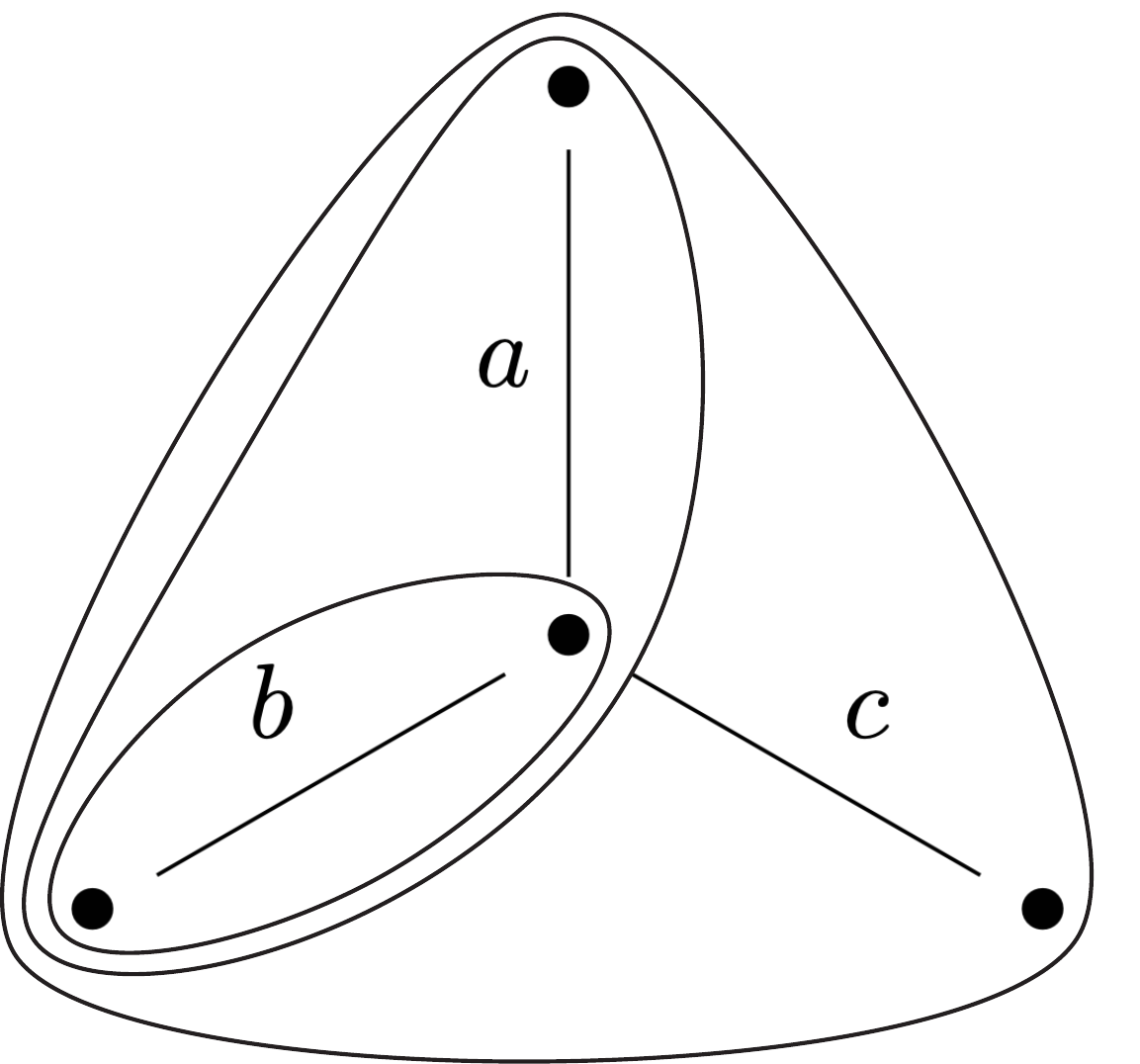}
    \caption{}
    \label{intro:startube}
  \end{subfigure}
  \begin{subfigure}{0.45\textwidth}
    \centering
    	\begin{tikzpicture}[scale=0.8,node distance=20mm]
  	\node (a) at (0,0)   {a};
  	\node[red] (b) at (2.5,0) {b};
  	\node (c) at (5,0)   {c};
  	\node[red] (d) at (1.25,1.8)   {ab};
  	\node (e) at (3.75,1.8) {bc};
  	\node (f) at (6.25,1.8)   {ac};
  	\node[red] (t) at (2.5,3.4) {$\hat{1}$};
	  \draw (d) -- (a);
	  \draw (d) -- (b);
	  \draw (e) -- (b);
	  \draw (e) -- (c);
	  \draw (f) -- (c);
	  \draw[black!40] (f) -- (a);
	  \draw (t) -- (d);
 	  \draw (t) -- (e);
  	  \draw (t) -- (f);
	\end{tikzpicture}
    \caption{}
    \label{intro:startubenest}
  \end{subfigure}
\caption{A nested set $\{b,ab,\hat{1}\}$ of the building set $\mc{G}$ (B) corresponds to a tubing of the graph (A).}
\label{intro:start}
\end{figure}
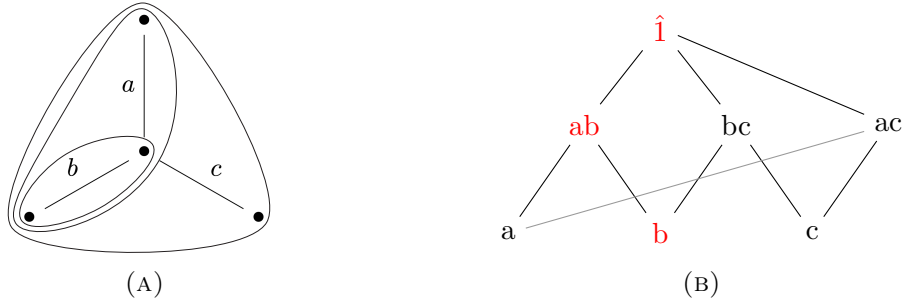

In additional to the free nested set fan, each nestoid $(L,\mathcal G)$ defines a \emph{universe fan} $\mc{U}_\mc{G}$, which encodes more combinatorial data. Consider the vector space $\mathbb{R}^{2\mc{G}}$ with two generators, $f^+$ and $f^-$, for each $f \in \mc{G}$. Then $\mc{U}_\mc{G}$ is a polyhedral fan in $\mathbb{R}^{2\mc{G}}$. In Section \ref{sec:universe}, we derive the generators and face lattice of $\mc{U}_\mc{G}$, which can be taken as a definition of $\mc{U}_\mc{G}$ for the purposes of this introduction. The generators of $\mc{U}_\mc{G}$ are given by nested sets $R \subset \mc{G}$ with one maximal element and all other elements incomparable. We call such a nested set a \emph{causal region}. For each causal region $R=\{r_*,r_1,\dots,r_k\}$, the vector
\[
w_R \;:=\; r_*^+ + \sum_i r_i^- \ \in\ \mathbb R^{\dbl\mathcal G}
\]
is a generator of $\mc{U}_\mc{G}$. Theorem \ref{prop:universe-rays} then shows that the maximal cones of $\mc{U}_\mc{G}$ are
\[
U_N=\cone\bigl(w_R \mid R\subseteq N\text{ a causal region}\bigr),
\]
for each maximal nested set $N$. The face lattice of $\mc{U}_\mc{G}$ is given by Theorem \ref{thm:universe-cones-and-markings}. The cones of $\mc{U}_\mc{G}$ span unimodular subspaces (\Cref{prop:max-lightcone-refine}). By taking the Laplace transform we obtain a rational function $\Psi(L,\mathcal G)\;:=\;\mathcal L(\mathcal U_{\mathcal G})$, the \emph{wavefunction} associated to $(L,\mathcal G)$.

\begin{ex}
For the star graph, as above, the universe fan $\mc{U}_\mc{G}$ in $\mathbb{R}^{2\mc{G}}$ has generators $w_f = f^+$ for each flat $f$ and $w_{\{f,g\}} = f^+ + g^-$ for each nested set $\{f,g\}$ with $g < f$. The 6 maximal cones of $\mc{U}_\mc{G}$ are
\[
U_N = \langle \hat 1^+, f^+, g^+, \hat 1^+ + f^-,  \hat 1^+ + g^-, g^+ + f^- \rangle
\]
for each nested set $N = (f < g < \hat{1})$. The Laplace transform of $\mc{U}_\mc{G}$ defines the \emph{wavefunction}
\[
\Psi(L,\mathcal G) = \sum_{f < g < \hat 1} \frac {E_{\hat 1}^+ + E_{f}^- + E_g^+} {E_{\hat 1}^+  E_f^+ E_g^+ (E_{\hat 1}^+ + E_{f}^-) (E_{\hat 1}^+ + E_g^-)  (E_{g}^+ + E_f^-)},
\]
where we introduce variables $E_f^\pm$ for each $f^\pm$. The wavefunction in this example agrees with a function associated to the triangle graph (the line graph of the star) in the construction of \cite{Glew2025AmplitubesGC}. Moreover, as Figure \ref{intro:regions} illustrates, generators of $\mc{U}_\mc{G}$ can be viewed as regions cut out by tubes. The notion of a causal region generalizes this beyond the graphical case.
\end{ex}

\begin{figure}
  \centering
    \begin{subfigure}{0.32\textwidth}
    \centering
\includegraphics[width=0.6\textwidth]{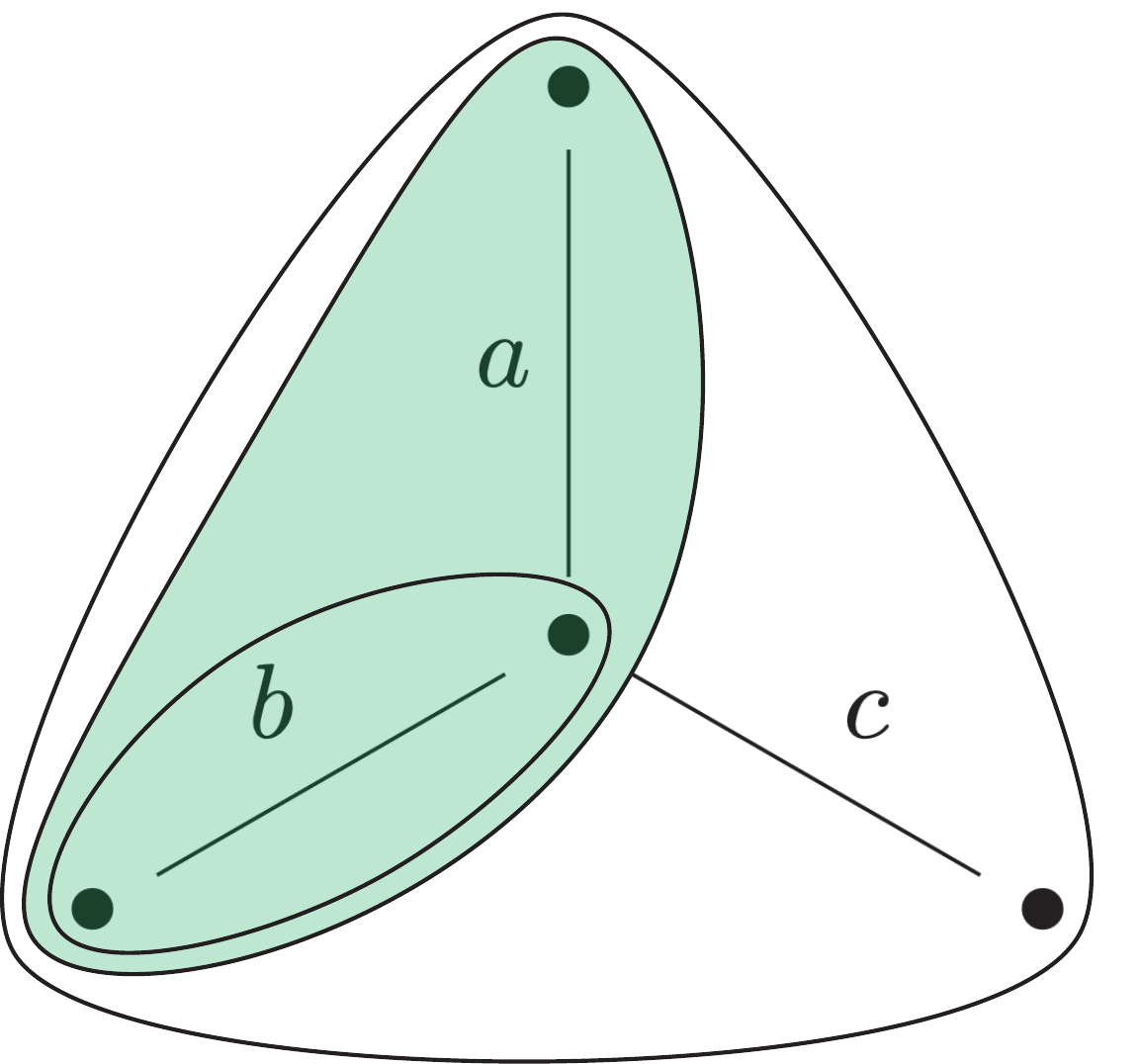}
    \caption{}
    \label{intro:regions2}
  \end{subfigure}
  \begin{subfigure}{0.32\textwidth}
    \centering
\includegraphics[width=0.6\textwidth]{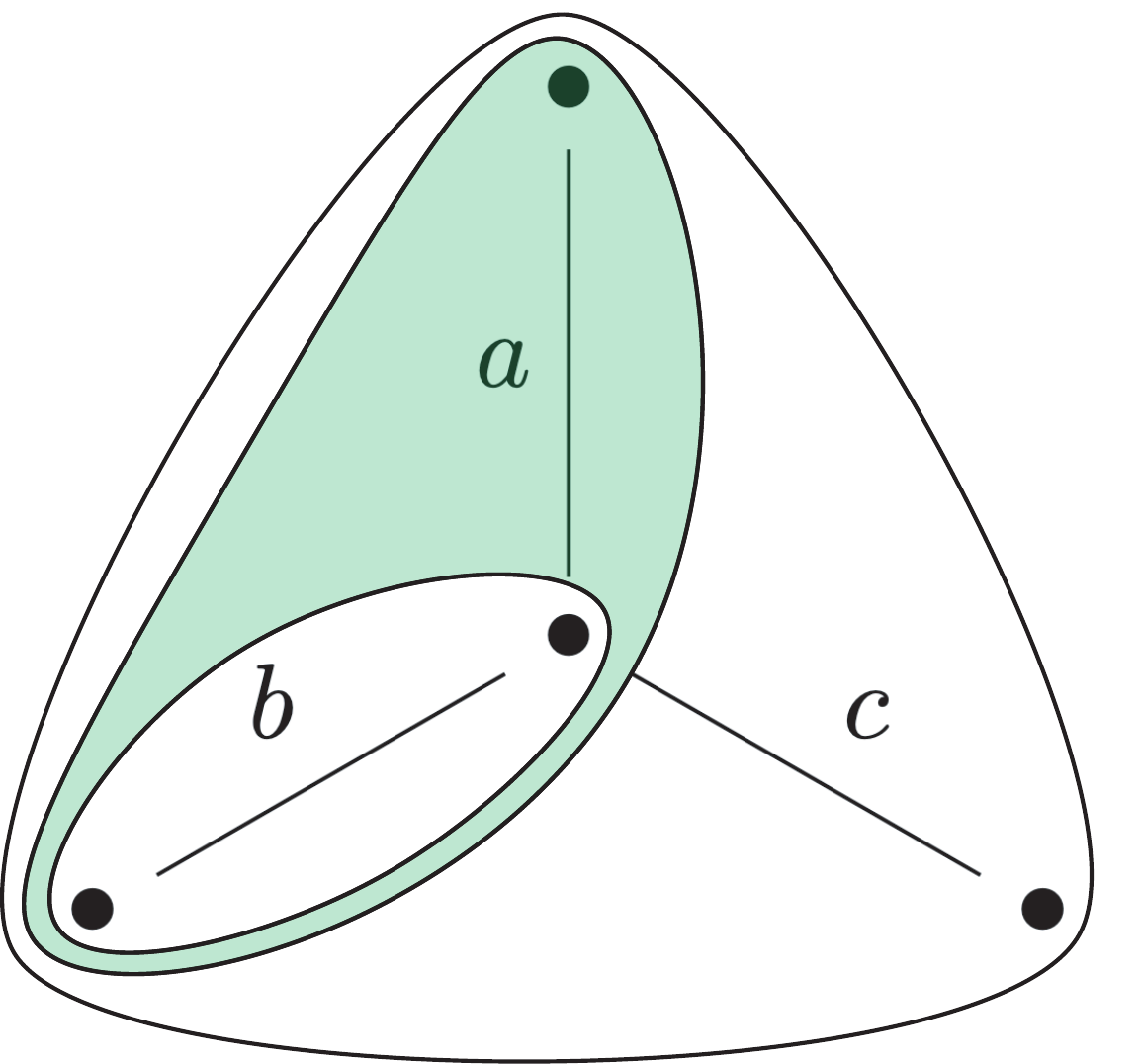}
    \caption{}
    \label{intro:regions1}
  \end{subfigure}
    \begin{subfigure}{0.32\textwidth}
    \centering
\includegraphics[width=0.6\textwidth]{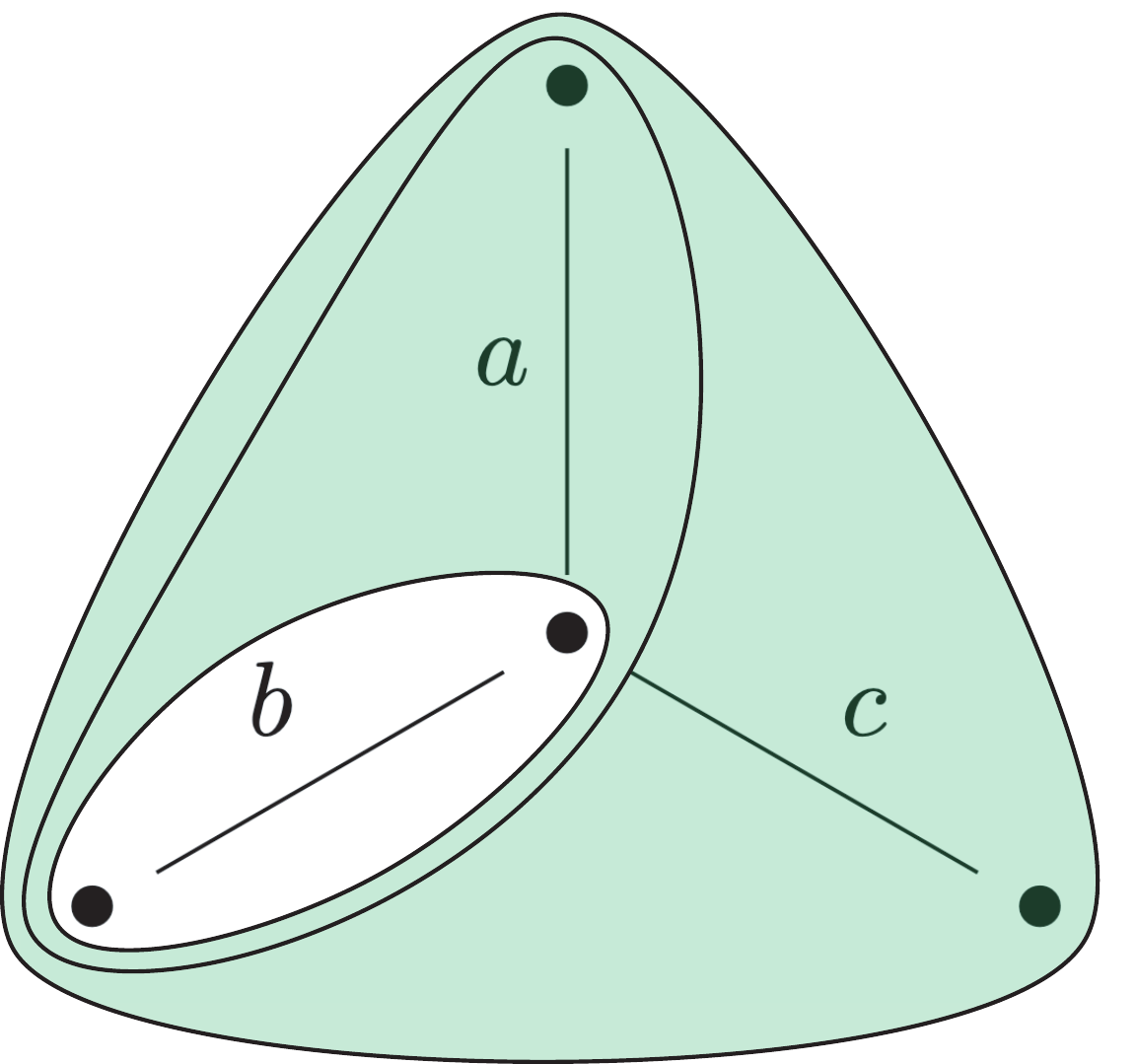}
    \caption{}
    \label{intro:regions3}
  \end{subfigure}
\caption{The generators of the universe fan $\mc{U}_\mc{G}$ correspond to regions cut out by tubes on the graph. (A), (B), (C) show the regions for $ab^+$, $ab^+ + b^-$, and $\hat{1}^+ + b^-$, respectively.}
\label{intro:regions}
\end{figure}

\subsection{Causal paths and properties of $\Psi$}
Despite the straightforward ray description of $\mathcal U_{\mathcal G}$ given above, we begin our study of the universe fan with a remarkably simple dual description. Let us write $p_f, m_f \in (\mathbb{R}^{2\mc{G}})^\vee$ for the basis of dual linear forms that is dual to the basis $f^+, f^-$ of $\mathbb{R}^{2\mc{G}}$. Then for each $g \in \mc{G}$, consider the linear forms
\begin{equation}\label{intro:Fg}
F_g \;:=\; p_{\hat 1}+\sum_{g<f<\hat 1}(p_f-m_f)-m_g.
\end{equation}
We define the \emph{causal cone} $C$ in $\mathbb R^{\dbl\mathcal G}$ by the positivity conditions (Definition \ref{dfn:causal-cone})
\[
m_{\hat 1}=0,\qquad p_g\ge 0,\quad m_g\ge 0,\quad F_g\ge 0\qquad (g\in\mathcal G).
\]

\begin{dfn*}[Definition~\ref{dfn:UG}]
Let $\dblfan{\mathcal G}\subset \mathbb R^{\dbl\mathcal G}$ be the \emph{doubled nested set fan}, with maximal cones to be $\cone(f^+, f^- \,|\, f \in N)$ for maximal nested sets $N$. The \emph{universe fan} $\mc{U}_\mc{G}$ is the intersection
\[
\mathcal U_{\mathcal G} \;:=\; \dblfan{\mathcal G}\ \cap C.
\]
\end{dfn*}

\noindent
This definition makes it clear that $\mc{U}_\mc{G}$ is a fan. It has a maximal cone $U_N$ for each maximal nested set $N$. Moreover, the duals of the maximal cones are
\[
U_N^\vee = \langle p_f, m_f, F_f \, | \, f \in N \backslash \{\hat 1\} \rangle.
\]
In Proposition \ref{prop:cosmological-polytopes}, we show that $U_N^\vee$ is in fact (the cone over) the \emph{cosmological polytope} of the Hasse diagram of $N$, in the sense of \cite{ArkaniHamed17}.

Moreover, we study the residues of $\Psi(L,\mc{G})$ using this definition of $\mc{U}_\mc{G}$. The poles of $\Psi(L,\mc{G})$ are given by $E_R = 0$, where
\[
E_R = E_{r_*}^+ + \sum_{i} E_{r_i}^-,
\]
for each causal region $R = \{r_*, r_1,\ldots, r_k\}$. We prove a factorization property of the residues of $\Psi(L,\mc{G})$ from studying the links of the generators of $\mc{U}_\mc{G}$. An essential step is to understand the behaviour of the functions $F_g$, equation \eqref{intro:Fg}, under the quotient by the generators of the fan.

\begin{thm*}[Theorem \ref{thm:universe-links}, Corollary \ref{cor:wavefunction-res}, Corollary \ref{cor:tot-energy-res}]
The residues of $\Psi(L,\mc{G})$ on its poles are given by
\begin{equation}\label{intro:resER}
\frac{1}{E_R} \Res_{E_R = 0} \Psi(L,\mc{G}) = A^\text{sh}([\hat{0},r_*],\mc{G}_1) \times  \Psi(L,\mc{G}_2),
\end{equation}
for two nestable sets $\mc{G}_1 \subset [\hat{0},r_*]$ and $\mc{G}_2 \subset L$, defined by $R$. Here, $A^\text{sh}([\hat{0},r_*],\mc{G}_1)$ is the amplitude $A([\hat{0},r_*],\mc{G}_1)$ after a redefinition of variables.
As an important special case, $\Psi(L,\mc{G})$ has a pole at $E_{\hat{1}}^+ = 0$, and
\begin{equation}\label{intro:Etot}
\Res_{E_{\hat{1}}^+=0} \Psi(L,\mc{G}) = A(L,\mc{G}).
\end{equation}
\end{thm*}

\noindent
We emphasize that, even if $\mc{G}$ is a building set of $L$ in the sense of \cite{feichtner2004chow}, $\mc{G}_2$ may not necessarily be a building set. This is why it is natural to instead work with nestoids. Given a nestoid $(L,\mc{G})$, the statement of the theorem defines two new nestoids, $([\hat{0},r_*],\mc{G}_1)$ and $(L,\mc{G}_2)$.

In the physical case, \eqref{intro:Etot} describes the limit that recovers scattering amplitudes, and \eqref{intro:resER} gives the factorization properties of the wavefunction, which is related to the principle of unitarity. It is striking that these properties follow from the simple definition of $\Psi$ using the positivity conditions, above. Section \ref{sec:physics} reviews the physical wavefunction from this point of view.

\subsection{Refinements and polytopes}
In Section \ref{sec:park}, we set out how vectors in the doubled nested set fan $\Sigma^{2\mc{G}} \subset \mathbb{R}^{2\mc{G}}$ can be naturally identified with \emph{cuts} or \emph{subgraphs}. In the graphical case, these are the tubing regions considered in the example above (e.g. Figure \ref{intro:regions}). For a given nested set, $N$, we identify $f^+$ and $f^-$ with two disjoint subsets of $N$. Then, for each $f \in N$, we can consider the linear forms
\[
\lambda_f = \sum_{\substack{g \in N \\ f \in g^+}} p_g + \sum_{\substack{g \in N \\ f \in g^-}} m_g.
\]
In Section \ref{sec:lightcone}, we show how the domains of linearity of the piecewise linear functions (for each nested set $N$)
\[
\min( \lambda_f \, |\, f \in N)
\]
define refinements of $\Sigma^{2\mc{G}}$ that we call \emph{lightcone refinements}. We will see in \Cref{prop:universe-as-subfan} that the minimal lightcone refinement of $\Sigma^{2\mc{G}}$ actually realizes the universe fan as a subfan, while \Cref{prop:max-lightcone-refine} shows that the `maximal' lightcone refinement induces a triangulation of $\mathcal U_{\mathcal G}$, with one simplex for every tubing of the Hasse diagram of a nested set $N$.

In the case that $L$ is a lattice of flats of a matroid with ground set $E$, then we can consider the projection
\[
\overline{p}: \mathbb{R}^\mc{G} \longrightarrow \mathbb{R}^E / \langle p(\hat{1}) \rangle ,\qquad f \mapsto \sum_{e\in f} e.
\]
The image of the free nested fan $\Sigma^\mc{G}$ under $\overline{p}$ is itself a fan, $\overline{\Sigma}_\mc{G}$. 

\begin{thm*}[\Cref{lem:induced-refinement}, \Cref{prop:coolbijection}, \Cref{cor:local-nest-refinement}]
Under these assumptions:
\begin{enumerate}
\item A lightcone refinement $\Gamma$ of $\mc{U}_\mc{G}$ induces a refinement $\overline{\Gamma}$ of $\overline{\Sigma}_\mc{G}$.
\item Maximal cones of $\Gamma$ are in bijection with the maximal cones of $\overline{\Gamma}$. 
\item The cones of $\overline{\Gamma}$ are the domains of linearity of tropical linear functions.
\end{enumerate}
\end{thm*}

In special cases, $\overline{\Sigma}_\mc{G}$ is the normal fan of a simple polytope, in particular we obtain the normal fans of generalized permutohedra in this way  \cite{Pos09,Dev06}. In the case relevant to physics, $\overline{\Sigma}_\mc{G}$ is the normal fan of the associahedron, and the induced lightcone refinements $\overline{\Gamma}$ of $\overline{\Sigma}_\mc{G}$ are the normal fans of combinatorial blow-ups of the associahedron. The \emph{cosmohedron} of \cite{ArkFigVaz24} is given by taking $\Gamma$ to be the maximal lightcone subdivision. Indeed, by the theorem, the vertices of this polytope (maximal cones of $\overline{\Gamma}$) are 1:1 with maximal tubing or \emph{Russian doll} pictures. More generally, we find a family of  ``cosmohedron-like'' polytopes,  including the polytopes studied in \cite{ForGleKim25}.

\subsection*{Acknowledgements}
We thank Tomasz Lukowski, Nima Arkani-Hamed, Carolina Figueiredo, Mieke Fink, and Martina Juhnke-Kubitzke for helpful discussions. HF is supported by the Sivian Fund. FL is supported by the Deutsche Forschungsgemeinschaft (DFG, German Research Foundation, CRC/TRR 388). Additional support from the European Union (ERC, UNIVERSE PLUS, 101118787).

%\include{oldintro.tex}

%%%%%%%%%%%%%%%%%%%%%%
%%%%%%%%%%%%%%%%%%%%%%
%%%%%%%%%%%%%%%%%%%%%%
%%%%%%%%%%%%%%%%%%%%%%
\section{Lattices and Amplitudes}\label{sec:amplitudes}

In this section, we associate rational functions to building sets and more generally \textit{nestable sets} in a lattice $L$, which we call \textit{nestoid amplitudes}. The construction is inspired by \cite{Lam24} and relies on a `trivial' realization of the \textit{nested set complex} as a polyhedral fan. We will define the nestoid amplitude as a Laplace transform of this fan. In the case where $L$ is the Las Vergnas face lattice of a tope, we recover the matroid amplitudes from \cite{Lam24} after a natural substitution of variables (\Cref{prop:nestoids and matroids}).\par

%%%%%%%%%%%%%%%%%%%%%%
\subsection{Building sets and Nestoids}

Fix some lattice, $L$, and write $\hat 1$ and $\hat 0$ for the unique maximal and minimal element, respectively. The following classical notion is the starting point of our construction.

\begin{dfn}[Building sets]\label{dfn:building sets}
A subset $\mathcal G$ of $L_{>\hat 0}$ is called a \textit{building set} if for every $f \notin \mathcal G$, the interval $[\hat 0, f]$ in $L$ is the product lattice $[\hat 0, g_1] \times \ldots \times [\hat 0, g_k]$, where $g_1,\ldots,g_k$ are the maximal elements in $\mathcal G$ below $f$.
\end{dfn}

The elements $g_1,\ldots,g_k$ in the definition are also called the \textit{factors} of $f$ in $\mathcal G$. %The elements $g_1,\ldots,g_k$ are also called the \textit{factors} of $f$ in $\mathcal G$.

The notion of building sets originated in the study of subspace arrangements \cite{conciniprocesi95}. It was later generalized to arbitrary lattices \cite{feichtner2004chow} and appears in the study of matroids and Bergman fans \cite{FeiStu04} and generalized permutahedra \cite{Pos09}.

An element $f$ of $L$ is called \textit{irreducible} if $[\hat 0, f]$ is not a product of two non-trivial subposets. Write $I(L)$ for the set of irreducible elements. Using this notion, building sets can equivalently be described in the following way:

\begin{prop}[{\cite[Proposition 2.11]{backman2024}}]\label{prop:building-char}
    A subset $\mathcal G$ of $L_{>\hat 0}$ is a building set if and only if it contains $\hat 1$, $I(L) \subseteq \mathcal G$ and for $f,g \in \mathcal G$, $f\lor g \notin \mathcal G$ implies that $f \land g = \hat 0$.
\end{prop}

We introduce the following terminology:

\begin{dfn}[Nestable sets]\label{dfn:nestable}
We call a subset $\mathcal G \subseteq L_{>\hat{0}}$ \textit{nestable} if it contains $\hat 1$, and if, for any two incomparable $f,g \in \mathcal G$, $f\lor g \notin \mathcal G$ implies that $f \land g = \hat 0$.
\end{dfn}

By definition, a nestable set is a building set if and only if it contains the irreducible elements in $L$. One might think of a nestable set as an `incomplete' building set.

Now fix a pair $(L,\mc{G})$ of a lattice together with a nestable set $\mc{G} \subseteq L_{> \hat{0}}$.
% We will sometimes refer to such a pair as a \textit{nestoid}.

\begin{dfn}[Nested sets]
A subset $N \subseteq \mathcal G$ is \textit{nested} if, for any set of pairwise incomparable elements $f_1,\ldots,f_k\in N$, we have $f_1 \lor \ldots \lor f_k \notin \mathcal G$. We denote the set of all nested sets in $\mathcal G$ by $\mathcal N(L,\mathcal G)$.
\end{dfn}

\noindent
Any subset of a nested set is nested, so that $\mathcal N(L,\mathcal G)$ forms an abstract simplicial complex, the \textit{nested set complex}  of $(L,\mathcal G)$. If $\mathcal G$ is a building set this complex is pure, see e.g. \cite[Corollary 4.3]{FeiMue05}. This fails for a general nestable set and motivates the following definition:
\begin{dfn}\label{dfn:nestoid}
   A \textit{nestoid} is a pair $(L,\mathcal G)$ of a lattice $L$ and a nestable set $\mathcal G$ such that $\mathcal N([\hat 0,f],\mathcal G \cap [\hat 0, f])$ is pure for all $f \in \mathcal G$.
\end{dfn}

\begin{prop}
    Any building set defines a nestoid.
\end{prop}
\begin{proof}
    The nested set complex of a building set $\mathcal G$ is pure due to \cite[Corollary 4.3]{FeiMue05}. The statement follows because $\mathcal G \cap [\hat 0, f]$ is a building set in $[\hat 0, f]$, which can be seen directly from the definition.
\end{proof}

\begin{rmk}
    Our main motivation for the introduction of nestoids is that, as we will see, links of nested set complexes of nestoids factor into nested set complexes of nestoids. This will imply that the nestoid amplitudes defined below are `stable' under taking residues, closely related to `unitarity' in physics.
\end{rmk} 

\begin{ex}\label{ex:nestoid-but-no-build}
    Consider the following lattice $L$:
    \[\begin{tikzcd}
	&&& {\color{red} 1234} \\
	{\color{red} 123} && 234 && 134 && {\color{red} 124} \\
	{\color{red}12} && 23 && 34 && 14 \\
	1 && 2 && 3 && {\color{red} 4} \\
	&&& {\hat 0}
	\arrow[no head, from=1-4, to=2-1]
	\arrow[no head, from=1-4, to=2-3]
	\arrow[no head, from=1-4, to=2-5]
	\arrow[no head, from=1-4, to=2-7]
	\arrow[no head, from=2-1, to=3-1]
	\arrow[no head, from=2-1, to=3-3]
	\arrow[no head, from=2-3, to=3-3]
	\arrow[no head, from=2-3, to=3-5]
	\arrow[no head, from=2-5, to=3-5]
	\arrow[no head, from=2-5, to=3-7]
	\arrow[no head, from=2-7, to=3-1]
	\arrow[no head, from=2-7, to=3-7]
	\arrow[no head, from=3-1, to=4-1]
	\arrow[no head, from=3-1, to=4-3]
	\arrow[no head, from=3-3, to=4-3]
	\arrow[no head, from=3-3, to=4-5]
	\arrow[no head, from=3-5, to=4-5]
	\arrow[no head, from=3-5, to=4-7]
	\arrow[no head, from=3-7, to=4-1]
	\arrow[no head, from=3-7, to=4-7]
	\arrow[no head, from=4-1, to=5-4]
	\arrow[no head, from=4-3, to=5-4]
	\arrow[no head, from=4-5, to=5-4]
	\arrow[no head, from=4-7, to=5-4]
    \end{tikzcd}\]
    Let $\mathcal G = \{4,12,123,124,1234\}$. One checks that $(L,\mathcal G)$ forms a nestoid (cf.\ also \Cref{lem:nest-factors} below). Note that $\mathcal G$ is not a building set in $L$ as it does not contain all irreducible elements.
\end{ex}

Let $(L,\mathcal G)$ be a nestoid. Its nested set complex can trivially be realized as a unimodular simplicial polyhedral fan in $\mathbb R^{\mathcal G}$:

\begin{dfn}[Free nested set fan]\label{dfn:free nested fans}
    Let $(L, \mathcal G)$ be a nestoid. For $N \in \mathcal N(L,\mathcal G)$, consider the cone
    $$\cone N \subseteq \mathbb R^{\mathcal G}.$$
    By definition, these cones assemble into a fan $\Sigma^{\mathcal G}$ in $\mathbb R^{\mathcal G}$ which we call the \textit{free nested set fan}.
\end{dfn}
\begin{rmk}
    Note that the fan $\Sigma^{\mathcal G}$ depends on the lattice $L$, even though we omit it from the notation for simplicity. 
\end{rmk}

By definition, the face lattice of $\Sigma^{\mathcal G}$ is the inclusion lattice of $\mathcal N(L,\mathcal G)$.

Let us remark on how the free nested set fan relates to the \textit{nested set fan} $\Sigma_{\mathcal G}$ from \cite[Section 5]{feichtner2004chow}. Indeed, in the case that $L$ is an atomic lattice, $\Sigma^{\mathcal G}$ is a lift of $\Sigma_{\mathcal G}$ in a higher-dimensional space. Recall that the atoms of $L$ are the minimal elements of $L_{>\hat 0}$ and that $L$ is called atomic if every $f \in L$ is the join of the atoms $e$ with $e\leq f$. We will sometimes identify $f\in L$ with the set of atoms $e \leq f$ in the following. Note that $f \cup g \subseteq f \lor g$ as sets but equality does not hold in general.

\begin{dfn}\label{dfn:nested-set-fan}
    Assume $L$ is an atomic lattice with set of atoms $E$. Then the nested set fan $\Sigma_{\mathcal G}\subseteq \mathbb R^E$ associated to $\mathcal G$ is the image of the free nested set fan $\Sigma^{\mathcal G} \subseteq \mathbb R^{\mathcal G}$ under the projection
\[
        p: \mathbb R^{\mathcal G} \to \mathbb R^E,\qquad f \mapsto \sum_{e\in f} e.
\]
\end{dfn}

\noindent
The image of $\Sigma^\mc{G}$ under this map is indeed again a fan. This follows from the following stronger result:

\begin{prop}\label{prop:nested-projection}\leavevmode
    The restriction of $p$ to $\Sigma^{\mathcal G}$ is injective.
\end{prop}
\begin{proof}
    We construct a retract (of sets) $r: p(\Sigma^{\mathcal G}) \to \Sigma^{\mathcal G}$. We do this by induction on the cardinality of the support of $p(x)$. 
    
    Write $x := \sum_{f \in N} \mu_f f$ for some nested set $N$. Wlog.\ we assume that $\mu_f > 0$ for all $f \in N$. The support $F$ of $p(x)$ is the set of $e\in E$ that are contained in a maximal flat in $N$. We claim that the set $N_{\mathrm{max}}$ of maximal elements of $N$ is precisely the set $\mathcal F$ of maximal elements in $\mathcal G$ that are contained in $F$ (as sets).

    Indeed, take $f \in \mathcal F$. For every $e \in f$ there is some $g \in N_{\mathrm{max}}$ with $e \in g$. Let $g_1,\ldots,g_k$ be the set of such $g$. In particular, $f \leq g_1 \lor \ldots \lor g_k$. We have $f \lor g_i \in \mathcal G$ for all $i$, since $f \land g_i \not= \hat 0$, which furthermore implies $f \lor g_1 \lor \ldots \lor g_k \in \mathcal G$. But $f \lor g_1 \lor \ldots \lor g_k = g_1 \lor \ldots \lor g_k$ which lies in $\mathcal G$ if and only if $k=1$ since $N$ is nested. This implies $\mathcal F = N_{\max}$.

    Thus we can define the retract as follows: if the support of $p(x)$ is empty, then we must have $x=0$ and so we map $p(x) \mapsto 0$. Otherwise, let $F \subseteq E$ be the support of $p(x)$. Let $\mathcal F$ be the set of maximal elements in $\mathcal G$ contained in $F$ and let $\lambda := \min_{e\in F} e^*(p(x))$. Set $r(p(x)) := \lambda \sum_{f\in \mathcal F} f + r(p(x) - \lambda \sum_{f\in \mathcal F} \sum_{e\in f} e)$ which is well-defined by induction (note that for $f\not=g \in \mathcal F$ we have $f\land g = \hat 0$). We claim that $r(p(x)) = x$. Indeed, note that we have $\lambda = \min_{f\in N_{\mathrm{max}}} \mu_f$. Set $y :=  x - \lambda \sum_{f \in N_{\mathrm{max}}} f$. By induction, we have $r(p(y)) = y$. But note that $p(y) = p(x) - \lambda \sum_{f\in \mathcal N_{\max}} \sum_{e\in f} e$ and so since $\mathcal F = N_{\mathrm{max}}$, we conclude.
\end{proof}

In \cite[Theorem 10.17]{Lam24}, the matroid amplitude is written as the Laplace transform of a certain nested set fan. Similarly, we define the amplitude of a nestoid as a Laplace transform of a free nested set fan:

\begin{dfn}[Nestoid amplitudes]\label{dfn:nestoid-amplitudes}
    Let $(L, \mathcal G)$ be a nestoid. Then the \textit{nestoid amplitude} $A(L, \mathcal G)$ is the Laplace transform of $\Sigma^{\mathcal G}/\langle \hat 1 \rangle$.
\end{dfn}

Unravelling the definitions, we have
$$A(L,\mathcal G)  = \sum_{N \in \mathcal N_{\max}(L,\mathcal G)} \prod_{f \in N} \frac 1 {s_f}$$
where $s_f$ is the coordinate function on $\mathbb R^{\mathcal G}$ associated to $f$.

\begin{ex}\label{ex:star-amplitude}
    Consider the lattice $L$ given by
    \[ \begin{tikzpicture}[scale=0.65,node distance=20mm]
  	\node (o) at (2.5,-1.2) {$\emptyset$};
	\node (a) at (0,0)   {1};
  	\node (b) at (2.5,0) {2};
  	\node (c) at (5,0)   {3};
  	\node (d) at (1.25,1.8)   {12};
  	\node (e) at (3.75,1.8) {23};
  	\node (f) at (6.25,1.8)   {13};
  	\node (t) at (2.5,3.4) {$123$};
	  \draw (d) -- (a);
	  \draw (d) -- (b);
	  \draw (e) -- (b);
	  \draw (e) -- (c);
	  \draw (f) -- (c);
	  \draw[black!40] (f) -- (a);
	  \draw (t) -- (d);
 	  \draw (t) -- (e);
  	  \draw (t) -- (f);
	  \draw (o) -- (a);
	  \draw (o) -- (b);
	  \draw (o) -- (c);
	\end{tikzpicture}
  \]
This is the boolean lattice on $\{1,2,3\}$. Let $\mc{G} = L \backslash \{\hat 0\}$. This defines a nestoid $(L,\mc{G})$. The free nested set fan $\Sigma^\mc{G}$ has six maximal cones:
\[
\langle 1, 12, 123 \rangle,~~~ \langle 2, 12, 123 \rangle,~~~\langle 2, 23, 123 \rangle,~~~ \langle 3, 23, 123 \rangle,~~~\langle 3, 13, 123 \rangle,~~~ \langle 1, 13, 123 \rangle.
\]
The Laplace transform of $\Sigma^\mc{G} / \langle \hat{1} \rangle$ gives the amplitude
\[
A(L,\mc{G}) = \frac{1}{s_1 s_{12} } + \frac{1}{s_2 s_{12}} + \frac{1}{s_2 s_{23} } + \frac{1}{s_3 s_{23}} + \frac{1}{s_3 s_{13} } + \frac{1}{s_1 s_{13}}.
\]
\end{ex}

\begin{ex}\label{ex:bowtie-amplitude}
Let us consider the `bowtie' graph
\[ \begin{tikzpicture}[scale=1.5,node distance=20mm]
  	\node (1) at (0,1) {};
	\node (2) at (0,0)   {};
  	\node (3) at (1,0.5) {};
  	\node (4) at (2,1)   {};
  	\node (5) at (2,0)  {};

    \draw (1.center) -- node[midway, xshift=-8pt, yshift=0]{$c$} (2.center);
    \draw (2.center) -- node[midway, xshift=5pt, yshift=-5pt]{$b$} (3.center);
    \draw (1.center) -- node[midway, xshift=5pt, yshift=5pt]{$a$} (3.center);
    \draw (3.center) -- node[midway, xshift=-5pt, yshift=5pt]{$d$} (4.center);
    \draw (4.center) -- node[midway, xshift=8pt, yshift=0]{$f$} (5.center);
    \draw (3.center) -- node[midway, xshift=-5pt, yshift=-5pt]{$e$} (5.center);
\end{tikzpicture}\]

Let us choose the \textit{lattice of flats} of the associated graphical matroid as our lattice $L$. This is simply the set of disjoint unions of connected induced subgraphs. Let $\mathcal G$ be the set of connected induced subgraphs. This is a building set in $L$. 
\begin{figure}[h]
\[\begin{tikzcd}
	&& {\hat 1} \\
	& abcd & abce & adef & bdef \\
	abc & ad & be & ae & bd & def \\
	a & b & c & d & e & f \\
	&& {\hat 0}
	\arrow[no head, from=1-3, to=2-2]
	\arrow[no head, from=1-3, to=2-3]
	\arrow[no head, from=1-3, to=2-4]
	\arrow[no head, from=1-3, to=2-5]
	\arrow[no head, from=2-2, to=3-1]
	\arrow[no head, from=2-2, to=3-2]
	\arrow[no head, from=2-2, to=3-5]
	\arrow[no head, from=2-3, to=3-1]
	\arrow[no head, from=2-3, to=3-3]
	\arrow[no head, from=2-3, to=3-4]
	\arrow[no head, from=2-4, to=3-2]
	\arrow[no head, from=2-4, to=3-4]
	\arrow[no head, from=2-4, to=3-6]
	\arrow[no head, from=2-5, to=3-3]
	\arrow[no head, from=2-5, to=3-5]
	\arrow[no head, from=2-5, to=3-6]
	\arrow[no head, from=3-2, to=4-4]
	\arrow[no head, from=3-3, to=4-2]
	\arrow[no head, from=3-3, to=4-5]
	\arrow[no head, from=3-4, to=4-1]
	\arrow[no head, from=3-4, to=4-5]
	\arrow[no head, from=3-5, to=4-2]
	\arrow[no head, from=3-5, to=4-4]
	\arrow[no head, from=3-6, to=4-4]
	\arrow[no head, from=3-6, to=4-5]
	\arrow[no head, from=3-6, to=4-6]
	\arrow[no head, from=4-1, to=3-1]
	\arrow[no head, from=4-1, to=3-2]
	\arrow[no head, from=4-1, to=5-3]
	\arrow[no head, from=4-2, to=3-1]
	\arrow[no head, from=4-2, to=5-3]
	\arrow[no head, from=4-3, to=3-1]
	\arrow[no head, from=4-3, to=5-3]
	\arrow[no head, from=4-4, to=5-3]
	\arrow[no head, from=4-5, to=5-3]
	\arrow[no head, from=4-6, to=5-3]
\end{tikzcd}\]
\caption{The poset $\mathcal G$ for the bowtie graph.}
\end{figure}
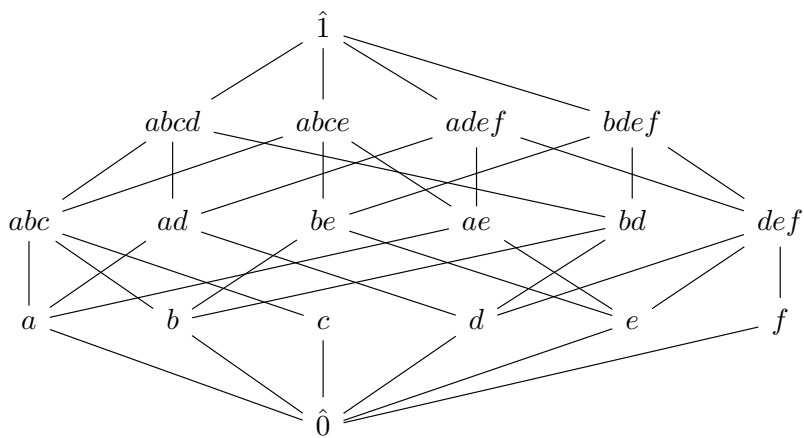
There are $38$ maximal nested sets in $\mathcal N(L,\mathcal G)$ and each maximal nested set has cardinality $4$. Accordingly, the amplitude $A(L,\mathcal G)$ is a sum of $38$ Laurent monomials, each of degree $-3$. The maximal nested sets are given by maximal singleton-free tubings of the graph. $28$ of these are simply chains in the lattice. The $10$ remaining sets are
\[
\begin{array}{ccccc}
\{a,abc,f,\hat 1\} &
\{b,abc,f,\hat 1\} &
\{d,def,c,\hat 1\} &
\{e,def,c,\hat 1\} &
\{c,abc,f,\hat 1\} \\[6pt]
\{c,def,f,\hat 1\} &
\{c,d,abcd,\hat 1\} &
\{c,e,abce,\hat 1\} &
\{f,a,adef,\hat 1\} &
\{f,b,bdef,\hat 1\}
\end{array}
\]
\end{ex}

%%%%%%%%%%%%%%%%%%%%%%
\subsection{Matroid amplitudes}\label{sec:matamp}
Let us now spell out the relation of nestoid amplitudes to the matroid amplitudes defined in \cite{Lam24}. For the convenience of the reader, we give a short overview in \Cref{app:matroid}. Recall that the \textit{positive nested set complex} associated to a tope $P$ of an oriented matroid $\mathcal M$ is the subset
$$\mathcal N(L(M), \mathcal G)_P := \{ N \in \mathcal N(L(M), \mathcal G) \ | \ N \subseteq L(P)\}$$
where $L(P) \subseteq L(M)$ is the Las Vergnas face lattice of $P$.

\begin{lem}\label{lem:tope-building}
   The set $\mathcal G \cap L(P)$ is a building set in $L(P)$ and
   $$\mathcal N(L(M), \mathcal G)_P  = \mathcal N(L(P),\mathcal G\cap L(P)).$$
\end{lem}
\begin{proof}
    If we have $f \lor g \notin \mathcal G \cap L(P)$ for $f,g \in L(P)$, then already $f \lor g \notin \mathcal G$, since $L(P)$ is stable under joins in $L(M)$. Thus, $\mathcal G\cap L(P)$ is a nestable set. Since the irreducible elements in $L(P)$ are still irreducible in $L(M)$, $\mathcal G \cap L(P)$ is a building set by \Cref{prop:building-char}. The second assertion then follows again by stability of $L(P)$ under joins in $L(M)$.
\end{proof}

The matroid amplitude $\mathcal A(P)$ is defined as the Laplace transform of the nested set fan associated to $\mathcal N(L(M), \mathcal G)_P$. See \Cref{app:matroid:laplace}.

\begin{prop}\label{prop:nestoids and matroids}\leavevmode
    The matroid amplitude $\mathcal A(P)$ is obtained from the nestoid amplitude $$A(L(P), \mathcal G \cap L(P))$$ for an arbitrary building set $\mathcal G$ in the lattice $L(M)$ by the substitution $s_f = \sum_{e \in f} a_e$.
\end{prop}
\begin{proof}
    Combine \Cref{prop:nested-projection} and \Cref{lem:tope-building}.
\end{proof}
\begin{rmk}\label{rmk:nestoid building invariance}
    The matroid amplitude does not depend on the choice of building set in $L(M)$: by \cite[Theorem 4.2]{FeiMue05}, different building sets induce different refinements of the nested set fan associated to the minimal building set $\mathcal G^{\min}$, and so the Laplace transform does not change. Meanwhile, the associated nestoid amplitudes clearly differ. However, it is implicit in the proof of \cite[Theorem 4.2]{FeiMue05} that every free nested set fan \textit{projects} to a refinement of the free nested set fan associated to $\mathcal G^{\min}$. This means that after replacing $s_g$ for $g \in \mathcal G$ by $\sum_{g'\in F(g)} s_{g'}$ where $F(g)$ is the set of factors of $g$ in $\mathcal G^{\min}$ (see \Cref{dfn:building sets}), the different nestoid amplitudes are identified.
\end{rmk}

\begin{ex}
    Consider the positive tope $P$ of the graphical oriented matroid
    \[ \begin{tikzpicture}[scale=0.8,
        decoration = {markings,
                        mark=at position .5 with {\arrow{Stealth[length=2mm]}}},
        dot/.style = {circle, fill, inner sep=1.8pt, node contents={},
                        label=#1},
    every edge/.style = {draw, postaction=decorate}
                            ]
            \node (1) at (-1.5,-1) [dot = below:$x$];
            \node (2) at (1.5,-1) [dot = below:$y$];
            \node (3) at (0,-0.2) [dot = below:$w$];
            \node (4) at (0,1) [dot = above:$z$];

            \path (1) edge (3);
            \path (2) edge (3);
            \path (3) edge (4);
            \path (1) edge (4);
            \path (2) edge (4);
    \end{tikzpicture}\]
    Its Las Vergnas face lattice is isomorphic to the lattice in \Cref{ex:star-amplitude}. We can choose $\mathcal G = L(M) \backslash \{\hat 0\}$, so that $(L(P),\mathcal G \cap L(P))$ is the nestoid that we considered there. Accordingly, the nestoid amplitude is
     \begin{align*}
        \frac 1 {s_{xw} s_{xwz}} + \frac 1 {s_{wz} s_{xwz}} + \frac 1 {s_{yw} s_{ywz}} + \frac 1 {s_{wz} s_{ywz}} + \frac 1 {s_{xw} s_{xwy}} + \frac 1 {s_{yw} s_{xwy}}.
    \end{align*}
    According to \Cref{prop:nestoids and matroids} we obtain the matroid amplitude by substitution:
    \begin{align*}
        \frac 1 {a_{xw} (a_{xw} + a_{wz} + a_{xz})} + \frac 1 {a_{zw} (a_{xw} + a_{wz} + a_{xz})} + \frac 1 {a_{yw} (a_{yw} + a_{wz} + a_{yz})} + \\
        \frac 1 {a_{zw} (a_{yw} + a_{wz} + a_{yz})} + \frac 1 {a_{xw} (a_{xw} + a_{yw})} + \frac 1 {a_{yw} (a_{xw} + a_{yw})}
    \end{align*}
    Note that the sum of the last two terms simplifies to $\frac 1 {a_{xw} a_{yw}}$. This corresponds to the fact that the minimal building set on $L(M)$ does not include the flat $\{xw,yw\}$: its factors in $\mathcal G^{\min}$ are $\{xw\}$ and $\{yw\}$. On the level of nestoids, we can verify \Cref{rmk:nestoid building invariance}: replacing $s_{xwy}$ with $s_{xw} + s_{yw}$, $A(L(P),\mathcal G \cap L(P))$ indeed simplifies to
    \begin{align*}
        A(L(P),\mathcal G^{\min} \cap L(P)) = \frac 1 {s_{xw} s_{xwz}} + \frac 1 {s_{wz} s_{xwz}} + \frac 1 {s_{yw} s_{ywz}} + \frac 1 {s_{wz} s_{ywz}} + \frac 1 {s_{xw} s_{yw}}.
    \end{align*}
\end{ex}

%%%%%%%%%%%%%%%%%%%%%%
\subsection{Factorization}
Let $(L, \mathcal G)$ be a nestoid, for example given by a building set in the lattice. We are interested in the links of rays in $\Sigma^{\mathcal G}$ as they correspond to residues of the associated amplitude. We will see that if the nestoid is sufficiently nice, the links factor again as products of free nested set fans.

\begin{dfn}
    We call a nestoid $(L,\mathcal G)$ \textit{stable} if for $f, g_1,\ldots,g_k \in \mathcal G$ with $f\lor g_i \notin \mathcal G$ for all $i$ we have $f\lor g_1 \lor \ldots \lor g_k \notin \mathcal G$.
\end{dfn}

\begin{rmk}
    If $L$ is a boolean lattice and $\mathcal G$ a building set, then $(L,\mathcal G)$ is stable if and only if $\mathcal G$ is graphical, see \cite[Proposition 7.3]{Zel05}. However, in a boolean lattice factorization works without the stableness assumption, see \cite[Proposition 3.2]{Zel05}.
\end{rmk}

We first make a useful observation:

\begin{lem}\label{lem:upstream-lower-set}
    Let $(L,\mathcal G)$ be a nestoid, $x \in L$ and $y \in \mathcal G$ with $x \lor y \notin \mathcal G$. Then $x \lor z \notin \mathcal G$ for all $z \in [\hat 0, y]$.
\end{lem}
\begin{proof}
    Assume $x \lor y \notin \mathcal G$ and let $z \leq y$. We can not have $x \lor z \leq y$, and $y \leq x \lor z$ implies $x \lor z = x \lor y \notin \mathcal G$. Otherwise, $y$ and $x \lor z$ are incomparable and so if $x \lor z \in \mathcal G$, $y \lor (x \lor z) = x \lor y \notin \mathcal G$ implies $y \land (x \lor z) = \hat 0$, since $\mathcal G$ is nestable. This is impossible since $z \leq y$ and $z \leq (x \lor z)$.
\end{proof}

% The lemma allows us to prove a factorization property of nested set complexes:

Using the lemma and the stableness assumption, we obtain a factorization property of nested set complexes:

\begin{lem}\label{lem:nest-factors}
    Let $(L,\mathcal G)$ be a stable nestoid and $f \in \mathcal G$ and consider the two sets
\[
\mathcal G^{\dstream f} := \{ g \in \mathcal G \ | \ g \leq f \} 
\qquad \text{and}\qquad 
\mathcal G^{\ustream f} := \{ g \in \mathcal G \ | \ g > f \text{ or } g \lor f \notin \mathcal G\}.
\]
Then both $([\hat 0,f],\mathcal G^{\dstream f})$ and $(L,\mathcal G^{\ustream f})$ are stable nestoids and
    $$\mathcal N(L,\mathcal G)_{\ni f} \simeq \mathcal N([0,f],\mathcal G^{\dstream f})_{\ni f} \times \mathcal N(L, \mathcal G^{\ustream f})$$
    where $\mathcal N(L,\mathcal G)_{\ni f}$ denotes the subset of nested sets containing $f$.
\end{lem}
\begin{proof}
    Clearly $\mathcal G^{\dstream f}$ is again nestable and stable since $[\hat 0, f]$ is closed under taking joins. If $b_1,\ldots,b_k \in \mathcal G^{\ustream f}$ with $b_1 \lor \ldots \lor b_k \in \mathcal G$ then by stableness also $b_1 \lor \ldots \lor b_k \in \mathcal G^{\ustream f}$. The contraposition implies that $\mathcal G^{\ustream f}$ is nestable and stable. The same reasoning shows that for every nested set $N \in \mathcal N(L,\mathcal G)_{\ni f}$ we have that $N \cap \mathcal G^{\dstream f}$ and $N \cap \mathcal G^{\ustream f}$ are again nested sets with respect to $\mathcal G^{\dstream f}$ and $\mathcal G^{\ustream f}$, and clearly $N = (N \cap \mathcal G^{\dstream f}) \cup (N \cap \mathcal G^{\ustream f})$. Conversely, take $N_1 \in \mathcal N([0,f],\mathcal G^{\dstream f})_{\ni f}$ and $N_2 \in \mathcal N(L, \mathcal G^{\ustream f})$. We claim that $N_1 \cup N_2$ is again nested. Indeed, let $g_1,\ldots,g_k \in N_1$ and $h_1 \lor \ldots \lor h_l \in N_2$ such that $g_1,\ldots,g_k,h_1,\ldots,h_l$ are pairwise incomparable. In particular, $g_1 \lor \ldots \lor g_k \notin \mathcal G$ and $h_1 \lor \ldots \lor h_l \notin \mathcal G$. Assume that $k,l \geq 1$. Let us check that $g_1 \lor \ldots \lor g_k \lor h_1 \lor \ldots \lor h_l \notin \mathcal G$. For this, note that by incomparability we must have $h_i \lor f \notin \mathcal G$ for all $i$. Now we can apply \Cref{lem:upstream-lower-set} with $x = h_1 \lor \ldots \lor h_l$ and $y = f$. By the stableness assumption $x \lor y \notin \mathcal G$ and thus $x \lor g_1 \lor \ldots \lor g_k \notin \mathcal G$.

    It remains to check the purity condition in \Cref{dfn:nestoid}. This is clear for $([\hat 0,f],\mathcal G^{\lceil f\rceil})$. To see this for $\mathcal G^{\ustream f}$, pick some $g \in \mathcal G^{\lfloor f \rfloor}$ and note that $\mathcal G^{\lfloor f\rfloor,\lceil g \rceil} = \mathcal G^{\lceil g \rceil,\lfloor f\rfloor}$. Since both $\mathcal G^{\lceil g \rceil}$ and $\mathcal G^{\lceil f \rceil}$ are nestoids, the factorization we already proved above implies that the nested set complex of $\mathcal G^{\lceil g \rceil,\lfloor f\rfloor}$ is again pure.
\end{proof}

\begin{rmk}
    If $\mathcal G$ is a stable building set, then $\mathcal G^{\dstream f}$ will again be a building set for the sublattice $[\hat 0,f]$ of $L$. However, it is clear that $\mathcal G^{\ustream f}$ will usually not be a building set in $L$ (see e.g.\ \Cref{ex:nestoid-but-no-build}).
\end{rmk}

\begin{cor}[Factorization of amplitudes]\label{cor:factor-amplitudes}
    Let $(L,\mathcal G)$ be a stable nestoid and let $f \in \mathcal G$. Then the link of $f$ in $\Sigma^{\mathcal G}$ is the product $\Sigma^{\mathcal G^{\dstream f}}/\langle f\rangle \times \Sigma^{\mathcal G^{\ustream f}}$. In particular, we have
    $$\Res_{s_f=0} \  A(L,\mathcal G) = A([0,f], \mathcal G^{\dstream f}) \cdot A(L, \mathcal G^{\ustream f}).$$
\end{cor}
\begin{proof}
    The first statement follows from the definitions and \Cref{lem:nest-factors}. The residue of the Laplace transform of $\Sigma^{\mathcal G}/\langle \hat 1\rangle$ at $s_f = 0$ is the Laplace transform of the link $\Sigma^{\mathcal G} / \langle \hat 1, f \rangle$. By the first statement, this is the product fan $\Sigma^{\mathcal G^{\dstream f}}/\langle f\rangle \times \Sigma^{\mathcal G^{\ustream f}} / \langle \hat 1 \rangle$. Note that $f = \hat 1_{[0,f]}$, so that this indeed Laplace transforms to the product $A([0,f], \mathcal G^{\dstream f}) \cdot A(L, \mathcal G^{\ustream f})$.
\end{proof}

\begin{rmk}
    This is an analogue of \cite[Theorem 17.4]{Lam24}: if $\mathcal G$ is a stable building set then the factorization is the one from loc.\ cit.\ under the substitution in \Cref{prop:nestoids and matroids}.
\end{rmk}

\begin{ex}
Recall the bowtie graph from \Cref{ex:bowtie-amplitude}. Let us compute the residue of the associated amplitude $A(L,\mathcal G)$ at $s_{abc}=0$. The posets $\mathcal G^{\lfloor abc \rfloor}$ and $\mathcal G^{\lceil abc \rceil}$ are given by
\vskip-15pt
\[ \begin{tikzpicture}[scale=1,node distance=20mm,baseline=15pt]
  	\node (123) at (0,1) {$\hat 1$};
	\node (1) at (-1,0)   {$abcd$};
  	\node (2) at (0,0) {$abce$};
  	\node (3) at (1,0)   {$f$};

    \draw (123) -- (1);
    \draw (123) -- (2);
    \draw (123) -- (3);
\end{tikzpicture} \quad \text{ and } \quad \begin{tikzpicture}[scale=1,node distance=20mm,baseline=15pt]
  	\node (123) at (0,1) {$abc$};
	\node (1) at (-1,0)   {$a$};
  	\node (2) at (0,0) {$b$};
  	\node (3) at (1,0)   {$c$};

    \draw (123) -- (1);
    \draw (123) -- (2);
    \draw (123) -- (3);
\end{tikzpicture}.\]
The maximal nested sets for $\mathcal G^{\lfloor abc \rfloor}$ and $\mathcal G^{\lceil abc \rceil}$ are thus
$$\{\{abcd, \hat 1\}, \{abce, \hat 1\}, \{f, \hat 1\}\} \quad \text{ and } \quad \{\{a, abc\}, \{b, abc\}, \{c, abc\}\}$$
and so by \Cref{cor:factor-amplitudes} we have
$$\Res_{s_{abc}=0} A(L,\mathcal G) = \left(\frac 1 {s_{abcd}} + \frac 1 {s_{abce}} + \frac 1 {s_f}\right) \cdot \left(\frac 1 {s_a} + \frac 1 {s_b} + \frac 1 {s_c}\right).$$
For the residue at $s_{ae} = 0$ we compute
\vskip-15pt
\[ \mathcal G^{\lfloor ae\rfloor} = \begin{tikzpicture}[scale=1,node distance=20mm,baseline=15pt]
  	\node (123) at (0,1) {$\hat 1$};
	\node (1) at (-1,0)   {$abce$};
  	\node (3) at (1,0)   {$adef$};

    \draw (123) -- (1);
    \draw (123) -- (3);
\end{tikzpicture} \quad \text{ and } \quad \mathcal G^{\lceil ae\rceil} = \begin{tikzpicture}[scale=1,node distance=20mm,baseline=15pt]
  	\node (123) at (0,1) {$ae$};
	\node (1) at (-1,0)   {$a$};
  	\node (3) at (1,0)   {$e$};

    \draw (123) -- (1);
    \draw (123) -- (3);
\end{tikzpicture}\]
and thus
$$\Res_{s_{ae}=0} A(L,\mathcal G) = \left(\frac 1 {s_{abce}} + \frac 1 {s_{adef}}\right) \cdot \left(\frac 1 {s_a} + \frac 1 {s_e}\right).$$

\end{ex}

%%%%%%%%%%%%%%%%%%%%%%
%%%%%%%%%%%%%%%%%%%%%%
%%%%%%%%%%%%%%%%%%%%%%
%%%%%%%%%%%%%%%%%%%%%%
\section{The Universe Fan and Wavefunctions}\label{sec:universe}

In the previous section, we associated amplitudes to building sets (and more generally, nestoids) in a lattice, by taking the Laplace transform of the free nested set fan. Similarly, we will now define the cosmological wavefunction of a nestoid as the Laplace transform of the \textit{universe fan}, which we introduce next.

% The key idea is to construct a `cosmological' analogue of the nested set fan, which we call the \textit{universe fan}.

%%%%%%%%%%%%%%%%%%%%%%
\subsection{The universe fan}
Let $(L, \mathcal G)$ be a nestoid, for example a lattice with a building set. For $f \in \mathcal G$ we introduce symbols $f^+, f^-$. Let $\mathbb R^{\dbl \mathcal G}$ denote the vector space spanned by all $f^+,f^-$. For a nested set $N$ we will also write $\mathbb R^{\dbl N} \subseteq \mathbb R^{\dbl \mathcal G}$ for the subspace spanned by $f^+,f^-$ with $f \in N$. 

\begin{dfn}\label{dfn:doubled-fan}
    The \textit{doubled free nested set fan} is the fan $\dblfan{\mathcal G} \subseteq \mathbb R^{\dbl \mathcal G}$ whose maximal cones are $\mathbb R^{\dbl N}_+$ for $N \in \mathcal N(L,\mathcal G)$.
\end{dfn}

We denote the vectors dual to $f^+, f^-$ by $p_f,m_f$.

\begin{dfn}[The causal cone] \label{dfn:causal-cone}
    Let $g \in \mathcal G$ and define
    $$F_g := p_{\hat 1} + \sum_{g < f < \hat 1} (p_f - m_f) - m_g.$$
    The \textit{causal cone} in $\mathbb R^{\dbl{\mathcal G}}$ is the cone cut out by $m_{\hat 1} = 0$ and the inequalities $p_g, m_g, F_g \geq 0$ for $g \in \mathcal G$.
\end{dfn}

This is already all that is needed to define the universe fan and the wavefunction:

\begin{dfn}[The universe fan]\label{dfn:UG}
    The universe fan $\mathcal U_{\mathcal G}$ is obtained by intersecting the cones of the doubled nested set fan $\dblfan{\mathcal G}$ with the causal cone. The \textit{wavefunction} $\Psi(L,\mathcal G)$ is the Laplace transform of $\mathcal U_{\mathcal G}$.
\end{dfn}
\begin{rmk}
    We will see in \Cref{prop:max-lightcone-refine} that the maximal cones of $\mathcal U_{\mathcal G}$ span unimodular subspaces, so that the Laplace transform is indeed well-defined.
\end{rmk}

\noindent
In particular, for every nested set $N \in \mathcal N(L,\mathcal G)$ there is a cone $U_N$ in the universe fan that is the intersection of the causal cone with the subspace $\mathbb R^{\dbl N} \subseteq \mathbb R^{\dbl \mathcal G}$ (i.e. the coordinate subspace $p_f=m_f=0$ for $f\notin N$). Since $N$ is nested, for any $g\in N$ the elements of $N$ strictly above $g$ are totally ordered; in particular there is a unique maximal element $h\in N$ with $g\le h$. So, restricted to $\RR^{\dbl N}$, $F_g$ becomes a sum over the unique chain/path from $g$ to $h$:
\[
F_g \;=\; p_{h} + \sum_{g<f<h(g)} (p_f-m_f) - m_g.
\]
Then $U_N$ is cut out by $p_f, m_f, F_f \geq 0$ (for $f\in N$). 

The maximal cones of the universe fan are the $U_N$ for maximal nested sets $N$. They have dimension $2|N|-1$ and intersect in codimension 2.

We can dually describe $\mathcal U_{\mathcal G}$ in terms of its generators.

\begin{dfn}[Causal regions]
    A \textit{causal region} in $\mathcal G$ is a nested set $R \in \mathcal N(L,\mathcal G)$ with one maximal element and all other elements incomparable. Write $\mathcal R(L,\mathcal G) \subseteq \mathcal N(L,\mathcal G)$ for the set of causal regions.
\end{dfn}
% (In the physical case of $\mc{G}_n$, these are sub-polygons of the $n$-gon.)
To each causal region $R = \{ r_*, r_1, \ldots, r_k \}$ in $\mathcal R(L,\mathcal G)$, with maximal element $r_*$, associate a vector
\[
w_R := r_*^+ + \sum_{i} r_i^-
\]
in $\mathbb R^{\dbl{\mathcal G}}$. These vectors are the generators of $\mc{U}_\mc{G}$:

\begin{thm}\label{prop:universe-rays}
    The rays of $\mc{U}_\mc{G}$ are the vectors $w_R$ for $R \in \mathcal R(L,\mathcal G)$. The maximal cones are given by
\[
U_N = \cone(w_R \ | \ R \subseteq N \textit{ causal region})
\]
for maximal nested sets $N \in \mathcal N(L,\mathcal G)$.
\end{thm}
\begin{proof}
   We will see this in \Cref{cor:rays-of-universe}.
\end{proof}

Let us compute first examples of universe fans and wavefunctions.

\begin{ex}\label{ex:star-wavefunction}
    Recall Example \ref{ex:star-amplitude}: $L$ was the lattice of flats of the star graph with $3$ edges, and $\mathcal G$ the maximal building set.
  
  There are $6$ maximal nested sets, given by the six maximal chains in $\mathcal G$. Let us determine the cone $U_N$ for a general maximal chain $N = (a \subset b \subset \hat 1)$. There are $6$ causal regions in $N$. The associated vectors $w_R$ span $U_N$:
    $$U_N = \langle \hat 1^+, \hat 1^+ + b^-, \hat 1^+ + a^-, b^+, b^+ + a^-, a^+ \rangle$$
    Using a suitable triangulation (see \Cref{ex:star-max-lightsub} below) we can compute the wavefunction $\Psi(L,\mathcal G)$:
    $$\sum_{a\subset b \subset \hat 1} \frac 1 {E_{\hat 1}^+  (E_{\hat 1}^+ + E_b^-)  (E_{b}^+ + E_a^-)  E_a^+} \cdot  \left(\frac 1 {E_b^+} + \frac 1 {E_{\hat 1}^+ + E_{a}^-}\right)$$
\end{ex}
\begin{ex}\label{ex:bowtie-wavefunction}
    Now consider the bowtie from Example \ref{ex:bowtie-amplitude}. Recall that there were $38$ maximal nested sets for the graphical building set, and so the universe fan has $38$ maximal cones. There are three different types of nested sets in terms of the poset structure: $28$ chains, e.g.\ $\{a,abc,abcd,\hat 1\}$, $6$ of the same type as $\{a,abc,f,\hat 1\}$ and $4$ of the same type as $\{c,d,abcd,\hat 1\}$. See \Cref{fig:bowtie-hasses}.
    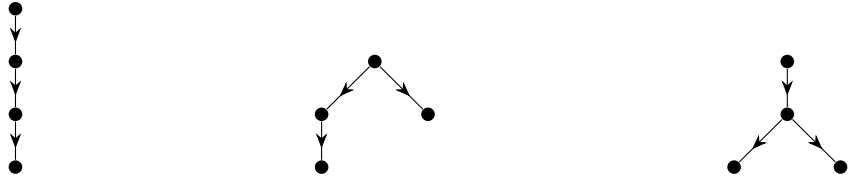
\begin{figure}[h]
       \begin{tikzpicture}[
        scale = 0.7,
       decoration = {markings,
                     mark=at position .7 with {\arrow{Stealth[length=2mm]}}},
       dot/.style = {circle, fill, inner sep=1.8pt, node contents={},
                     label=#1},
every edge/.style = {draw, postaction=decorate}
                        ]
        \node (1) at (0,4) [dot];
        \node (2) at (0,3) [dot];
        \node (3) at (0,2) [dot];
        \node (4) at (0,1) [dot];

        \path (1) edge (2);
        \path (2) edge (3);
        \path (3) edge (4);
\end{tikzpicture} \hspace{100pt} \begin{tikzpicture}[scale=0.7,
       decoration = {markings,
                     mark=at position .7 with {\arrow{Stealth[length=2mm]}}},
       dot/.style = {circle, fill, inner sep=1.8pt, node contents={},
                     label=#1},
every edge/.style = {draw, postaction=decorate}
                        ]
        \node (1) at (0,4) [dot];
        \node (2) at (-1,3) [dot];
        \node (3) at (1,3) [dot];
        \node (4) at (-1,2) [dot];

        \path (1) edge (2);
        \path (2) edge (4);
        \path (1) edge (3);
\end{tikzpicture} \hspace{100pt} \begin{tikzpicture}[scale=0.7,
       decoration = {markings,
                     mark=at position .7 with {\arrow{Stealth[length=2mm]}}},
       dot/.style = {circle, fill, inner sep=1.8pt, node contents={},
                     label=#1},
every edge/.style = {draw, postaction=decorate}
                        ]
        \node (1) at (0,4) [dot];
        \node (2) at (0,3) [dot];
        \node (3) at (-1,2) [dot];
        \node (4) at (1,2) [dot];

        \path (1) edge (2);
        \path (2) edge (4);
        \path (2) edge (3);
\end{tikzpicture}
        \caption{The three different types of Hasse diagrams for nested sets associated to the bowtie graph}\label{fig:bowtie-hasses}
    \end{figure}
    The structure of the poset determines the number of causal regions in the nested set. While the cones for the first two types have $10$ generators each, the cones for the third type are spanned by $11$ rays. For the three nested sets $N_1 = \{a,abc,abcd,\hat 1\}$, $N_2 = \{a,abc,f,\hat 1\}$ and $N_3 = \{c,d,abcd,\hat 1\}$ we obtain the cones
    \begin{small}
    \begin{align*}
        U_{N_1} &= \langle \hat 1^+, \hat 1^+ + abcd^-, \hat 1^+ + abc^-, \hat 1^+ + a^-, abcd^+, abcd^+ + abc^-, abcd^+ + a^-, abc^+, abc^+ + a^-, a^+\rangle \\
        U_{N_2} &= \langle \hat 1^+, \hat 1^+ + f^-, \hat 1^+ + abc^-, \hat 1^+ + abc^- + f^-, \hat 1^+ + a^-, \hat 1^+ + a^- + f^-, abc^+, abc^+ + a^-, a^+, f^+\rangle\\
        U_{N_3} &= \langle \hat 1^+, \hat 1^+ + abcd^-, \hat 1^+ + c^-, \hat 1^+ + d^-, \hat 1^+ + c^- + d^-, abcd^+,  abcd^+ + c^-, abcd^+ + d^-,\\
        & \hskip30em abcd^+ + c^- + d^-, c^+, d^+ \rangle
    \end{align*}
    \end{small}
\end{ex}

%%%%%%%%%%%%%%%%%%%%%%
\subsection{Links and factorization}

So far, it is unclear why our construction should have anything to do with the wavefunction of the universe. As a first indication, let us prove that the nestoid wavefunction simplifies to the amplitude on the `total energy residue'. In fact, in our construction this is a simple consequence of the definitions:
\begin{lem}
    We have $\mathcal U_{\mathcal G}/\langle{\hat 1^+}\rangle \simeq \dblfan{\mathcal G}/\langle\hat 1^+,\hat 1^-\rangle$.
\end{lem}
\begin{proof}
    This can be seen directly from the definition since the inequalities cutting out the positive cone are trivialized in the quotient by $\hat 1^+$.
\end{proof}

It is clear from the definition that the Laplace transform $A_E(L,\mathcal G)$ of $\dblfan{\mathcal G}/\langle\hat 1^+,\hat 1^-\rangle$ agrees with the nestoid amplitude $A(L,\mathcal G)$ under the substitution $s_f = E_f^+ E_f^-$.

\begin{cor}[The total energy residue]\label{cor:tot-energy-res}
    We have
    $$\Res_{E_{\hat 1}^+ = 0} \Psi(L,\mathcal G) = A_E(L,\mathcal G)$$
\end{cor}

In light of \Cref{prop:nestoids and matroids}, this shows that every matroid amplitude appears as the total energy residue of some nestoid wavefunction.

\begin{ex}
Continuing with \Cref{ex:star-wavefunction}, we see that the residue of $\Psi(L,\mathcal G)$ at $E_{\hat 1^+} = 0$ is given by
$$\sum_{a\subset b \subset \hat 1} \frac 1 {E_b^-  (E_{b}^+ + E_a^-)  E_a^+} \cdot  \left(\frac 1 {E_b^+} + \frac 1 {E_{a}^-}\right) = \sum_{a\subset b \subset \hat 1} \frac 1 {{E_b^+} E_b^- E_a^+ {E_{a}^-}}$$
which is indeed the amplitude $A(L,\mathcal G)$ from \Cref{ex:star-amplitude} under the substitution $s_f = E_f^+ E_f^-$.
\end{ex}

If $(L,\mathcal G)$ is a stable nestoid, \Cref{cor:tot-energy-res} is an instance of a general factorization property. To state and prove this, we need the following lemma.

\begin{lem}\label{lem:adv-nest-factors}
    Assume $(L,\mathcal G)$ is stable and let $R \in \mathcal R(L,\mathcal G)$ be a causal region with maximal element $f$ and incomparable elements $g_1,\ldots, g_k$. Set $$\mathcal G^{\dstream R} := \mathcal G^{\dstream f, \ustream{g_1}, \ldots, \ustream{g_k}} \quad \text{and} \quad \mathcal G^{\ustream R} := \mathcal G^{\ustream f} \cup \mathcal G^{\dstream{g_1}} \cup \ldots \cup \mathcal G^{\dstream{g_k}}.$$
    Then both $([0,f],\mathcal G^{\dstream R})$ and $(L,\mathcal G^{\ustream R})$ are stable nestoids and we have
    $$\mathcal N(L,\mathcal G)_{\supseteq R} \cong \mathcal N([0,f], \mathcal G^{\dstream R})_{\ni f} \times \mathcal N(L, \mathcal G^{\ustream R})_{\ni g_1,\ldots,g_k}.$$
\end{lem}
\begin{proof}
    By \Cref{lem:nest-factors}, $([0,f],\mathcal G^{\dstream R})$ is again a stable nestoid and so is $\mathcal G^{\ustream f}$ and every $\mathcal G^{\dstream{g_i}}$. It is then easy to show that also $\mathcal G^{\ustream R}$ is a stable nestoid, using that $R$ is a causal region. Finally, by iteratively applying \Cref{lem:nest-factors}, the factorization statement reduces to the claim
    $$\mathcal N(L, \mathcal G^{\ustream R})_{\ni g_1,\ldots,g_k} \cong \mathcal N(L, \mathcal G^{\ustream f}) \times \mathcal N(L, \mathcal G^{\dstream{g_1}})_{\ni g_1} \times \ldots \times \mathcal N(L, \mathcal G^{\dstream{g_k}})_{\ni g_k}$$
    which is easily verified using \Cref{lem:upstream-lower-set}.
\end{proof}

\begin{thm}[Links of the universe fan]\label{thm:universe-links}
    Assume $(L,\mathcal G)$ is stable and let $R \in \mathcal R(L,\mathcal G)$ be a causal region with maximal element $f$ and incomparable elements $g_1,\ldots,g_k$. Set $\mathcal H := \mathcal G^{\ustream R} \cup \{f\}$. Then we have
    $$\mathcal U_\mathcal G / \langle w_R\rangle \cong \dblfan{\mathcal G^{\dstream R}}/\langle f^+,f^- \rangle \times \mathcal U_{\mathcal H}/\langle w_R\rangle.$$
\end{thm}
\begin{proof}
    Consider the map
    \begin{align*}
        \sigma: \mathbb R^{\dbl{\mathcal G^{\dstream R}}} &\to \mathbb R^{\dbl{\mathcal G}} / \langle w_R\rangle\\
        h^+ &\mapsto h^+ + \sum_{g_i \leq h} g_i^-\\
        h^- &\mapsto h^- + f^+ + \sum_{g_i \lor h \notin \mathcal G} g_i^-
    \end{align*}
    Clearly, its kernel is $\langle f^-,f^+\rangle$. It follows that $\sigma$ induces a linear isomorphism 
    $$\sigma \times \mathrm{id}: \mathbb R^{\dbl{\mathcal G^{\dstream R}}}/\langle f^-,f^+\rangle \times \mathbb R^{\dbl{\mathcal H}}/\langle w_R\rangle \to \mathbb R^{\dbl{\mathcal G}}/\langle w_R\rangle.$$
    By definition, every cone in the link of $u_R$ is contained in some $\mathbb R^{\dbl N}$ for $N\in \mathcal N(L,\mathcal G)_{\supseteq R}$.
    The inequalities for each such cone are then those inequalities $p_g,m_g,F_g\geq 0$ for $g \in \mathcal G$ that are tight at $u_R$. These are the $p_g\geq 0$ for $g \not= f$, $m_g\geq 0$ for $g \notin \{g_1,\ldots,g_k\}$ and $F_g \geq 0$ for $g \not\leq f$ or $g\leq g_i$ for some $i$. The dual isomorphism $(\sigma \times \mathrm{id})^*$ leaves $p_g$ invariant for $g \not= f$ and $m_g$ invariant for $g \notin \{g_1,\ldots,g_k\}$. For $g \not\leq f$, it also sends $F_g$ to itself. However for $g \in \mathcal G^{\dstream R}\backslash \{f\}$, $F_g$ becomes
    $$p_{\hat 1} + \sum_{\hat 1 > h > f} (p_h - m_h) + p_f + \sum_{h \in \mathcal G^{\dstream R}} m_h - m_f + \sum_{f > h > g} (p_h - m_h) - m_g$$
    and so $F_g \geq 0$ is already implied by $F_f \geq 0$. Finally, for $g \leq g_i$, $F_g$ is mapped to
    $$p_{\hat 1} + \sum_{\hat 1 > h > f} (p_h - m_h) + p_f - m_f + \sum_{g_i > h > g} (p_h - m_h) - m_g.$$
    \Cref{lem:adv-nest-factors} implies that the nested sets containing $R$ are exactly the unions $N_1 \sqcup N_2$ such that $N_1 \in \mathcal N([0,f],\mathcal G^{\dstream R})$ does not contain $f$ and $N_2 \in \mathcal N(L,\mathcal H)$ contains $f,g_1,\ldots,g_k$. Moreover, we saw that the constraints on $\mathbb R^{\dbl {\mathcal G}}$ become independent constraints on $\mathbb R^{\dbl{\mathcal G^{\dstream R}}}/\langle f^-,f^+\rangle$ and $\mathbb R^{\dbl{\mathcal H}}/\langle u_R\rangle$ under the isomorphism $\sigma \times \mathrm{id}$. They are precisely the constraints for the cones in the doubled nested set fan $\dblfan{\mathcal G^{\dstream R}}/\langle f^+,f^- \rangle$ and the universe fan $\mathcal U_{\mathcal H}/\langle u_R\rangle$.
\end{proof}

\begin{cor}\label{cor:wavefunction-res}
    In the situation and notation of \Cref{thm:universe-links}, set $E_R:= E_f^+ + E_{g_1}^- + \ldots + E_{g_k}^-$. Then we have
\[
\frac{1}{E_R} \Res_{E_R=0} \Psi(L,\mathcal G) = A_E^{sh}([\hat 0,f],\mathcal G^{\dstream R}) \cdot {\Psi(L,\mathcal H)}
\]
    where $A_E^{sh}(\mathcal G^{\dstream R})$ is obtained from $A_E(\mathcal G^{\dstream R})$ by substituting
    \begin{align*}
        E_h^+ \mapsto E_h^+ + \sum_{g_i \leq h} E_{g_i}^-, \qquad
        E_h^- \mapsto E_h^- - \sum_{g_i \leq h} E_{g_i}^-.
    \end{align*}
\end{cor}
\begin{proof}
    This follows by applying the Laplace transform to \Cref{thm:universe-links}, noting that in $\mathcal U_{\mathcal H}$ every maximal cone contains $w_R$ and that in the quotient $h^- + f^+ + \sum_{g_i \lor h \notin \mathcal G} g_i^- \equiv h^- - \sum_{g_i \leq h} g_i^-$.
\end{proof}

\begin{ex}
    Continuing with the bowtie in \Cref{ex:bowtie-amplitude} and \Cref{ex:bowtie-wavefunction}, let us compute some residues of $\Psi(L, \mathcal G)$. First, consider the residue at $E_R = 0$ for $R = \{\hat 1, abc\}$, i.e., $E_R = E_{\hat 1}^+ + E_{abc}^-$. According to \Cref{cor:wavefunction-res}, we first need to determine $\mathcal H = \mathcal G^{\lfloor R\rfloor} \cup \{\hat 1\}$ and $\mathcal G^{\lceil R \rceil}$. We have
    \vskip-15pt
\[ \mathcal H = \begin{tikzpicture}[scale=1,node distance=20mm,baseline=20pt]
  	\node (1hat) at (0,2) {$\hat 1$};
    \node (abc) at (0,1) {$abc$};
	\node (a) at (-1,0) {$a$};
  	\node (b) at (0,0) {$b$};
    \node (c) at (1,0)   {$c$};

    \draw (1hat) -- (abc);
    \draw (abc) -- (a);
    \draw (abc) -- (b);
    \draw (abc) -- (c);
\end{tikzpicture} \qquad \text{ and } \qquad \mathcal G^{\lceil R\rceil} = \begin{tikzpicture}[scale=1,node distance=20mm,baseline=15pt]
  	\node (1hat) at (0,1) {$\hat 1$};
	\node (abcd) at (-1,0)   {$abcd$};
  	\node (abce) at (0,0)   {$abce$};
    \node (f) at (1,0)   {$f$};

    \draw (1hat) -- (abcd);
    \draw (1hat) -- (abce);
    \draw (1hat) -- (f);
\end{tikzpicture}\]
    The amplitude $A_E(L,\mathcal G^{\lceil R \rceil})$ is the function 
    $\frac{1}{E_{abcd}^+ E_{abcd}^- E_f^+ E_f^-} + \frac{1}{E_{abce}^+ E_{abce}^- E_f^+ E_f^-}$. According to \Cref{cor:wavefunction-res} we obtain $A_E^{sh}(L, \mathcal G^{\lceil R \rceil})$ by substitution:
    \begin{small}
    $$\frac{1}{(E_{abcd}^+ + E_{abc}^-) (E_{abcd}^- - E_{abc}^-) E_f^+ E_f^-} + \frac{1}{(E_{abce}^+ + E_{abc}^-) (E_{abce}^- - E_{abc}^-) E_f^+ E_f^-}$$
    \end{small}
    Next, we need to compute the wavefunction $\Psi(L,\mathcal H)$. The maximal nested sets are the three chains in $\mathcal H$. As in \Cref{ex:star-wavefunction} we use the triangulation of the three maximal cones $U_N$ from \Cref{prop:max-lightcone-refine} to obtain
    \begin{align*}
        \Psi(L,\mathcal H) &= \frac 1 {E_{\hat 1}^+ (E_{\hat 1}^+ + E_{abc}^-) (E_{abc}^+ + E_a^-) E_a^+} \left(\frac 1 {E_{abc}^+} + \frac 1 {E_{\hat 1}^+ + E_a^-}\right) + \\
        &\hphantom{=} \frac 1 {E_{\hat 1}^+ (E_{\hat 1}^+ + E_{abc}^-) (E_{abc}^+ + E_b^-) E_b^+} \left(\frac 1 {E_{abc}^+} + \frac 1 {E_{\hat 1}^+ + E_b^-}\right) + \\
        &\hphantom{=}\frac 1 {E_{\hat 1}^+ (E_{\hat 1}^+ + E_{abc}^-) (E_{abc}^+ + E_c^-) E_c^+} \left(\frac 1 {E_{abc}^+} + \frac 1 {E_{\hat 1}^+ + E_c^-}\right).
    \end{align*}
    Note that $E_{\hat 1}^+ + E_{abc}^-$ appears indeed as a simple pole in all three terms.

    As a second example, consider $R = \{adef,a\}$. Then
    \[ \mathcal H = \begin{tikzpicture}[scale=1,node distance=20mm,baseline=20pt]
  	\node (1hat) at (0,2) {$\hat 1$};
    \node (adef) at (0,1) {$adef$};
	\node (a) at (0,0) {$a$};

    \draw (1hat) -- (adef);
    \draw (adef) -- (a);
\end{tikzpicture} \qquad \text{ and } \qquad \mathcal G^{\lceil R\rceil} = \begin{tikzpicture}[scale=1,node distance=20mm,baseline=15pt]
  	\node (adef) at (0,1) {$adef$};
	\node (ae) at (-1,0)   {$ae$};
  	\node (ad) at (0,0)   {$ad$};
    \node (f) at (1,0)   {$f$};

    \draw (adef) -- (ae);
    \draw (adef) -- (ad);
    \draw (adef) -- (f);
\end{tikzpicture}\]
and thus $A_E^{sh}([\hat 0,adef],\mathcal G^{\lceil R\rceil})$ is given by
\begin{small}
    $$\frac{1}{(E_{ae}^+ + E_a^-) (E_{ae}^- - E_a^-) E_f^+ E_f^-} + \frac{1}{(E_{ad}^+ + E_a^-) (E_{ad}^- - E_a^-) E_f^+ E_f^-}$$
\end{small}
and 
$$\Psi(L,\mathcal H) = \frac 1 {E_{\hat 1}^+ (E_{\hat 1}^+ + E_{adef}^-) (E_{adef}^+ + E_a^-) E_a^+} \left(\frac 1 {E_{adef}^+} + \frac 1 {E_{\hat 1}^+ + E_a^-}\right).$$
\end{ex}

%%%%%%%%%%%%%%%%%%%%%%
\subsection{The maximal cones and cosmological polytopes}\label{sec:cp}
Fix a maximal nested set $N$. By definition, the maximal cone $U_N$ in the universe fan is cut out by $p_g,m_g,F_g \geq 0$ for $g \in N$ and $p_g = m_g = 0$ for $g \notin N$. View $U_N$ as a full dimensional cone in the $2|N|-1$ dimensional coordinate hyperplane $m_{\hat 1} = 0$ in $\mathbb R^{\dbl N}$. Then $p_g,m_g,F_g$ (for $g \neq \hat{1}$) are the extremal rays of the \emph{dual cone} $U_N^\vee$. In particular, there are $3|N|-3$ of them. On the other hand, for incomparable $f,g \in N$ we have $f\land g = \hat 0$, and so the Hasse diagram of $N$ is a rooted tree $T_N$ with $|N|$ vertices and $|N|-1$ edges. The vertices of $T_N$, $V(T_N)$, are labelled by the $f \in N$, with $\hat{1}$ the top vertex or root. The edges of $T_N$, $E(T_N)$, are labelled by $f \in N$ with $f \neq \hat{1}$: i.e. each edge is labelled by the vertex at the bottom of the edge. See Figure \ref{fig:HasseTN}.

\begin{figure}
\begin{tikzpicture}[
  >=stealth,
  every node/.style={inner sep=2pt},
  edgelab/.style={midway,fill=white,inner sep=1pt,font=\small}
]
  \node (d) at (1,3) {$\hat{1}$};
  \node (b) at (0,1.5) {$b$};
  \node (a) at (-1,0) {$a$};
  \node (c) at ( 1,0) {$c$};
  \node (e) at (2,1.5) {$d$};

  \draw (a) -- (b) node[red, edgelab, left=2pt] {$a$};
  \draw (c) -- (b) node[red, edgelab, right=2pt] {$c$};
  \draw (b) -- (d) node[red, edgelab, left=3pt] {$b$};
  \draw (e) -- (d) node[red, edgelab, right=3pt] {$d$};
\end{tikzpicture}
\caption{The tree $T_N$ for a nested set $N = \{a,b,c,d,\hat{1}\}$ with $a<b$, $c<b$, $b< \hat{1}$ and $d<\hat{1}$. Edges are labelled by their lower vertices.}
\label{fig:HasseTN}
\end{figure}
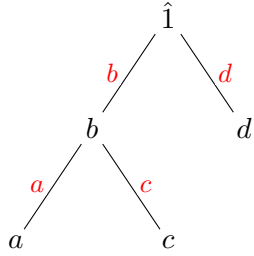

The \emph{cosmological polytope} of a graph is defined by \cite{ArkaniHamed17}. In the case of a tree $T_N$, introduce a vector $\mbf{x}_f$ for every vertex in $V(T_N)$ and a vector $\mbf{y}_f$ (with $f\neq \hat{1}$) for every edge in $E(T_N)$. These vectors span a vector space $\mb{R}^{E(T_N) \cup V(T_N)}$. Then, for an edge $f<g$ of $T_N$, consider the three vectors
\[
\mbf{x_f} + \mbf{x}_g - \mbf{y}_f,\qquad \mbf{x}_f - \mbf{x}_g + \mbf{y}_f,\qquad -\mbf{x}_f + \mbf{x}_g - \mbf{y}_f.
\]
The \textit{cosmological polytope} of $T_N$ is the convex hull of these vectors for all edges of $T_N$.

\begin{prop}\label{prop:cosmological-polytopes}
    For a maximal nested set $N$, the cone dual to $U_N$ is the cone over the cosmological polytope of $T_N$ under a linear isomorphism. The isomorphism is given by
    \begin{align*}
        (\mathbb R^{\dbl N})^\vee/\langle m_{\hat 1}\rangle &\to \mathbb R^{E(T_N) \cup V(T_N)}\\
        p_{\hat 1} &\mapsto 2 \mathbf x_{\hat 1} \\
        p_g &\mapsto \mathbf x_g + \mathbf y_g - \mathbf x_f \\
        m_g &\mapsto \mathbf x_f + \mathbf y_g - \mathbf x_g
    \end{align*}
    where $g < \hat 1$ and $f$ is the unique element covering $g$ in $N$.
\end{prop}
\begin{proof}
    To see that this maps defines an isomorphism, just note that $p_g + m_g \mapsto 2 \mathbf y_g$ and $p_g - m_g \mapsto 2 (\mathbf x_g - \mathbf x_f)$. Using that $p_{\hat 1} \mapsto 2 \mathbf x_{\hat 1}$ and that $N$ defines a tree this shows that the map is a bijection.\par
    Finally, we see that
    $$F_g = p_{\hat 1} + \sum_{\substack{f \in N \\ g < h < \hat 1}} (p_h - m_h) - m_g \mapsto 2 \mathbf x_f - (\mathbf x_f + \mathbf y_g - \mathbf x_g) = \mathbf x_f - \mathbf y_g + \mathbf x_g$$
    where $f$ is the unique element covering $g$, and conclude by definition of the cosmological polytope. Here the sum over $g<f<\hat{1}$ is taken along the unique chain in $N$ above $g$ (since $N$ is nested).
\end{proof}

\begin{rmk}
    The same proof works for any (not necessarily maximal) nested set $N$, replacing $\hat 1$ by the maximal elements of $N$. In particular, all the cones $U_N$ are dual to cosmological polytopes of forests.
\end{rmk}

\begin{rmk}[The universal cosmological polytope]
The dual of the causal cone in $\ker m_{\hat 1} \subseteq \mathbb R^{\dbl \mathcal G}$ is spanned by extremal rays $p_g,m_g,F_g$ for $g < \hat 1$ for all $g \in \mathcal G$. Projectively, the dual is a polytope $Q_{\mc{G}}$ with vertices $p_g,m_g,F_g$. Via \Cref{prop:cosmological-polytopes} it follows then that every cosmological polytope of a tree $T_N$ is a projection of this larger polytope $Q_{\mathcal G}$.
\end{rmk}

%%%%%%%%%%%%%%%%%%%%%%
\subsection{The cones of the universe fan}

By definition, a general cone of the universe fan is cut out by equalities $p_g,m_g = 0$ for $g$ in the complement of some nested set $N$, equalities $p_g = 0$, $m_g = 0$, $F_g = 0$ for some $g$ in $N$ and inequalities $p_g, m_g, F_g \geq 0$ for the remaining $g$. Inspired by \cite{ArkaniHamed17}, we can capture this data as follows:

\begin{dfn}
    A \textit{marking} of a nested set $N$ is a triple $(\beta_N^+,\beta_N^-,\beta_N^\bullet)$ of subsets $\beta_N^+,\beta_N^-,\beta_N^\bullet \subseteq N$ such that
    \begin{enumerate}
        \item $N = \beta_N^+ \cup \beta_N^-$
        \item $\max N \subseteq \beta_N^+ \backslash (\beta_N^- \cup \beta_N^\bullet)$
        \item if $f\in \beta_N^+$ covers $g$ in $N$ then $g \in \beta^-_N \cup \beta^\bullet_N$
        \item if $g\in \beta_N^-$ is covered by $f$ in $N$ then $f \in \beta^+_N \cup \beta^\bullet_N$
        \item if $g \in \beta_N^\bullet$, then $f \in \beta^+_N \cup \beta^\bullet_N$ for $f$ covering $g$ and $h \in \beta^-_N \cup \beta^\bullet_N$ for every $h$ covered by $g$.
    \end{enumerate}
\end{dfn}
\begin{ex}\label{ex:first-look-markings}
    Consider the boolean lattice on the set $\{1,2,3\}$ and take $\mathcal G$ to be all intervals $[i,j]$. The set $N=\{123,12,1\}$ is nested. All of the following are markings of $N$:
    \begin{itemize}
        \item $\{123_+,12_{+-\bullet},1_{+-\bullet}\}$ (i.e., $\beta_N^+ = \{123,12,1\}$, $\beta_N^- = \{12,1\}$, $\beta_N^\bullet = \{12,1\}$)
        \item $\{123_+,12_{+-},1_{\bullet-}\}$
        \item $\{12_{+},1_{-}\}$
    \end{itemize}
    The following are not markings of $N$:
    \begin{itemize}
        \item $\{123_+,12_{+\bullet},1_{\bullet}\}$ (violates $i$ in the definition)
        \item $\{123_-,12_{+-\bullet},1_{+-\bullet}\}$ (violates $ii$)
        \item $\{123_+,12_{+\bullet},1_{+}\}$ (violates $iii$)
    \end{itemize}
\end{ex}

\begin{prop}\label{prop:cone-to-marking}
    Let $C$ be a cone in the universe fan $\mathcal U_{\mathcal G}$ and set
    \begin{align*}
        \beta^+(C) &:= \{g \in \mathcal G \ | \ \exists \ x \in C : p_g(x) \not = 0 \}\\
        \beta^-(C) &:= \{g \in \mathcal G \ | \ \exists \ x \in C : m_g(x) \not = 0\}\\
        N(C) &:= \beta^+ \cup \beta^-\\
        \beta^\bullet(C) &:= \{g \in N \ | \ \exists \ x \in C : F_g(x) \not = 0\}
    \end{align*}
    Then $(N(C),\beta^+(C),\beta^-(C),\beta^\bullet(C))$ is a marked nested set.
\end{prop}
\begin{proof}
    First, note that $N$ must be a nested set by definition of the universe fan. Next, we need to check $(i)-(v)$ in the definition of a marking. $(i)$ holds by definition of $N$. For $ii)$, take a maximal $g$ in $N$. Then on $C$ we have $p_f = 0$ for $f>g$. This implies that $F_g = - m_g$ on $C$, and since $C$ lies in the positive cone, $F_g = m_g = 0$, i.e., $g \notin \beta^-,\beta^\bullet$. For $iii)$, $iv)$ and $v)$, take $f,g \in N$ such that $f$ covers $g$. Note that $F_g = F_f + p_f - m_g$, obtained by comparing adjacent terms in the chain expression for $F_g$. Thus, if $m_g = F_g = 0$ on $C$ we must have $p_f = F_f = 0$. Similarly, if $p_f = F_f = 0$, we must have $F_g = m_g = 0$.
\end{proof}

Conversely, a marking determines a cone of the universe fan by imposing $p_g = 0$ for $g \notin \beta_N^+$, $m_g = 0$ for $g \notin \beta_N^-$, and $F_g = 0$ for $g \in N \backslash \beta_N^\bullet$. To see that the two constructions are inverse to each other, let us determine the rays of the associated cone. Recall from \Cref{prop:universe-rays} that the rays of the universe fan are the vectors $w_R$ for causal regions $R$. Let us call a region $R = \{r_*,r_1,\ldots,r_k\}$ \textit{adapted} to $\beta_N$ if $r_* \in \beta_N^+$, $r_1,\ldots,r_k \in \beta_N^-$ and for every $g \in N$ with $g \leq r_*$, $g \not\leq r_1,\ldots,r_k$, we have $g \in \beta_N^\bullet$.
\begin{prop}\label{prop:universe-cone-rays}
    The cone $C$ associated to a marked nested set $(N,\beta_N)$ is spanned by the rays $w_R$ for all causal regions $R$ adapted to $\beta_N$.
\end{prop}
\begin{proof}
    It suffices to check which rays are contained in $C$. It is easy to see that a ray $w_R$ is contained in $C$ if and only if $R$ is adapted to $(N,\beta_N)$.
\end{proof}

\begin{ex}
    Consider the markings from \Cref{ex:first-look-markings}. Let us compute the cone associated to $\{123_+,12_{+-\bullet},1_{+-\bullet}\}$. The causal regions adapted to this marking are
    $$\{123\}, \{123, 12\}, \{123, 1\}, \{12\}, \{12, 1\} \text{ and } \{1\}.$$
    Consequently, the associated cone is
    $\langle 123^+, 123^+ + 12^-, 123^+ + 1^-, 12^+, 12^+ + 1^-, 1^+\rangle.$
    For $\{123_+,12_{+-},1_{\bullet-}\}$, we obtain the cone $\langle123^+ + 12^-, 12^+, 12^+ + 1^-\rangle$, and $\{12_{+},1_{-}\}$ corresponds to the one-dimensional cone spanned by $12^+ + 1^-$.
\end{ex}

\begin{thm}\label{thm:universe-cones-and-markings}
    The cones of the universe fan are one-to-one with marked nested sets via \Cref{prop:cone-to-marking}.
\end{thm}
\begin{proof}
    Starting with a cone $C$, it is clear that the cone of the marking associated to $C$ is $C$. Now let $(N,\beta_N)$ and $C$ the associated cone. Pick $g \in \beta^+_N$. Let $R$ be the union of all maximal chains of coverings $h_1 \leq \ldots \leq h_k \leq g$ such that $h_2,\ldots,h_k \notin \beta^-_N$. By definition of markings, for every such chain we must have $h_2,\ldots,h_k \in \beta_N^\bullet$ and either $h_1 \in \beta^-_N$, or $h_1$ is minimal and $h_1 \in \beta^\bullet_N$. Let $R^-$ be the set of all $h_1$ with $h_1 \in \beta^-_N$. Then $R:=\{g\} \cup R^-$ is a causal region adapted to $\beta_N$. By \Cref{prop:universe-cone-rays}, $w_R \in C$, and since $p_g(w_R) \not=0$, $g \in \beta^+(C)$. Analogous arguments show that $\beta_N^- = \beta^-(C)$ and $\beta_N^\bullet = \beta^\bullet(C)$.
\end{proof}

\begin{rmk}
    Recall from \Cref{prop:cosmological-polytopes} that the maximal cones of the universe fan are dual to cosmological polytopes. In particular, we recover precisely the markings from \cite{ArkaniHamed17} after restricting to subsets of a fixed nested set $N$.
\end{rmk}

The face lattice of the fan is the inclusion lattice of rays, and so we obtain a full description of the face lattice in terms of marked nested sets. Indeed, let us define a partial order on marked nested sets by letting $(N,\beta_N) \leq (M,\beta_M)$ if $N\subseteq M$ and every causal region adapted to $\beta_N$ is also adapted to $\beta_M$. This is the case if and only if $\beta^+_N \subseteq \beta^+_M$, $\beta^-_N \subseteq \beta^-_M$, $\beta_N^\bullet \subseteq \beta^\bullet_M$ and for every causal region $\{r_*,r_1,\ldots,r_k\}$ adapted to $\beta_N$ and for every $g \in M\backslash N$ with $g\leq r_*$, $g\not\leq r_1,\ldots,r_k$, we have $g \in \beta^\bullet_M$.

\begin{ex}
    Consider the boolean lattice $L$ on the set $\{1,2\}$, with the maximal building set $L \backslash \{\hat 0\}$. Using marked nested sets, we can draw the face lattice of the associated universe fan $\mathcal U_{\mathcal G}$.

    \begin{small}
    \[\begin{tikzcd}[sep=small]
	& {\{12_+,1_{+-\bullet}\}} &&&& {\{12_+,2_{+-\bullet}\}} \\
	{\{12_+,1_{+-}\}} & {\{12_+,1_{+\bullet}\}} & {\{12_+,1_{-\bullet}\}} && {\{12_+,2_{-\bullet}\}} & {\{12_+,2_{+\bullet}\}} & {\{12_+,2_{+-}\}} \\
	& {\{1_+\}} & {\{12_+,1_{-}\}} & {\{12_+\}} & {\{12_+,2_{-}\}} & {\{2_+\}} \\
	&&& \emptyset
	\arrow[no head, from=1-2, to=2-1]
	\arrow[no head, from=1-2, to=2-2]
	\arrow[no head, from=1-2, to=2-3]
	\arrow[no head, from=1-6, to=2-5]
	\arrow[no head, from=1-6, to=2-6]
	\arrow[no head, from=1-6, to=2-7]
	\arrow[no head, from=2-1, to=3-2]
	\arrow[no head, from=2-1, to=3-3]
	\arrow[no head, from=2-2, to=3-2]
	\arrow[no head, from=2-2, to=3-4]
	\arrow[no head, from=2-3, to=3-3]
	\arrow[no head, from=2-3, to=3-4]
	\arrow[no head, from=2-5, to=3-4]
	\arrow[no head, from=2-5, to=3-5]
	\arrow[no head, from=2-6, to=3-4]
	\arrow[no head, from=2-6, to=3-6]
	\arrow[no head, from=2-7, to=3-5]
	\arrow[no head, from=2-7, to=3-6]
	\arrow[no head, from=3-2, to=4-4]
	\arrow[no head, from=3-3, to=4-4]
	\arrow[no head, from=3-4, to=4-4]
	\arrow[no head, from=3-5, to=4-4]
	\arrow[no head, from=3-6, to=4-4]
    \end{tikzcd}\]
    \end{small}

\end{ex}

%%%%%%%%%%%%%%%%%%%%%%
%%%%%%%%%%%%%%%%%%%%%%
%%%%%%%%%%%%%%%%%%%%%%
%%%%%%%%%%%%%%%%%%%%%%
\section{Trees, tubes, and boolean building sets}\label{sec:park}
The Universe Fan was defined in Section~\ref{sec:universe} inside the \emph{doubled} vector space $\mathbb{R}^{2\mc{G}}$, with two generators $f^+$ and $f^-$ for each $f\in\mc{G}$. The goal of this section is to explain a simple graph interpretation of these doubled generators after restricting to a fixed cone, i.e. to a fixed nested set $N\subseteq \mc{G}$. Each $f^\pm$ corresponds to a subset of $N$ obtained by cutting a single edge in the Hasse diagram of $N$ to produce two connected components. This viewpoint leads us to study subdivisions of $\mathbb{R}_+^N$ defined by piecewise linear functions. The nested set fan is a special case.  In Section~\ref{sec:refinements} we pull back these subdivisions to obtain global refinements of the universe fan. One such refinement is used to complete the proof of Theorem~\ref{prop:universe-rays}. Finally, these subdivisions also serve as a combinatorial input to the results on polytopes and normal fans in Section~\ref{sec:cosmohedra}.

%%%%%%%%%%%%%%%%%%%%%%
\subsection{Cuts of $T_N$ and the canonical projection}
\label{sec:cuts-projection}
Fix a nested set $N\subseteq \mc G$. Since $f\wedge g=\hat 0$ for incomparable $f,g\in N$, the Hasse diagram of $N$ is a forest. Let $\overline T_N$ denote the connected rooted tree obtained from the Hasse diagram of $N$ by adjoining a common root vertex $\star$ (Figure \ref{fig:rooA}). We regard the root as the `top' of the tree, and each edge of $\overline T_N$ is uniquely labelled by its lower endpoint. In other words, we identify its edge and vertex sets as
\[
E(\overline T_N)\;\cong\;N, 	\qquad 	V(\overline T_N)\;=\;N\cup\{*\}.
\]

\begin{figure}[h]
\begin{subfigure}{0.4\textwidth}
\centering
\begin{tikzpicture}[every node/.style={inner sep=2pt}]
  \node (one) at (-1,0) {$1$};
  \node (thr) at ( 1,0) {$3$};
  \node (abc) at ( 0,1.6) {$123$};
  \node (six) at (3,0) {$6$};
  \node (fsx) at (3,1.6) {$56$};
  \node[red] (roo) at (1.5,2.4) {$\star$};
{
\small
  \draw (one) -- node[midway,blue,left=2pt]{$1$} (abc);
  \draw (thr) -- node[midway,blue,right=2pt]{$3$} (abc);
  \draw (six) --node[midway,blue,right=2pt]{$6$} (fsx);
  \draw[red] (abc) --node[midway,blue,above left=1pt]{$123$} (roo) --node[midway,blue,above right=1pt]{$56$}  (fsx);
}
\end{tikzpicture}
\caption{}
\label{fig:rooA}
\end{subfigure}
\begin{subfigure}{0.4\textwidth}
\centering
\tikzset{
  treenode/.style={inner sep=1.5pt},
  amoeba/.style={draw, thick, fill opacity=0.15, draw opacity=0.8}
}
\newcommand{\Amoeba}[3][]{%
  \begin{scope}[on background layer]
    \path[amoeba, #1]
      plot[smooth cycle, tension=0.4]
      coordinates {#3};
  \end{scope}%
}
\begin{tikzpicture}[every node/.style={inner sep=2pt}]
  \node (one) at (-1,0) {$1$};
  \node (thr) at ( 1,0) {$3$};
  \node (abc) at ( 0,1.6) {$123$};
  \node (six) at (3,0) {$6$};
  \node (fsx) at (3,1.6) {$56$};
  \node (roo) at (1.5,2.4) {$\star$};
{
\small
  \draw (one) -- (abc);
  \draw (thr) -- (abc);
  \draw (six) -- (fsx);
  \draw (abc) -- (roo) -- (fsx);
}
\path
%  coordinate (p1) at ($(abc.north)$)
  coordinate (p2) at ($(abc.north east)$)
  coordinate (p3) at ($(thr.north east)$)
  coordinate (p4) at ($(thr.south)$)
  coordinate (p5) at ($(one.south)$)
  coordinate (p6) at ($(one.north west)$)
  coordinate (p7) at ($(abc.north west)$);
\Amoeba[fill=red, draw=red]{H}{(p2)(p3)(p4)(p5)(p6)(p7)}
\path
  coordinate (p1) at ($(abc.south east)$)
  coordinate (p2) at ($(one.south east)$)
  coordinate (p3) at ($(one.west)$)
%  coordinate (p4) at ($(one.west)$)
  coordinate (p5) at ($(abc.north west)$)
  coordinate (p6) at ($(roo.north)$)
%  coordinate (p7) at ($(roo.north east)$)
  coordinate (p8) at ($(fsx.north east)$)
  coordinate (p9) at ($(six.south east)$)
  coordinate (p10) at ($(six.south west)$)
%  coordinate (p11) at ($(six.west)$)
  coordinate (p12) at ($(fsx.south west)$)
  coordinate (p13) at ($(roo.south)$);
\Amoeba[fill=blue, draw=blue]{H}{(p1)(p2)(p3)(p5)(p6)(p8)(p9)(p10)(p12)(p13)}
\end{tikzpicture}
\caption{}
\label{fig:rooB}
\end{subfigure}
\caption{(A) The Hasse diagram for $N = \{1,3,123,6,56\}$ (black) and the connected tree $\overline{T}_N$ obtained by adding a root $\star$. Edge labels are shown in blue. (B) The subtree $123^+$ (red) and the subtree $3^-$ (blue) intersect in a subtree with edge set $E(123^+) \cap E(3^-) = \{1\}$.}
\label{fig:roo}
\end{figure}

Now take any edge $f\in N$ of $\overline T_N$. Removing this edge produces two connected trees: the tree below $f$, and the complementary component. The edge sets of the two trees are
\[
E({f^+}) =\{\, g\in N \,|\, g < f \,\},
\qquad
E({f^-}) =\{\, g\in N \,|\,  g>f \ \text{or}\ g\wedge f=\hat 0 \,\}.
\]
Also write $V(f^\pm)\subseteq V(\overline T_N)$ for the vertex sets of these components.

As suggested by the notation, we identify the generators $f^\pm$ of $\mathbb{R}^{2N}$ with these subtrees of $\overline T_N$. In fact, we can view positive integer linear combinations of these generators as being subgraphs of $\overline T_N$ and vice versa. Consider the edge sets $E(f^+)$ and $E(g^-)$ for some $g < f$ in $N$, as the associated subtrees. The union of these two subtrees is the whole tree, $E(f^+) \cup E(g^-) = N$. Their intersection of these two subtrees is the subtree below $f$ but not below $g$, $E(f^+) \cap E(g^-) = \{ h\,|\, h< f, \, h\not < g\}$ (see Figure \ref{fig:rooB}). Motivated by this, define a map
\[
\phi:\mathbb{R}^{\dbl N}\longrightarrow \mathbb{R}^N/\langle u_{\mathrm{tot}} \rangle,\qquad \phi(f^\pm)\;:=\;\sum_{g\in N_{f^\pm}} g,
\]
where $u_{\mathrm{tot}}$ is the vector corresponding to the whole tree, $u_{\mathrm{tot}} = \sum_{g \in N} g$. $\phi$ is a surjection. Moreover, if $N$ has a unique maximal element, $f_{\max}$, note that $\phi(f_{\max}^-) = 0$, so that $\phi$ is well defined on the quotient space $\RR^{2N} / \langle f_{\max}^- \rangle$.

In Section \ref{sec:lightcone}, we will see that subdivisions of $\RR_+^N$ pullback under $\phi$ to give subdivisions of $\RR_+^{2N}$. To this end, we collect some simple statements about subdivisions of $\RR_+^N$ in Sections \ref{sec:boole} and \ref{sec:boole2}, below.

\begin{ex}\label{ex:phi}
For a first example of how we use $\phi$, take $N = \{g,f\}$ with $f > g$. Then $\overline{T}_N$ is the linear graph with two edges. The map $\phi$ acts on generators as
\[
\phi: \RR^{2N}/\langle f^-\rangle \rightarrow \RR^N / \langle f+g\rangle,\qquad \phi:~ f^+ \mapsto g,\qquad g^+ \mapsto 0,\qquad g^- \mapsto f.
\]
Then consider the subdivision of $\RR^N_+$ into the two cones $\langle f, f+g\rangle$ and $\langle g,f+g\rangle$. These pullback under $\phi$ to give a subdivision of $\RR_+^{2N}/\langle f^-\rangle$ into two simplicial cones: $\langle g^+, g^-, g^-+f^+\rangle$ and $\langle g^+, f^+, g^-+f^+\rangle$. See Figure \ref{fig:proj}.
\end{ex}
\begin{figure}
\centering
\begin{tikzpicture}[every node/.style={inner sep=3pt}]
  \coordinate (pg) at (-0.7,0); \node[left] at (pg) {$g^+$};
  \coordinate (mg) at (1,1); \node[above] at (mg) {$g^-$};
  \coordinate (pf) at (1,-1); \node[below] at (pf) {$f^+$};
  \draw[->] (1.5,0) --node[midway,above]{$\phi$} (2.5,0);
  \coordinate (f) at (3,1); \node[above] at (f) {$f$};
  \coordinate (g) at (3,-1); \node[below] at (g) {$g$};
  
   \draw (pg) -- (mg) -- (pf) -- (pg);
  \draw (f) -- (g);

  \coordinate (m) at (3,0); \node[red] at (m) {$\bullet$};
  \coordinate (mpre) at (1,0); \node[red] at (mpre) {$\bullet$};
  \draw[red] (pg) -- (mpre);
\end{tikzpicture}
\caption{Cartoon of the map $\phi$ for $N = \{f,g\}$ ($f>g$) in Example \ref{ex:phi}. The pullback of a subdivision of $\mathbb{R}_+^N$ defines a subdivision of $\mathbb{R}_+^{2N}$ (shown in red).}
\label{fig:proj}
\end{figure}
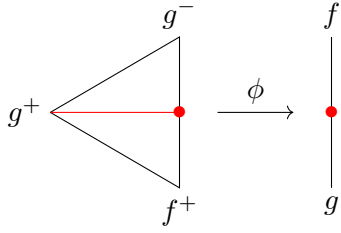

\begin{rmk}
The surjection $\phi$ factors through a natural isomorphism which records subtrees by both its edges \emph{and} vertices. Let $\mathbb{R}^{E(\overline T_N)}$ and $\mathbb{R}^{V(\overline T_N)}$ be the vector spaces with bases the edges and vertices of $\overline T_N$, and define
\[
\hat{u}_{\mathrm{tot}}\;:=\;\sum_{e\in E(\overline T_N)} e \;+\; \sum_{v\in V(\overline T_N)} v.
\]
Define
\begin{equation}\label{eq:cut-iso}
\psi:\mathbb{R}^{\dbl N} \longrightarrow \bigl(\mathbb{R}^{E(\overline T_N)}\oplus \mathbb{R}^{V(\overline T_N)}\bigr)/\langle \hat{u}_{\mathrm{tot}}\rangle,\qquad
\psi(f^\pm) :=\; \sum_{e\in E(f^\pm)} e \;+\; \sum_{v\in V(f^\pm)} v.
\end{equation}
This is a linear isomorphism. The inverse map $\psi^{-1}$ can be described explicitly:
\begin{itemize}
\item If $v\in V(\overline T_N)\setminus\{*\}$, let $f$ be the unique edge above $v$ (toward the root) and let $g_1,\dots,g_k$ be the edges below $v$.
Then
\[
\psi^{-1}(v)\;=\; f^+ \;+\; g_1^- \;+\;\cdots\;+\; g_k^-.
\]
\item For an edge $e\in E(\overline T_N)$, one has
\[
\psi^{-1}(e)\;=\;-\,(e^+ + e^-).
\]
\end{itemize}
(We view $e$ and $v$ as basis vectors in $\mathbb{R}^{E(\overline T_N)}$ and $\mathbb{R}^{V(\overline T_N)}$, and we implicitly pass to the quotient.) We emphasize that this isomorphism is \emph{not} the same as the cut-related identification appearing in Section~\ref{sec:cp}; the two maps encode different data and play different roles in the paper.
\end{rmk}

%%%%%%%%%%%%%%%%%%%%%%
\subsection{Subdivisions from Boolean building sets}
\label{sec:boole}
We have seen that (integer) points of $\RR_+^{2N}$ are naturally associated to subtrees of $\overline T_N$ (via the map $\phi$). Since $E(\overline T_N)=N$, any subtree determines a subset of $N$ by its edge set. These have a natural poset structure, ordered by inclusion in the Boolean lattice $2^N$. The edge sets of connected subtrees of $\overline T_N$ form a distinguished family of subsets in $2^N$. This is an example of a \emph{Boolean building set}, which we now review. Boolean building sets encode the nesting properties of the subtrees of $\overline T_N$. It is this combinatorics of nested subtrees that allows us to define the fan refinements and polytopes we study in Sections~\ref{sec:refinements} and~\ref{sec:cosmohedra}.

\smallskip
A Boolean building set on $N$ is a collection of subsets $\mc B \;\subseteq\; 2^N$ containing all singletons, and such that whenever $S,S'\in\mc B$ intersect, their union again lies in $\mc B$:
\[
S\cap S'\neq\emptyset \ \Longrightarrow\ S\cup S' \in \mc B
\qquad
\text{for all }S,S'\in \mc B.
\]
In other words, if two subgraphs in our collection share a common edge, then their union should also be in our collection.

\smallskip
An important special case is the \emph{graphical} building set of the tree $\overline T_N$, given by all the connected subgraphs:
\[
\mc B(\overline T_N) :=
\{\, S\subseteq N : \text{$S$ is the edge set of a connected subgraph of $\overline T_N$} \,\}.
\]
This choice will be used in Section~\ref{sec:refinements} to produce a simplicial refinement of the constructions of Section~\ref{sec:universe}, culminating in Proposition~\ref{prop:max-lightcone-refine}.

Now take $\mc B$ an arbitrary Boolean building set on $N$. Let $x=(x_g)_{g\in N}$ denote coordinates on $\RR^N$, and consider the positive orthant, $x_g \geq 0$. Consider the following two piecewise linear functions on $\RR_+^N$:
\[
\alpha^{\max}_{\mc B}(x)\;:=\;\sum_{S\in\mc B}\max_{g\in S} x_g,
\qquad
\alpha^{\min}_{\mc B}(x)\;:=\;\sum_{S\in\mc B}\min_{g\in S} x_g.
\]
The domains of linearity of $\alpha^{\max}_{\mc B}$ and $\alpha^{\min}_{\mc B}$ both define polyhedral subdivisions of $\RR_+^N$. We denote the corresponding fans as
\[
\Sigma^{\max / \min}_{\mc B} \ =\;
\text{the fan of domains of linearity of $\alpha^{\max/\min}_{\mc B}$}.
\]
Adding a constant to all coordinates shifts both $\max_{g\in S}x_g$ and $\min_{g\in S}x_g$ by the same constant, so the domains of linearity are invariant under translation by $u_{\mathrm tot} = \langle(1,\dots,1)\rangle$. Consequently, both of these fans descend to the quotient and define subdivisions of $\RR_+^N/\langle u_{\mathrm tot} \rangle$.

\begin{ex}[Barycentric subdivision]
Let $\mc{B} = 2^N \backslash \emptyset$ be the building set consisting of \emph{all} subsets of $N$. Then each of the domains of linearity of $\alpha^{\max}_{\mc B}$ (and similarly of $\alpha^{\min}_{\mc B}$) is given by the cones cut out by imposing a total order on $N$:
\[
x_{g_1} \geq x_{g_2} \geq \cdots \geq x_{g_n}.
\]
In other words, $\Sigma^{\max}_{\mc B}$ is the (cone over the) barycentric subdivision of the simplex (also known as the braid or permutohedral fan).
\end{ex}

%%%%%%%%%%%%%%%%%%%%%%
\subsection{Combinatorial model}\label{sec:boole2}
A domain of linearity of $\alpha^{\max}_{\mc B}$ is given by a consistent choice of a maximal element for each $S \in \mc{B}$. Such a choice defines a tree order on $N$ that we write as a rooted tree $T$, with vertex set $N$. This is a directed tree, where the root is the source (i.e. the greatest element). The rooted trees that arise in this way are known as \textit{$\mc{B}$-trees}:
\begin{dfn}[{\cite[Def.~7.7]{Pos09}}]
A $\mc B$-tree is a rooted tree $T$ with vertex set $N$ that satisfies:
\begin{enumerate}
\item For each vertex $i$, the set $T_i = \{ j \in N \, | \, j \leq_T i\}$ of vertices below $i$ (with respect to the tree order) is an element of $\mc{B}$, and
\item there are no subsets of pairwise incomparable vertices $i,j,\ldots,k$ such that $T_i \cup T_j \cup \cdots \cup T_k \in \mc{B}$.
\end{enumerate}
\end{dfn}

For any rooted tree $T$ on $N$, write the associated cones in $\RR_+^N$ as
 \[
 C^T = \{ x_i \geq x_j \text{ for each edge } i \rightarrow j \in T \},\qquad C_T = \{ x_i \leq x_j \text{ for each edge } i \rightarrow j \in T \}.
 \]
Dually, the generators of $C^T$ are given by the subtrees of $T$ that contain the root,
\[
C^T = \left\langle \sum_{i \in T'} e_i \ \big| \ T' \text{ a rooted subtree of } T \right\rangle.
\]
On the other hand, the generators of $C_T$ are
\[
C_T = \left\langle \sum_{j \in T_i} e_j \ \big| \ i \in N \right\rangle,
\]
where $T_i$ denotes the descendant set of $i$ (including $i$ itself).
We summarize this discussion as:
 \begin{prop}
For a Boolean building set $\mc{B}$, the maximal cones of $\Sigma^{\max}_{\mc B}$ are given by the cones $C^T$ and the maximal cones of $\Sigma^{\min}_{\mc B}$ are given by the cones $C_T$, for all $\mc{B}$-trees $T$.
\end{prop}

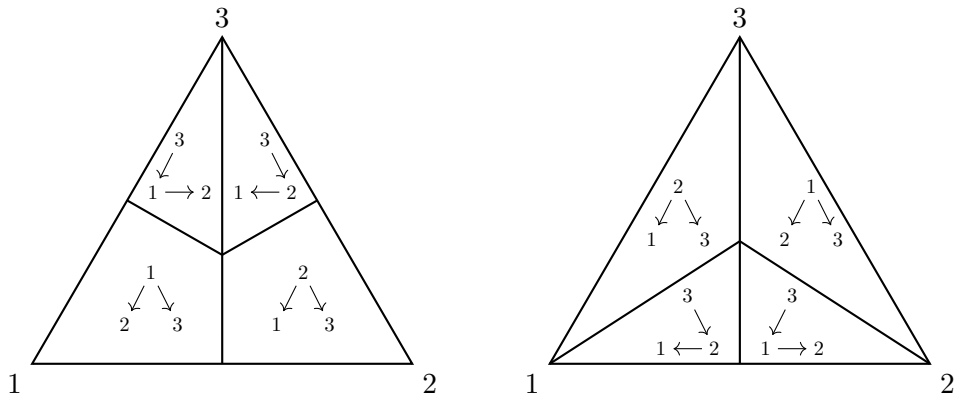
\begin{figure}
\begin{center}
\newcommand{\TreeA}[1][1]{
  \tikz[scale=0.7, transform shape]{
    \node (1) at (1,1) {1};
    \node (2) at (0.5,0) {2};
    \node (3) at (1.5,0) {3};
    \draw[->] (1) -- (2);
    \draw[->] (1) -- (3);
  }
}
\newcommand{\TreeB}[1][1]{
  \tikz[scale=0.7, transform shape]{
                \node (2) at (1,1) {2};
                \node (1) at (0.5,0) {1};
                \node (3) at (1.5,0) {3};
                \draw[->] (2) -- (1);
                \draw[->] (2) -- (3);
  }
}
\newcommand{\TreeC}[1][1]{
  \tikz[scale=0.7, transform shape]{
                \node (3) at (1,1) {3};
                \node (1) at (0.5,0) {1};
                \node (2) at (1.5,0) {2};
                \draw[->] (3) -- (2);
                \draw[->] (2) -- (1);
  }
}
\newcommand{\TreeD}[1][1]{
  \tikz[scale=0.7, transform shape]{
                \node (3) at (1,1) {3};
                \node (1) at (0.5,0) {1};
                \node (2) at (1.5,0) {2};
                \draw[->] (3) -- (1);
                \draw[->] (1) -- (2);
  }
}

\begin{tikzpicture}[scale=2.5]
  \coordinate (A) at (0,0);
  \node at (A) [below left] {$1$};
  \coordinate (B) at (2,0);
  \node at (B) [below right] {$2$};
  \coordinate (C) at (1,{sqrt(3)});
  \node at (C) [above] {$3$};
  \coordinate (Mab) at ($(A)!0.5!(B)$);
  \coordinate (Mbc) at ($(B)!0.5!(C)$);
  \coordinate (Mca) at ($(C)!0.5!(A)$);
  \coordinate (G) at (1,{sqrt(3)/3});

  \draw[thick] (A) -- (B) -- (C) -- cycle;
  \draw[thick] (Mca) -- (G);
  \draw[thick] (G) -- (Mbc);
  \draw[thick] (Mab) -- (C);

  \coordinate (TA) at ($ (A)!0.6!(G) $);
  \coordinate (TB) at ($ (B)!0.6!(G) $);
  \coordinate (TC) at (1.2,1.05);
  \coordinate (TD) at (0.75,1.05);
  
  \node at (TA) {\TreeA};
  \node at (TB) {\TreeB};
  \node at (TC) {\TreeC};
  \node at (TD) {\TreeD};
\end{tikzpicture}
\qquad
\begin{tikzpicture}[scale=2.5]
  \coordinate (A) at (0,0);
  \node at (A) [below left] {$1$};
  \coordinate (B) at (2,0);
  \node at (B) [below right] {$2$};
  \coordinate (C) at (1,{sqrt(3)});
  \node at (C) [above] {$3$};
  \coordinate (Mab) at ($(A)!0.5!(B)$);
  \coordinate (Mbc) at ($(B)!0.5!(C)$);
  \coordinate (Mca) at ($(C)!0.5!(A)$);
  \coordinate (G) at (1,0.65);

  \draw[thick] (A) -- (B) -- (C) -- cycle;
  \draw[thick] (A) -- (G);
  \draw[thick] (G) -- (B);
  \draw[thick] (Mab) -- (C);

  \coordinate (TA) at ($ (G)!0.7!(Mbc) $);
  \coordinate (TB) at ($ (G)!0.7!(Mca) $);
  \coordinate (TD) at (1.25,0.22);
  \coordinate (TC) at (0.7,0.22);
  
  \node at (TA) {\TreeA};
  \node at (TB) {\TreeB};
  \node at (TC) {\TreeC};
  \node at (TD) {\TreeD};
\end{tikzpicture}
\caption{The max-subdivision $\Sigma^{\max}_{\mc B}$ (left) and min-subdivision $\Sigma^{\min}_{\mc B}$ (right) induced by the building set $\mc{B} = \{1,2,3,12,123\}$. In both cases, the cones are labelled by $\mc{B}$-trees. The min-subdivision is simplicial, but the max-subdivision is not.}
\label{fig:subdiv}
\end{center}
\end{figure}

\begin{ex}
    Let $\mathcal B = \{1,2,3,12,123\}$. The associated causal and cocausal subdivisions of $\Delta_3$ are the domains of linearity of
\[
\max(x_1,x_2,x_3) + \max(x_1,x_2),\qquad \text{and} \qquad \min(x_1,x_2,x_3) + \min(x_1,x_2),
\]
respectively. In both cases, we can label the maximal cones by the following rooted trees:
\[
\newcommand{\TreeA}[1][1]{
  \tikz[scale=0.7, transform shape, baseline=5]{
    \node (1) at (1,1) {1};
    \node (2) at (0.5,0) {2};
    \node (3) at (1.5,0) {3};
    \draw[->] (1) -- (2);
    \draw[->] (1) -- (3);
  }
}
\newcommand{\TreeB}[1][1]{
  \tikz[scale=0.7, transform shape, baseline=5]{
                \node (2) at (1,1) {2};
                \node (1) at (0.5,0) {1};
                \node (3) at (1.5,0) {3};
                \draw[->] (2) -- (1);
                \draw[->] (2) -- (3);
  }
}
\newcommand{\TreeC}[1][1]{
  \tikz[scale=0.7, transform shape, baseline=5]{
                \node (3) at (1,1) {3};
                \node (1) at (0.5,0) {1};
                \node (2) at (1.5,0) {2};
                \draw[->] (3) -- (2);
                \draw[->] (2) -- (1);
  }
}
\newcommand{\TreeD}[1][1]{
  \tikz[scale=0.7, transform shape, baseline=5]{
                \node (3) at (1,1) {3};
                \node (1) at (0.5,0) {1};
                \node (2) at (1.5,0) {2};
                \draw[->] (3) -- (1);
                \draw[->] (1) -- (2);
  }
}
T_1 = \TreeA,\qquad T_2 = \TreeB,\qquad T_3 = \TreeC, \qquad T_4 = \TreeD.
\]
These are precisely the $\mc B$-trees for this $\mc B$. For the max-subdivision we read an arrow $i \rightarrow j$ in one of these trees, $T$, as an inequality $x_i \geq x_j$ defining the cone $C^T$. Whereas for the min-subdivision we read an arrow $i \rightarrow j$ as $x_i \leq x_j$, defining the cone $C_T$. The two subdivisions are given in Figure \ref{fig:subdiv}. In the max-subdivision, the generators of each cone $C^T$ are given by rooted subtrees of $T$ that contain the root. For example,
 \[
 C^{T_1} = \langle e_1, e_1 + e_2, e_1 + e_3, e_1 + e_2 + e_3 \rangle.
 \]
In the min-subdivision, the generators of each cone $C_T$ are given by descendant sets $T_i$. For example,
 \[
 C_{T_1} = \langle e_2, e_3, e_1 + e_2 + e_3 \rangle.
 \]
Here $T_2=\{2\}$, $T_3=\{3\}$, and $T_1=\{1,2,3\}$. Note that the min-subdivision is \emph{simplicial}, whereas the max-subdivision is not.
\end{ex}

\begin{rmk}
The sets $T_i$ of a $\mc{B}$-tree form a \emph{nested set} of $\mc{B}$, and the min-subdivision $\Sigma^{\min}_{\mc B}$ is the nested set fan associated to $\mc B$ (e.g. \cite{Pos09}). More broadly, nested set complexes and their realizations as simplicial fans occur throughout combinatorics and tropical geometry; for context see e.g.\ \cite{FeiStu04,ArdKli06}.
\end{rmk}

%%%%%%%%%%%%%%%%%%%%%%
%%%%%%%%%%%%%%%%%%%%%%
%%%%%%%%%%%%%%%%%%%%%%
%%%%%%%%%%%%%%%%%%%%%%
\section{Lightcone Refinements}\label{sec:refinements}

In this section we show how the considerations of the previous section lead to refinements of the doubled free nested set fan $\Sigma^{\dbl \mathcal G}$, which we will call \textit{lightcone refinements}. We will see that the universe fan $\mathcal U_{\mathcal G}$ arises naturally as a subfan of the minimal such refinement (\Cref{prop:universe-as-subfan}); more generally, every lightcone refinement contains a subfan that refines $\mathcal U_{\mathcal G}$. In a special case, this yields a unimodular triangulation of $\mathcal U_{\mathcal G}$ indexed by \textit{tubings} of the Hasse diagrams of nested sets (\Cref{prop:max-lightcone-refine}).
%\footnote{In view of \Cref{prop:cosmological-polytopes} this will not be a surprise to those familiar with cosmological polytopes.}

Recall that the cones of the doubled free nested set fan $\Sigma^{\dbl \mathcal G}$ are the simplices $\mathbb R_+^{\dbl N}$ and their faces. We start with introducing \textit{lightcone subdivisions} of these simplices.

%%%%%%%%%%%%%%%%%%%%%%
\subsection{Lightcone subdivisions}\label{sec:lightcone}

 Let us fix a nested set $N$ and consider the associated cone $\mathbb R_+^{\dbl N}$, spanned by $f^+,f^-$ for $f\in N$.

Let $\overline T_N$ be the tree associated to $N$ as defined in \Cref{sec:cuts-projection}, and let $\mathcal B_N\subseteq N = E(\overline T_N)$ be a boolean building set on $N$. We denote the associated nested set fan by $\Sigma_{\mathcal B_N} = \Sigma^{\mathrm{min}}_{\mathcal B_N}$. As explained in \Cref{sec:cuts-projection} we can interpret the symbols $f^+,f^-$ as cuts in the tree $\overline T_N$. Let
\begin{align*}
    \phi: \mathbb R^{\dbl N} &\to \mathbb R^{N}/\langle u_{\mathrm{tot}} \rangle  \\
    c &\mapsto \sum_{e \in E(c)} e
\end{align*}
be the map sending a cut to its edges. Here we have again $u_{\mathrm{tot}} := \sum_{e\in \overline T_N} e = \sum_{f\in N} f$.

\begin{dfn}\label{dfn:lightcone-subdivision}
    Let $\Gamma$ denote the pullback of $\Sigma_{\mathcal B_N}/u_{\mathrm{tot}}$ under $\phi$. We call $\Gamma \cap \mathbb R^{\dbl N}_+$ the \textit{lightcone subdivision} associated to $\mathcal B_N$.
\end{dfn}

Let us explore these subdivisions in more detail. For this, recall the isomorphism \eqref{eq:cut-iso}
\begin{align*}
        \psi: \mathbb R^{\dbl N} &\to (\mathbb R^{E(\overline T_N)} \oplus \mathbb R^{V(\overline T_N)})/\langle \hat u_{\mathrm{tot}} \rangle  \\
        c &\mapsto \sum_{e \in E(c)} e + \sum_{v \in V(c)} v,
\end{align*}
that sends a cut to its edges and vertices.

We can write $\phi = p_E \circ \psi$ where $p_E$ is the projection to $\mathbb R^{E(\overline T_N)}/u_{\mathrm{tot}}$.

We introduce the following terminology: a subgraph $s$ of $\overline T_N$ is \textit{spanned} by a nested set $M\subseteq \mathcal B_N$ if it is a vertex or $s = t_1 \cup \ldots \cup t_k$ for $t_i \in M$ such that there is some $t_i$ with at least as many connected components as $s$. In other words, the $t_i$ form a chain of vertex-overlapping subgraphs.
 
For a subgraph $s$ that is spanned by $M$, let us set
\begin{equation}\label{eq:ws-def}
    w_s := \sum_{e \in \delta^+(s)} e^+ + \sum_{e \in \delta^-(s)} e^-
\end{equation}
where $\delta^+(s),\delta^-(s)$ are the sets of edges in $\overline T_N$ entering or leaving $s$, respectively. Note that $w_s =  \psi^{-1} \left(\sum_{e\in s} e + \sum_{v\in s} v\right)$.

Let us describe the maximal cones of the lightcone subdivision.

\begin{prop}\leavevmode \label{prop:lightcone-subdivision}
    Let $\Gamma$ denote the lightcone subdivision of $\mathbb R^{\dbl N}_+$ associated to $\mathcal B_N$.
    \begin{enumerate}
        \item The maximal cones of $\Gamma$ are indexed by maximal nested sets $M \subseteq \mathcal B_N$. The cone associated to $M$ is spanned by the vectors $w_s$ for subgraphs $s\not=\overline T_N$ spanned by $M$.
        \item The maximal cones of $\Gamma$ are the domains of linearity of
        $$\sum_{t \in \mathcal B_N} \min_{e \in t}(\lambda_e)$$
        on $\mathbb R^{\dbl N}_+$, where $$\lambda_e := \sum_{\substack{f \in N \\ e \in f^+}} p_f + \sum_{\substack{f \in N \\ e \in f^-}} m_f.$$
    \end{enumerate}
\end{prop}
\begin{proof}
    By definition, the lightcone subdivision is the pullback of $\Sigma_{\mathcal B_N}/\langle u_{\mathrm{tot}}\rangle$ under $\phi = p_E \circ \psi$.
    Pulling back $\Sigma_{\mathcal B_N}$ along $p_E$ yields the complete fan $(\Sigma_{\mathcal B_N} \times \mathbb R^{V(T_N)})/\langle \hat u_{\mathrm{tot}}\rangle$. A maximal cone is given by $(D_M \times \mathbb R^{V(\overline T_N)})/\langle \hat u_{\mathrm{tot}}\rangle$, where $D_M$ is the cone associated to a maximal nested set $M$ in $\mathcal B_N$. Its rays are the vectors $\sum_{e \in E(t)} e$ for $t \in M, t \not= \overline T_N$ (Section \ref{sec:park}) and the vectors $v,-v$ for $v \in V(\overline T_N)$.

    The image of the positive orthant $\mathbb R_+^{\dbl N}$ under $\psi$ is the cone $P$ cut out by all inequalities $e^* \leq v^*$ where $v$ is a vertex of $e$. In particular, under the isomorphism, the maximal cones of the lightcone subdivision are the intersections $C_M := (D_M \times \mathbb R^{V(\overline T_N)}/\langle \hat u_{\mathrm{tot}}\rangle) \cap P$. Thus, we need to prove that its rays are the vectors $w'_s := \sum_{e\in s} e + \sum_{v\in s} v$ for $s$ spanned by $M$.

Clearly, the $w'_s$ are contained in $C_M$. Let us prove that they are spanning. We can write any element of $C_M$ as $x + y$ where $x \in D_M$ and $y \in \mathbb R^{V(T_N)}$ with $v^*(y) \geq e^*(x)$ whenever $v$ is a vertex of $e$.

Certainly $\{0\} \times \mathbb R^{V(\overline T_N)}_+$ is in the span. If there is some $x+y \in D_M$ which is not in the span of the $v$ and $w'_s$, we can choose it such that the support of $x$ is minimal. As in the proof of \Cref{prop:nested-projection} one argues that the support of $x$ is given by a union of edge-disjoint flats $t_i \in M$; however, they might overlap in vertices. Consider the connected components of the graph with vertices $t_i$ and edges $\{t_i,t_j\}$ whenever $t_i$ and $t_j$ overlap in a vertex. For every such connected component $C$, we obtain a subgraph $s_j$ of $\overline T_N$ that is spanned by $M$ by taking the union of all $t_i$ contained in $C$.

Let $\lambda:= \min_{e^*(x)>0} e^*(x)$. Then the support of $x - \lambda \sum_i \sum_{e\in s_i} e$ is strictly smaller than the support of $x$. Moreover, by construction we have $e^*(\sum_i w'_{s_i}) = v^*(\sum_i w'_{s_i})$ for every edge $e$ in the support and vertex $v$ of $e$. Thus, $x+y-\lambda \sum_i w'_{s_i} \in D_M$. Now by assumption this implies that $x+y$ is in the span of the $v$ and $w'_s$, contradicting the choice of $x+y$. Thus, $D_M$ is spanned by the $v$ and $w'_s$, as desired.

   A similar argument shows that none of the $w'_s$ can be expressed as a positive linear combination of other $w'_s$ and $v$, and so we conclude.

    For $ii)$, just note that the fan $\Sigma_{\mathcal B_N}$ is given by the domains of linearity of
    $$\sum_{t \in\mathcal B_N} \min_{e \in t}(e^*)$$
    on the positive orthant (Section \ref{sec:park}). We conclude since $\lambda_e$ is by definition just the image of $e^*$ under the dual of $\phi$.
\end{proof}

Two choices for the building set $\mathcal B_N$ are of special interest for us. The first one is the \textit{minimal building set}, given by $\mathcal B_N := \{N\} \cup \{f \ | \ f\in N\}$. We call the associated subdivision the \textit{minimal lightcone subdivision}.

\begin{prop}\label{prop:cosmo-as-subcone}
   $U_N$ is a cone of the minimal lightcone subdivision of $\mathbb R_+^{\dbl N}$.
   \end{prop}
   
    Recall that $U_N$ denotes the intersection of $\mathbb R_+^{\dbl N}$ with the causal cone (\Cref{dfn:causal-cone}), and can be identified with the dual of a cosmological polytope (see \Cref{prop:cosmological-polytopes}). Here, we view $\mathbb R_{\ge 0}^{\dbl N}$ as the coordinate subspace of $\mathbb R_{\ge 0}^{\dbl \mathcal G}$ with $p_f=m_f=0$ for $f\notin N$.

\begin{proof}
    Let $h_1,\ldots,h_\ell$ denote the maximal elements of $N$.
    We claim that $U_N$ is the domain of linearity
    $$\{x \in V(p_{h_1},\ldots,p_{h_\ell}) \cap \mathbb R_+^{\dbl N} \ | \ \mathrm{min}(\lambda_f(x) \ | \ f \in N) = \lambda_{h_1}(x)\}$$
    of $\alpha_N$ on the boundary of $\mathbb R_+^{\dbl N}$. Note that on this boundary we have $\lambda_{h_1} = \ldots = \lambda_{h_\ell}$.

    Indeed, we have $\lambda_f \leq \lambda_g$ for $f \geq g$ in $N$ if and only if $p_f + \sum_{f > h > g} (p_h - m_h) - m_g \geq 0$. It follows that the constraints $F_g\geq 0$ in the definition of the causal cone are equivalent to the constraints $\lambda_{h_i} \leq \lambda_g$ for $g\leq h_i$. Thus, the intersection of the causal cone with $\mathbb R_+^{\dbl N}$ is precisely the region in $V(p_{h_1},\ldots,p_{h_\ell}) \cap \mathbb R_+^{\dbl N}$ where $\lambda_{h_1} \leq \lambda_g$ for all $g \in N$.
\end{proof}

As any building set contains the minimal one, and containment of building sets corresponds to refinement, we see that all lightcone subdivisions induce subdivisions of $U_N$.

Let us determine the rays of maximal cones in the minimal lightcone subdivision. As a corollary, we will obtain the rays of $U_N$.

\begin{prop}\label{prop:cones-of-min-lightsub}
    The maximal cones of the minimal lightcone subdivision are indexed by the elements of $N$. The cone associated to $f \in N$ is spanned by the vectors $w_s$ from \eqref{eq:ws-def} for all connected subgraphs $s$ of $\overline T_N$ that do not contain the edge $f$.
\end{prop}
\begin{proof}
    A maximal nested set in $\mathcal B_N$ is of the form $M:= N \backslash \{f\}$ for some $f \in N$. According to \Cref{prop:lightcone-subdivision}, there is one cone in the lightcone subdivision of $\mathbb R_+^{\dbl N}$ for each such set $M$. Its rays are the vectors $w_s$ for subgraphs $s$ of $\overline T_N$ spanned by $M$. These are precisely the connected subgraphs $s$ of $\overline T_N$ that do not contain the edge $f$.
\end{proof}

\begin{cor}\label{cor:rays-of-cosmo}
    The rays of $U_N$ are the vectors $w_s$ for connected subgraphs $s$ of the Hasse diagram $F_N$ of $N$.
\end{cor}
\begin{proof}
    We claim that the vectors $w_s$ for connected subgraphs of $\overline T_N$ that do not contain any maximal element of $N$ are precisely those that lie in the causal cone, which implies the statement by \Cref{prop:cosmo-as-subcone} and \Cref{prop:cones-of-min-lightsub}. This is true because there is an edge entering $s$ in $\overline T_N$ if and only if $s$ does not contain the edge $g$ for a maximal element $g \in N$.
\end{proof}

\begin{ex}\label{ex:star-min-lightsub}
    Recall the star graph from \Cref{ex:star-wavefunction}. A maximal nested set $N$ is a chain $a \subset b \subset \hat 1$. The tree $\overline T_N$ is simply the linear graph
    \[\begin{tikzpicture}
        \node (1) at (0,0) {$\bullet$};
        \node (2) at (1,0) {$\bullet$};
        \node (3) at (2,0) {$\bullet$};
        \node (4) at (3,0) {$\bullet$};

        \draw (1) -- (2) node [midway, above] {$\hat 1$};
        \draw (2) -- (3) node [midway, above] {$b$};
        \draw (3) -- (4) node [midway, above] {$a$};
    \end{tikzpicture}\]
    The associated minimal lightcone subdivision of $\mathbb R_+^{\dbl N}$ has $3$ cones, given by
    \begin{align*}
        C_{\hat 1} &= \langle \hat 1^-, \hat 1^+, \hat 1^+ + b^-, \hat 1^+ + a^-, b^+, b^+ + a^-, a^+ + b^-, a^+\rangle \\
        C_{b} &= \langle \hat 1^-, b^-, \hat 1^+ + b^-, b^+, b^+ + a^-, a^+ \rangle \\
        C_{a} &= \langle \hat 1^-, b^-, a^-, \hat 1^+ + b^-, \hat 1^+ + a^-, b^+ + a^-, a^+ \rangle
    \end{align*}
    Intersecting with the causal cone, we see that $U_N$ is indeed a face of $C_{\hat 1}$:
    $$U_N = \langle \hat 1^+, \hat 1^+ + b^-, \hat 1^+ + a^-, b^+, b^+ + a^-, a^+ + b^-, a^+\rangle$$
\end{ex}

The second interesting special choice of $\mathcal B_N$ is the \textit{graphical building set}, given by all tubes (that is, connected induced subgraphs) in $\overline T_N$.

\begin{prop}\label{prop:max-lightsub}\leavevmode
    \begin{enumerate}
        \item The lightcone refinement associated to the graphical building set is a unimodular triangulation of $\mathbb R_+^{\dbl N}$. Its simplices are given by
        $$\cone(w_s \ | \ s \in \mathcal T)$$
        for tubings $\mathcal T$ of $\overline T_N$ that do not use the maximal tube.
        
        \item The induced triangulation of $U_N$ is given by the simplices $\cone(w_s \ | \ s \in \mathcal T)$ where $\mathcal T$ is a tubing of $F_N \subseteq \overline T_N$, i.e., where no tube in $\mathcal T$ contains the root vertex.
    \end{enumerate}
\end{prop}
\begin{proof}
    We show $i)$, which implies $ii)$ by \Cref{cor:rays-of-cosmo}.

    Recall from \Cref{prop:lightcone-subdivision} that if $M$ is a maximal nested set in $\mathcal B_N$, then the associated maximal cone in the lightcone subdivision of $\mathbb R^{\dbl N}$ is spanned by the vectors $w_s$ for $s\not= \overline T_N$ spanned by $M$. But since $\mathcal B_N$ contains precisely the connected subgraphs, $s$ is spanned by $M$ if and only if $s$ is a vertex or $s \in M$. The nested sets for the graphical building set $\mathcal B_N$ are precisely the vertex-free tubings of $\overline T_N$, and so the subgraphs spanned by $M$ are precisely the elements of a maximal tubing $\mathcal T$. The cardinality of $\mathcal T - \overline T_N$ is $|E(\overline T_N)| + |V(\overline T_N)| - 1 = 2|N|$. Thus, $\cone(w_s \ | \ s \in \mathcal T)$ is a simplex.

    To show unimodularity, we use that the isomorphism $\psi$ from \eqref{eq:cut-iso} is defined over $\mathbb Z$, so it suffices to show unimodularity of the cone $\cone(\psi(w_s) \ | \ s \in \mathcal T)$. By definition, we have $\psi_{w_s} = \sum_{e \in s} e + \sum_{v \in s} v$. Since a maximal tubing $\mathcal T$ contains all vertices, it suffices to show that $\cone(\sum_{e \in s} e \ | \ s \in \mathcal T)$ is unimodular. But this is a cone of the nested set fan $\Sigma_{\mathcal B_N}$ which is unimodular by \cite[Proposition 2]{feichtner2004chow}.
\end{proof}

\begin{rmk}
Due to \Cref{prop:cosmological-polytopes} the triangulation from \Cref{prop:max-lightsub} corresponds to a triangulation of the dual of the cosmological polytope. Indeed, this is precisely the triangulation from \cite{ArkaniHamed17} corresponding to old fashioned perturbation theory.
\end{rmk}

\begin{ex}\label{ex:star-max-lightsub}
    Continuing with \Cref{ex:star-min-lightsub}, let us compute the associated maximal lightcone subdivision. The graph
        \[\begin{tikzpicture}
        \node (1) at (0,0) {$\bullet$};
        \node (2) at (1,0) {$\bullet$};
        \node (3) at (2,0) {$\bullet$};
        \node (4) at (3,0) {$\bullet$};

        \draw (1) -- (2) node [midway, above] {$\hat 1$};
        \draw (2) -- (3) node [midway, above] {$b$};
        \draw (3) -- (4) node [midway, above] {$a$};
    \end{tikzpicture}\]
    has $5$ maximal tubings. The associated simplices are
    \begin{align*}
        D_1 &= \langle a^-, b^-, \hat 1^-, \hat 1^+ + b^-, b^+ + a^-, a^+\rangle \\
        D_2 &= \langle a^-, \hat 1^+ + a^-, \hat 1^-, \hat 1^+ + b^-, b^+ + a^-, a^+\rangle \\
        D_3 &= \langle \hat 1^+, \hat 1^+ + a^-, \hat 1^-, \hat 1^+ + b^-, b^+ + a^-, a^+\rangle \\
        D_4 &= \langle \hat 1^+, b^+, \hat 1^-, \hat 1^+ + b^-, b^+ + a^-, a^+\rangle\\
        D_5 &= \langle b^-, b^+, \hat 1^-, \hat 1^+ + b^-, b^+ + a^-, a^+\rangle
    \end{align*}
    Comparing with \Cref{ex:star-min-lightsub} we see that $C_{\hat 1} = D_3 \cup D_4$, $C_{b} = D_5$ and $C_{a} = D_1 \cup D_2$. In particular we obtain the triangulation
    $$U_N = \langle \hat 1^+, \hat 1^+ + a^-, \hat 1^+ + b^-, b^+ + a^-, a^+ \rangle \cup \langle \hat 1^+, b^+, \hat 1^+ + b^-, b^+ + a^-, a^+\rangle.$$
\end{ex}

%%%%%%%%%%%%%%%%%%%%%%
\subsection{Lightcone refinements}

A crucial question is when the lightcone subdivisions of simplices $\mathbb R_+^{\dbl N}$ defined in the previous section are compatible in the doubled nested set fan $\Sigma^{\dbl \mathcal G}$ and thus the universe fan $\mathcal U_{\mathcal G}$. The next result gives a sufficient condition.

\begin{lem}\label{lem:lightcone-refinement}
    For every nested set, choose a boolean building set $\mathcal B_N$ on $N$ such that
    \begin{enumerate}
        \item each $t \in \mathcal B_N$ is a connected subgraph of $\overline T_N$, and
        \item for $N' \subseteq N$ we have $\mathcal B_{N'} = \{t \cap N' \ | \ t \in \mathcal B_N\}$.
    \end{enumerate}
    Then the lightcone subdivisions of simplices $\mathbb R^{\dbl N}$ associated to the collection $(\mathcal B_N)_{N}$ assemble into a fan in $\mathbb R^{\dbl \mathcal G}$.
\end{lem}
\begin{proof}
   By \Cref{prop:lightcone-subdivision}, the lightcone subdivision on each cone $\mathbb R_+^{\dbl N}$ is given by the domains of linearity of
    $$\alpha_N := \sum_{t \in \mathcal B_N} \min(\lambda_e \ | \ e \in t).$$
    Let $N$ and $M$ be two nested sets. Then $\mathbb R_+^{\dbl N}$ and $\mathbb R_+^{\dbl M}$ intersect in $\mathbb R_+^{\dbl (N\cap M)}$. Since $N\cap M$ is again nested, it suffices to show that for $M \subseteq N$ the subdivision of $\mathbb R_+^{\dbl M}$ is the restriction of the subdivision of $\mathbb R_+^{\dbl N}$. Wlog.\ we can assume that $M = N \backslash \{f\}$ for some $f\in N$. We claim that $\alpha_N$ and $\alpha_M$ agree on $p_f = m_f = 0$. Note that the linear forms $\lambda_e$ restrict correctly for $e \not=f$. Moreover we have the identities $\lambda_f = \lambda_a + a^+ - f^- = \lambda_b + b^- - f^+$ for the edge $a$ above $f$ or an edge $b$ below $f$. By assumption $i)$, at least one of these is contained in every non-singleton $t \in \mathcal B_N$ that contains $f$. It follows that on the positive orthant $\lambda_f \geq \lambda_a, \lambda_b$ if $p_f=m_f = 0$. Thus, on $\mathbb R_+^{\dbl M}$ the function $\alpha_N$ simplifies to
    $$\sum_{t \in \mathcal B_N} \min(\lambda_e \ | \ e \in t\cap M)$$
    and by assumption $ii)$ this has the same domains of linearity as $\alpha_M$.
\end{proof}

We call the resulting fan the \textit{lightcone refinement} of $\Sigma^{\dbl \mathcal G}$ associated to the collection $(\mathcal B_N)_N$.

The two special cases of building sets $\mathcal B_N$ discussed in the previous subsection are both examples of collections satisfying \Cref{lem:lightcone-refinement}; in fact they are the minimal and maximal examples. Let us study the corresponding refinements.

We call the refinement obtained by choosing for each nested set the \textit{minimal building set} $\mathcal B_N := \{N\} \cup \{f \ | \ f\in N\}$ on $\overline T_N$ the \textit{minimal lightcone refinement}.

\begin{thm}\label{prop:universe-as-subfan}
    The universe fan $\mathcal U_{\mathcal G}$ is a subfan of the minimal lightcone refinement of $\Sigma^{\dbl \mathcal G}$.
\end{thm}
\begin{proof}
    This follows from \Cref{prop:cosmo-as-subcone}.
\end{proof}

In particular, there is a subfan of every lightcone refinement that refines $\mathcal U_{\mathcal G}$.

Next, let us describe the rays of the minimal lightcone refinement. In particular, this yields the generators of cones in the universe fan.

\begin{dfn}
    A \textit{lightcone region} in $\mathcal G$ is a pair $(R^+,R^-)$ of disjoint subsets of $\mathcal G$ with $|R^+|\leq 1$ such that $R:=R^+ \cup R^-$ is nested, $R^+ \subseteq \max R$ and $R^- \subseteq \min R$. We say that $(R^+,R^-)$ \textit{contains} $f \in \mathcal G$ if for $a \in R^+$ we have $f < a$ (strictly) and for $b \in R^-$ we have $f \not\leq b$. We define
$$w_{(R^+,R^-)} := \sum_{f \in R^+} f^+ + \sum_{f\in R^-} f^-.$$
\end{dfn}

Note that if $|R^+| = 1$ then $R^+ \cup R^-$ is a causal region.

\begin{cor}\label{cor:rays-of-universe}\leavevmode
    \begin{enumerate}
        \item The maximal cones of the minimal lightcone refinement of $\Sigma^{\mathcal G}$ are indexed by pairs $(N,f)$ of maximal nested sets $N$ and $f\in N$. The cone associated to $(N,f)$ is spanned by the vectors $w_{(R^+,R^-)}$ for all lightcone regions $(R^+,R^-)$ with $R^+, R^- \subseteq N$ that do \emph{not} contain $f$.
        \item The maximal cones of the minimal lightcone refinement of $\mathcal U_{\mathcal G}$ are indexed by maximal nested sets $N$. The cone associated to $N$ is spanned by the vectors $w_R$ for causal regions $R \subseteq N$.
    \end{enumerate}
\end{cor}
\begin{proof}
    For a fixed nested set $N$, the cones in the subdivision of $\mathbb R_+^{\dbl N}$ correspond to $f\in N$, see \Cref{prop:cones-of-min-lightsub}. More precisely, the associated cone is spanned by the $w_s$ defined in \eqref{eq:ws-def} for connected subgraphs $s$ of $\overline T_N$ that do not contain the edge $f$. The lightcone regions $(R^+,R^-)$ with $R^+,R^- \subseteq N$ that do not contain $f$ are precisely the sets of edges that enter resp.\ leave such a connected subgraph $s$. This implies $i)$ and by \Cref{cor:rays-of-cosmo} also $ii)$.
\end{proof}

\begin{ex}
    Take again the star graph from \Cref{ex:star-wavefunction}. If the edges are labeled $1,2,3$, then the poset $\mathcal G$ is
     \[ \begin{tikzpicture}[scale=0.65,node distance=20mm]
  	\node (o) at (2.5,-1.2) {$\emptyset$};
	\node (a) at (0,0)   {1};
  	\node (b) at (2.5,0) {2};
  	\node (c) at (5,0)   {3};
  	\node (d) at (1.25,1.8)   {12};
  	\node (e) at (3.75,1.8) {23};
  	\node (f) at (6.25,1.8)   {13};
  	\node (t) at (2.5,3.4) {$123$};
	  \draw (d) -- (a);
	  \draw (d) -- (b);
	  \draw (e) -- (b);
	  \draw (e) -- (c);
	  \draw (f) -- (c);
	  \draw[black!40] (f) -- (a);
	  \draw (t) -- (d);
 	  \draw (t) -- (e);
  	  \draw (t) -- (f);
	  \draw (o) -- (a);
	  \draw (o) -- (b);
	  \draw (o) -- (c);
	\end{tikzpicture}
  \]
  Let us determine the cones indexed by $(N,12)$ for maximal nested sets $N$ with $12 \in N$. These are the two chains $\{1,12,\hat 1\}$ and $\{2,12,\hat 1\}$. The associated lightcone regions that do not contain $12$ are
  $$\{(\emptyset,\{\hat 1\}), (\{\hat 1\}, \{12\}),(\emptyset,\{12\}),(\{12\},\emptyset),(\{12\},\{1\}),(\{1\},\emptyset)\}$$
  and
  $$\{(\emptyset,\{\hat 1\}), (\{\hat 1\}, \{12\}),(\emptyset,\{12\}),(\{12\},\emptyset),(\{12\},\{2\}),(\{2\},\emptyset)\}$$
  They have the lightcone regions $\{(\emptyset,\{\hat 1\}), (\{\hat 1\}, \{12\}),(\emptyset,\{12\}),(\{12\},\emptyset)\}$ in common. In particular, the associated maximal cones in the refinement intersect in
  $$\langle \hat 1^-,\hat 1^+ + 12^-, 12^-, 12^+\rangle.$$
\end{ex}

Our second special choice of $(\mathcal B_N)_N$ is the set of graphical building sets on the trees $\overline T_N$. This is the maximal choice of building sets $\mathcal B_N$ satisfying the assumptions of \Cref{lem:lightcone-refinement}, and we refer to it as the \textit{maximal lightcone refinement}.
\vspace{1em}
\begin{prop}\label{prop:max-lightcone-refine}\leavevmode
    \begin{enumerate}
        \item The maximal lightcone refinement of $\Sigma^{\dbl \mathcal G}$ is a unimodular simplicial fan. Its simplices are given by
        $$\cone(w_t \ | \ t \in \mathcal T)$$
        where $\mathcal T$ is a tubing of $\overline T_N$ for a nested set $N$ such that $\mathcal T$ does not contain $\overline T_N$ and crosses each edge of $\overline T_N$ at least once. Here, $w_t$ denotes again the vectors from \eqref{eq:ws-def}.
        
        \item The simplices of the induced triangulation of $\mathcal U_{\mathcal G}$ correspond to those $\mathcal T$ in $i)$ that define tubings of the Hasse diagram $F_N$, i.e., that do not contain the root of $\overline T_N$.
    \end{enumerate}
\end{prop}
\begin{proof}
    Both $i)$ and $ii)$ follow from \Cref{prop:max-lightsub}. If $C$ is a simplex in the refinement and $N$ minimal with $C \subseteq \mathbb R_+^{\dbl N}$ then $C$ corresponds to a tubing $\mathcal T$ of $\overline T_N$, and since $N$ is minimal each edge of $\overline T_N$ is crossed at least once by $\mathcal T$.
\end{proof}

We will also refer to this triangulation as the \textit{tubing refinement} of the universe fan.

\begin{ex}
    Let us compute two simplices in the maximal lightcone refinement of $\Sigma^{\dbl N}$ and their intersection, again for the star graph from \Cref{ex:star-wavefunction}. Consider the tubings (A) and (B) in \Cref{fig:tubing-ex}. They correspond to the simplices
    $$C_1 = \langle \hat 1^+,\hat 1^+ + 2^-, \hat 1^-, \hat 1^+ + 12^-, 12^+ + 2^-, 2^+ \rangle$$
    and
    $$C_2 = \langle \hat 1^+,\hat 1^+ + 2^-, \hat 1^-, \hat 1^+ + 23^-, 23^+ + 2^-, 2^+ \rangle$$
    which intersect in
    $$C_3 = \langle \hat 1^+, \hat 1^-, \hat 1^+ + 2^-, 2^+\rangle,$$
    the simplex corresponding to the third tubing (C) in \Cref{fig:tubing-ex}. Accordingly, the faces
    $$\langle \hat 1^+,\hat 1^+ + 2^-, \hat 1^+ + 12^-, 12^+ + 2^-, 2^+ \rangle \qquad \text{and} \qquad \langle \hat 1^+,\hat 1^+ + 2^-, \hat 1^+ + 23^-, 23^+ + 2^-, 2^+ \rangle$$ of $C_1$ and $C_2$ are maximal cones of the universe fan, intersecting in $\langle \hat 1^+, \hat 1^+ + 2^-, 2^+\rangle$.

    \begin{figure}
        \centering
        \scalebox{0.9}{
        \begin{subfigure}{0.25\linewidth}
            \centering
            \definecolor{c1a1a1a}{RGB}{26,26,26}
\definecolor{cff3737}{RGB}{255,100,100}
\definecolor{c333333}{RGB}{51,51,51}

\def \globalscale {1.000000}
\begin{tikzpicture}[y=1pt, x=1pt, yscale=\globalscale,xscale=\globalscale, every node/.append style={scale=\globalscale}, inner sep=0pt, outer sep=0pt]
\path[draw=c1a1a1a,fill=cff3737,fill opacity=0.4,line width=1.5pt] (67.9, 
  182.7)arc(-0.0:90.0:14.7 and -14.8)arc(90.0:180.0:14.7 and 
  -14.8)arc(180.0:270.0:14.7 and -14.8)arc(270.0:360.0:14.7 and -14.8) -- cycle;

\path[draw=c1a1a1a,fill=cff3737,fill opacity=0.4,line width=1.7pt] (80.9, 
  108.3).. controls (80.9, 87.2) and (76.6, 57.3) .. (55.8, 54.9).. controls 
  (30.7, 52.1) and (2.6, 83.7) .. (2.6, 109.1).. controls (2.6, 134.5) and 
  (30.7, 166.2) .. (55.8, 163.3).. controls (76.8, 160.9) and (81.0, 129.7) .. 
  (80.9, 108.3) -- cycle;

\path[draw=c1a1a1a,fill=cff3737,fill opacity=0.4,line width=1.5pt] (41.0, 
  153.2).. controls (32.2, 147.2) and (16.9, 135.3) .. (10.5, 124.5).. controls 
  (5.4, 116.0) and (5.2, 101.8) .. (12.6, 95.0).. controls (20.3, 88.0) and 
  (30.7, 89.4) .. (39.9, 94.4).. controls (49.6, 99.6) and (65.9, 116.1) .. 
  (70.6, 125.9).. controls (74.6, 134.5) and (73.8, 147.2) .. (66.8, 153.7).. 
  controls (59.5, 160.6) and (49.4, 158.9) .. (41.0, 153.2) -- cycle;

\path[draw=c1a1a1a,fill=cff3737,fill opacity=0.4,line width=1.5pt] (38.8, 
  109.4)arc(0.0:90.0:14.7 and -14.8)arc(90.0:180.0:14.7 and 
  -14.8)arc(180.0:270.0:14.7 and -14.8)arc(-90.0:0.0:14.7 and -14.8) -- cycle;

\path[draw=c1a1a1a,fill=cff3737,fill opacity=0.4,line width=1.5pt] (67.9, 
  138.2)arc(-0.0:90.0:14.7 and -14.8)arc(90.0:180.0:14.7 and 
  -14.8)arc(180.0:270.0:14.7 and -14.8)arc(-90.0:0.0:14.7 and -14.8) -- cycle;

\path[draw=c1a1a1a,fill=cff3737,fill opacity=0.4,line width=1.5pt] (67.9, 
  79.4)arc(-0.0:90.0:14.7 and -14.8)arc(90.0:180.0:14.7 and 
  -14.8)arc(180.0:270.0:14.7 and -14.8)arc(270.0:360.0:14.7 and -14.8) -- cycle;

\path[fill,line width=1.2pt] (58.4, 182.7)arc(0.0:90.0:5.2 and 
  -5.2)arc(90.0:180.0:5.2 and -5.2)arc(180.0:270.0:5.2 and 
  -5.2)arc(270.0:360.0:5.2 and -5.2) -- cycle;

\path[fill,line width=1.2pt] (58.4, 138.3)arc(0.0:90.0:5.2 and 
  -5.2)arc(90.0:180.0:5.2 and -5.2)arc(180.0:270.0:5.2 and 
  -5.2)arc(270.0:360.0:5.2 and -5.2) -- cycle;

\path[draw=black,fill opacity=0.4,line width=1.5pt] (53.2, 182.7) -- (53.2, 
  138.2);

\path[fill,line width=1.2pt] (29.2, 109.4)arc(0.0:90.0:5.2 and 
  -5.2)arc(90.0:180.0:5.2 and -5.2)arc(180.0:270.0:5.2 and 
  -5.2)arc(270.0:360.0:5.2 and -5.2) -- cycle;

\path[draw=black,fill=c333333,line width=1.5pt] (53.2, 138.2) -- (23.8, 108.9);

\path[fill,line width=1.2pt] (58.4, 79.4)arc(-0.0:90.0:5.2 and 
  -5.2)arc(90.0:180.0:5.2 and -5.2)arc(180.0:270.0:5.2 and 
  -5.2)arc(270.0:360.0:5.2 and -5.2) -- cycle;

\path[draw=black,fill=c333333,line width=1.5pt] (53.2, 79.6) -- (23.8, 108.9);

\node[text=black,line width=0.8pt,scale=1.0,anchor=south west] (text1) at (66.5,
   162.0){$\hat 1$};

\node[text=black,line width=0.8pt,scale=1.0,anchor=south west] (text1-3) at 
  (42.8, 112.9){$12$};

\node[text=black,line width=0.8pt,scale=1.0,anchor=south west] (text1-3-8) at 
  (25.7, 78.8){$2$};

\end{tikzpicture}
            \vspace{-20pt}
            \caption{}
        \end{subfigure}
        \hspace{20pt}
        \begin{subfigure}{0.25\linewidth}
            \centering
            \definecolor{c1a1a1a}{RGB}{26,26,26}
\definecolor{cff3737}{RGB}{255,100,100}
\definecolor{c333333}{RGB}{51,51,51}

\def \globalscale {1.000000}
\begin{tikzpicture}[y=1pt, x=1pt, yscale=\globalscale,xscale=\globalscale, every node/.append style={scale=\globalscale}, inner sep=0pt, outer sep=0pt]
\path[draw=c1a1a1a,fill=cff3737,fill opacity=0.4,line width=1.5pt] (44.1, 
  182.7)arc(180.0:90.0:14.7 and -14.8)arc(90.0:0.0:14.7 and 
  -14.8)arc(360.0:270.0:14.7 and -14.8)arc(270.0:180.0:14.7 and -14.8) -- cycle;

\path[draw=c1a1a1a,fill=cff3737,fill opacity=0.4,line width=1.7pt] (31.1, 
  108.3).. controls (31.2, 87.2) and (35.4, 57.3) .. (56.3, 54.9).. controls 
  (81.3, 52.1) and (109.4, 83.7) .. (109.4, 109.1).. controls (109.4, 134.5) and
   (81.3, 166.2) .. (56.3, 163.3).. controls (35.2, 160.9) and (31.1, 129.7) .. 
  (31.1, 108.3) -- cycle;

\path[draw=c1a1a1a,fill=cff3737,fill opacity=0.4,line width=1.5pt] (71.0, 
  153.2).. controls (79.8, 147.2) and (95.2, 135.3) .. (101.5, 124.5).. controls
   (106.6, 116.0) and (106.8, 101.8) .. (99.4, 95.0).. controls (91.7, 88.0) and
   (81.3, 89.4) .. (72.1, 94.4).. controls (62.4, 99.6) and (46.1, 116.1) .. 
  (41.5, 125.9).. controls (37.4, 134.5) and (38.3, 147.2) .. (45.2, 153.7).. 
  controls (52.5, 160.6) and (62.6, 158.9) .. (71.0, 153.2) -- cycle;

\path[draw=c1a1a1a,fill=cff3737,fill opacity=0.4,line width=1.5pt] (73.2, 
  109.4)arc(180.0:90.0:14.7 and -14.8)arc(90.0:0.0:14.7 and 
  -14.8)arc(0.0:-90.0:14.7 and -14.8)arc(270.0:180.0:14.7 and -14.8) -- cycle;

\path[draw=c1a1a1a,fill=cff3737,fill opacity=0.4,line width=1.5pt] (44.1, 
  138.2)arc(180.0:90.0:14.7 and -14.8)arc(90.0:0.0:14.7 and 
  -14.8)arc(360.0:270.0:14.7 and -14.8)arc(270.0:180.0:14.7 and -14.8) -- cycle;

\path[draw=c1a1a1a,fill=cff3737,fill opacity=0.4,line width=1.5pt] (44.1, 
  79.4)arc(180.0:90.0:14.7 and -14.8)arc(90.0:-0.0:14.7 and 
  -14.8)arc(360.0:270.0:14.7 and -14.8)arc(270.0:180.0:14.7 and -14.8) -- cycle;

\path[fill,line width=1.2pt] (53.6, 182.7)arc(180.0:90.0:5.2 and 
  -5.2)arc(90.0:0.0:5.2 and -5.2)arc(360.0:270.0:5.2 and 
  -5.2)arc(270.0:180.0:5.2 and -5.2) -- cycle;

\path[fill,line width=1.2pt] (53.6, 138.3)arc(180.0:90.0:5.2 and 
  -5.2)arc(90.0:0.0:5.2 and -5.2)arc(360.0:270.0:5.2 and 
  -5.2)arc(270.0:180.0:5.2 and -5.2) -- cycle;

\path[draw=black,fill opacity=0.4,line width=1.5pt] (58.8, 182.7) -- (58.8, 
  138.2);

\path[fill,line width=1.2pt] (82.8, 109.4)arc(180.0:90.0:5.2 and 
  -5.2)arc(90.0:0.0:5.2 and -5.2)arc(360.0:270.0:5.2 and 
  -5.2)arc(270.0:180.0:5.2 and -5.2) -- cycle;

\path[draw=black,fill=c333333,line width=1.5pt] (58.8, 138.2) -- (88.2, 108.9);

\path[fill,line width=1.2pt] (53.6, 79.4)arc(180.0:90.0:5.2 and 
  -5.2)arc(90.0:-0.0:5.2 and -5.2)arc(360.0:270.0:5.2 and 
  -5.2)arc(270.0:180.0:5.2 and -5.2) -- cycle;

\path[draw=black,fill=c333333,line width=1.5pt] (58.8, 79.6) -- (88.2, 108.9);

\node[text=black,line width=0.8pt,scale=1.0,anchor=south west] (text1) at (76.5,
   163.4){$\hat 1$};

\node[text=black,line width=0.8pt,scale=1.0,anchor=south west] (text1-3) at 
  (56.3, 112.9){$23$};

\node[text=black,line width=0.8pt,scale=1.0,anchor=south west] (text1-3-8) at 
  (75.8, 79.8){$2$};

\end{tikzpicture}
            \vspace{-20pt}
            \caption{}
        \end{subfigure}
        \hspace{20pt}
        \begin{subfigure}{0.25\linewidth}
            \centering
            \definecolor{c1a1a1a}{RGB}{26,26,26}
\definecolor{cff3737}{RGB}{255,100,100}

\def \globalscale {1}
\begin{tikzpicture}[y=1pt, x=1pt, yscale=\globalscale,xscale=\globalscale, every node/.append style={scale=\globalscale}, inner sep=0pt, outer sep=0pt]
\path[draw=c1a1a1a,fill=cff3737,fill opacity=0.4,line width=1.5pt] (15.6, 
  182.7)arc(180.0:90.0:14.7 and -14.8)arc(90.0:0.0:14.7 and 
  -14.8)arc(360.0:270.0:14.7 and -14.8)arc(270.0:180.0:14.7 and -14.8) -- cycle;

\path[draw=c1a1a1a,fill=cff3737,fill opacity=0.4,line width=1.5pt] (10.6, 
  69.5).. controls (14.1, 62.9) and (21.8, 57.2) .. (29.2, 57.1).. controls 
  (37.1, 56.9) and (45.5, 62.6) .. (49.4, 69.5).. controls (62.0, 92.1) and 
  (60.6, 123.8) .. (49.4, 147.1).. controls (46.1, 153.9) and (39.3, 160.9) .. 
  (31.8, 161.2).. controls (23.3, 161.6) and (14.6, 154.6) .. (10.6, 147.1).. 
  controls (-1.6, 124.3) and (-1.6, 92.3) .. (10.6, 69.5) -- cycle;

\path[draw=c1a1a1a,fill=cff3737,fill opacity=0.4,line width=1.5pt] (15.6, 
  138.2)arc(180.0:90.0:14.7 and -14.8)arc(90.0:0.0:14.7 and 
  -14.8)arc(360.0:270.0:14.7 and -14.8)arc(270.0:180.0:14.7 and -14.8) -- cycle;

\path[draw=c1a1a1a,fill=cff3737,fill opacity=0.4,line width=1.5pt] (15.6, 
  79.4)arc(180.0:90.0:14.7 and -14.8)arc(90.0:-0.0:14.7 and 
  -14.8)arc(360.0:270.0:14.7 and -14.8)arc(270.0:180.0:14.7 and -14.8) -- cycle;

\path[fill,line width=1.2pt] (25.2, 182.7)arc(180.0:90.0:5.2 and 
  -5.2)arc(90.0:0.0:5.2 and -5.2)arc(360.0:270.0:5.2 and 
  -5.2)arc(270.0:180.0:5.2 and -5.2) -- cycle;

\path[fill,line width=1.2pt] (25.2, 138.3)arc(180.0:90.0:5.2 and 
  -5.2)arc(90.0:0.0:5.2 and -5.2)arc(360.0:270.0:5.2 and 
  -5.2)arc(270.0:180.0:5.2 and -5.2) -- cycle;

\path[draw=black,fill opacity=0.4,line width=1.5pt] (30.3, 182.7) -- (30.3, 
  138.2);

\path[fill,line width=1.2pt] (25.2, 79.4)arc(180.0:90.0:5.2 and 
  -5.2)arc(90.0:-0.0:5.2 and -5.2)arc(360.0:270.0:5.2 and 
  -5.2)arc(270.0:180.0:5.2 and -5.2) -- cycle;

\path[draw=c1a1a1a,fill=black,line width=1.5pt] (30.0, 138.3) -- (30.3, 79.4) --
   (30.3, 79.4) -- (30.3, 79.4);

\node[text=black,line width=1.5pt,anchor=south west] (text4) at (45.9, 
  161.2){$\hat 1$};

\node[text=black,line width=1.5pt,anchor=south west] (text4-0) at (35.2, 
  105.5){$2$};

\path[line width=0.9pt] (17.1, 52) rectangle (41.8, 51.1);

\end{tikzpicture}
            \vspace{-20pt}
            \caption{}
        \end{subfigure}
        }
        \caption{Tubings of $\overline T_N$ for different $N$.}\label{fig:tubing-ex}
    \end{figure}

\end{ex}

%%%%%%%%%%%%%%%%%%%%%%

%%%%%%%%%%%%%%%%%%%%%%
%%%%%%%%%%%%%%%%%%%%%%
%%%%%%%%%%%%%%%%%%%%%%
%%%%%%%%%%%%%%%%%%%%%%
\section{Generalized Cosmohedra from nested set fans}\label{sec:cosmohedra}

It was observed in \cite{ArkFigVaz24} that there is a `blowup' of the associahedron, the \textit{Cosmohedron}, that has the same combinatorics as the `old-fashioned perturbation theory' of the cosmological wavefunction. Dually, this `blowup' corresponds to a refinement of normal fans.

In this section, we explain how this and other refinements are obtained from the lightcone refinements of the universe fan. In fact, our construction works for a general atomic lattice $L$ and building set $\mathcal G$. We denote the set of atoms by $E$.

As in \cref{sec:refinements}, our strategy is to consider first the doubled free nested set fan $\Sigma^{\dbl \mathcal G}$ and then restrict to the universe fan. Let $\Sigma_{\mathcal G}$ denote the nested set fan that we recalled in \Cref{dfn:nested-set-fan}. We will see that refinements of $\Sigma^{\dbl \mathcal G}$ induce refinements of $\Sigma_{\mathcal G}$, while subfans of these refinements are induced by restriction to $\mathcal U_{\mathcal G}$, which follows from \Cref{prop:universe-as-subfan}. Of particular interest is the link $\overline \Sigma_{\mathcal G} := \Sigma_{\mathcal G}/\langle u_{\mathrm{tot}} \rangle$ (where again $u_{\mathrm{tot}} = \sum_{e\in E} e$) and its induced refinements. If $L$ is a boolean lattice, $\overline \Sigma_{\mathcal G}$ is the normal fan of a polytope (a \textit{generalized permutahedron}, see \cite{Pos09}).

The heart of the construction is the following diagram:
\[\begin{tikzcd}
	& {\mathbb R^{\mathcal G}} \\
	{\mathbb R^{\dbl {\mathcal G}}} && {\mathbb R^E}
	\arrow["p",from=1-2, to=2-3]
	\arrow[swap,"\Delta_0",from=1-2, to=2-1]
    \end{tikzcd}\]
Here, $\Delta_0$ is the embedding sending $\hat 1$ to $\hat 1^+$ and $\hat 1 \not= g$ to $g^+ + g^-$. The map $p$ is the projection $\mathbb R^{\mathcal G} \to \mathbb R^E, \ f \mapsto \sum_{e\in f} e$ from \Cref{prop:nested-projection}. Let us write $\overline p:\mathbb R^{\mathcal G} \to \mathbb R^E/\langle u_{\mathrm{tot}} \rangle$ for its composition with the quotient map $\mathbb R^E \to \mathbb R^E/\langle u_{\mathrm{tot}}\rangle$.

\begin{lem}\label{lem:induced-refinement}
    We have $p(\Delta_0^{-1}(\Sigma^{\dbl \mathcal G})) = \Sigma_{\mathcal G}$ and $\overline p(\Delta_0^{-1}(\mathcal U_{\mathcal G})) = \overline \Sigma_{\mathcal G} := \Sigma_{\mathcal G}/\langle u_{\mathrm{tot}} \rangle$. In particular, refinements of $\Sigma^{\dbl \mathcal G}$ resp.\ $\mathcal U_{\mathcal G}$ induce refinements of $\Sigma_{\mathcal G}$ resp.\ $\overline \Sigma_{\mathcal G}$.
\end{lem}
\begin{proof}
    The first statement holds by definition and \Cref{prop:nested-projection}. For the second assertion, recall that $\mathcal U_{\mathcal G}$ is by definition the intersection of $\Sigma^{\dbl \mathcal G}$ with the causal cone (\Cref{dfn:causal-cone}). Under $\Delta_0$, the causal cone pulls back to the cone cut out in $\mathbb R^{\mathcal G}$ by $\hat 1^* \geq g^*$ for all $g<\hat 1$. Thus, the projection of $\Delta^{-1}(\mathcal U_{\mathcal G})$ to $\mathbb R^{\mathcal G}/\langle \hat 1 \rangle$ agrees with $\Sigma^{\mathcal G}/\langle \hat 1 \rangle$. Applying $p$, this shows the statement. Finally, the third statement follows since the projection $\mathbb R^{\mathcal G} \to \mathbb R^E$ is injective on $\Sigma^{\mathcal G}$ by \Cref{prop:nested-projection}.
\end{proof}

\begin{dfn}
    A \textit{lightcone refinement} of the nested set fan $\Sigma_{\mathcal G}$ is the fan $p(\Delta_0^{-1}\Gamma)$ for a lightcone refinement $\Gamma$ of $\Sigma^{\dbl \mathcal G}$ (see \Cref{lem:lightcone-refinement}).
\end{dfn}

Note that the induced refinements of $\overline \Sigma_{\mathcal G}$ agree with $\overline p(\Delta_0^{-1}\Gamma')$ where $\Gamma'$ is the subfan of $\Gamma$ refining $\mathcal U_{\mathcal G}$, cf.\ \Cref{prop:universe-as-subfan}.

Let us study the lightcone refinements of nested set fans. As a first observation, let us show that they are `good' lower-dimensional representations of the lightcone refinements of $\Sigma^{\dbl \mathcal G}$ and $\mathcal U_{\mathcal G}$.

\begin{prop}\label{prop:coolbijection}
    Let $\Gamma$ denote a lightcone refinement of $\Sigma^{\dbl \mathcal G}$. Then $p\circ \Delta_0^{-1}$ induces a bijection between the maximal cones in $\Gamma$ and the maximal cones in $p(\Delta_0^{-1}\Gamma)$. The same is true if $\Gamma$ is a lightcone refinement of $\mathcal U_{\mathcal G}$.
\end{prop}
\begin{proof}
    It suffices to show this for the maximal lightcone refinement, since it refines every other one. By \Cref{prop:max-lightcone-refine}, a maximal cone of $\Gamma$ is spanned by $w_t$ for $t \in \mathcal T$ where $\mathcal T$ is a maximal tubing of the tree $\overline T_N$ associated to a maximal nested set $N$. It suffices to construct a vector in the interior of the cone that lies on the diagonal $\Delta_0$. For this, we inductively define subsets of $\mathcal T$ as follows: let $\mathcal T_0$ denote the (inclusion-wise) minimal elements of $\mathcal T$ and let $\mathcal T'_0$ be those elements of $\mathcal T_0$ that are contained in another tube in $\mathcal T$. For $i>0$ let $\mathcal T_i$ be the minimal elements in $\mathcal T \backslash \bigcup_{j<i} \mathcal T'_j$ and let $\mathcal T'_i$ be those elements of $\mathcal T_i$ that are contained in another tube in $\mathcal T$. Note that there is some $i_0$ with $\mathcal T'_{i_0} = \emptyset$.

    Since $\mathcal T$ is maximal, we have $\bigcup_{t \in \mathcal T_i} V(t) = V(\overline T_N)$ for all $i$ and by definition the tubes in $\mathcal T_i$ are disjoint. Thus, $\sum_{t \in \mathcal T_i} w_{t_i} = \sum_{f} f^+ + f^-$ where the second sum runs over those $f$ that enter or leave a tube in $\mathcal T_i$. In particular, $\sum_{i \leq i_0} \sum_{t \in \mathcal T_i} w_{t_i}$ is a vector in $\Delta_0$, and since $\bigcup_{i\leq i_0} \mathcal T_i = \mathcal T$ it is also a vector in the interior of the cone.

    Finally, we see that this construction restricts to the maximal cones of the universe fan: if $\mathcal T$ is a maximal tubing without a tube containing the root vertex $*$ of $\overline T_N$, define the sets $\mathcal T_i$ in the same way as above; now $\bigcup_{t \in \mathcal T_i} V(t) = V(\overline T_N) \backslash \{*\}$, implying $\sum_{t \in \mathcal T_i} w_{t_i} = k\cdot \hat 1^+ + \sum_{f} f^+ + f^-$ for some $k \in \mathbb N$, where now the second sum runs over $f\not=\hat 1$ entering or leaving a tube in $\mathcal T_i$. Thus, we conclude.
\end{proof}

Let us now study the lightcone refinements of $\Sigma_{\mathcal G}$ in more detail. The following lemma will be helpful.

\begin{lem}\label{lem:nest-local-inv}
    For every $N \in \mathcal N(L,\mathcal G)$ there is a linear map $s_N: \mathbb R^E \to \mathbb R^N$ with $s_N \circ p|_{\mathbb R^N} = \mathrm{id}_{\mathbb R^N}$.
\end{lem}
\begin{proof}
    For a nested set $N$ we construct the map $s_N: \mathbb R^E \to \mathbb R^N$ as follows: for every $f \in N$, let $M_f := \max {g \in N \ | \ g < f}$ and choose $h_f \in f \backslash \bigcup_{g \in M_f} g$. Since $\bigcup_{g \in M} g \subseteq \bigvee_{g \in M} g \notin \mathcal G$, this is always possible. Now we define the map $s_N: \mathbb R^E \to \mathbb R^N$ by mapping $h_f$ to $f$ and all other $e$ to $0$.

    For $f \in N$, consider $s_N\circ p(f) = s_N(\sum_{e\in f} e)$. By definition, this agrees with
    $$\sum_{h_g \in f} s(g) = \sum_{h_g \in f} \left(g - \sum_{g' \in M_g} g'\right).$$
    This is a telescoping sum which simplifies to $f$: since for $a,b \in N$ we either have $a\leq b$ or $a\land b = \hat 0$, the $g\in N$ with $h_g \in f$ are precisely the $g$ with $g \leq f$. The Hasse diagram of the poset of such $g$ is a tree, in which the children of $g$ are precisely the set $M_g$. This proves the claim.
\end{proof}

\begin{prop}\label{cor:local-nest-refinement}
    Fix a nested set $N$ and a building set $\mathcal B_N$ on $\overline T_N$. Let $\Gamma$ denote the lightcone subdivision of $\mathbb R^{\dbl N}_+$ associated to $\mathcal B_N$ (see \Cref{dfn:lightcone-subdivision}).
    \begin{enumerate}
        \item The subdivision $\Delta_0^{-1} \Gamma$ of $\mathbb R^N_+$ is given by the domains of linearity of the function
        $$\sum_{t \in \mathcal B_N} \max_{e \in t}(e^*).$$
        In other words, it is the fan $\Sigma_{\mathcal B_N}^{\mathrm{max}}$ from \Cref{sec:boole}.
        \item Let $C_N := p(\mathbb R^N_+)$ and consider the map $s_N$ from \Cref{lem:nest-local-inv}. The lightcone subdivision of $C_N$ associated to $\mathcal B_N$ is given by the domains of linearity of
        $$\sum_{t \in \mathcal B_N} \max_{e \in t}(e^*\circ s_N)$$
        on $C_N$.
    \end{enumerate}
\end{prop}
\begin{proof}
    The second assertion follows from $i)$, \Cref{prop:nested-projection} and the definition. For $i)$, note that on the diagonal we have $p_f = m_f =: f^*$. Thus, $\lambda_e$ from \Cref{prop:lightcone-subdivision} becomes $\sum_{f \in N \backslash \{e\}} f^*$. Subtracting $\sum_{f \in N} f^*$ everywhere does not change the domains of linearity, and thus we can replace $\min_{e \in t}(\lambda_e)$ in \Cref{prop:lightcone-subdivision} by $\min_{e \in t}(-e^*) = -\max_{e \in t}(e^*)$.
\end{proof}

Note that while the map $s_N$ is not unique, its restriction to $C_N$ is: it is the inverse of the isomorphism $p|_{\mathbb R^N_+}:\mathbb R^N_+ \to C_N$ (see \Cref{prop:nested-projection}).

\begin{ex}\label{ex:123-bool-lc-refine}
    Consider the boolean lattice $L$ on the set $\{1,2,3\}$ and the graphical building set $\mathcal G$ with edges $12$ and $23$. The five maximal nested sets are
    $$\{123,12,1\},\{123,12,2\},\{123,23,2\},\{123,23,3\},\{123,1,3\}.$$
    The associated nested set fan is a subdivision of the positive orthant $\mathbb R^{\{1,2,3\}}_+$, shown in \Cref{fig:nested-set-fan}.
    Let us describe the maximal lightcone refinement of the nested set fan. We start with the cone associated to $\{123,12,1\}$, spanned by the vectors $1+2+3,\ 1+2$ and $1$. On this cone, the projection $p$ is inverted by mapping $1 \mapsto 1, 2\mapsto 12-1$ and $3\mapsto 123-12$. Dually, we see that $1^* \mapsto 1^* - 2^*, 12^* \mapsto 2^* - 3^*$ and $123^* \mapsto 3^*$. By \Cref{cor:local-nest-refinement}, the subdivision on our cone is induced by the domains of linearity of the function
    $$\max(123^*,12^*,1^*) + \max(123^*,12^*) + \max(12^*, 1^*)$$
    on $\mathbb R_+^N$. Thus, it is given by the domains of linearity of
    $$\max(3^*,2^*-3^*,1^* - 2^*) + \max(3^*,2^*-3^*) + \max(2^*-3^*,1^* - 2^*)$$
    on the cone $\langle 1,1+2,1+2+3\rangle$. Similarly, we obtain the subdivisions of the other cones associated to chains. For $\{123,1,3\}$, we obtain the function
    $$\max(2^*,1^*-2^*,3^*-2^*) + \max(2^*,1^*-2^*) + \max(2^*,3^*-2^*).$$
    Putting everything together, we obtain the subdivision shown in \Cref{fig:max-lightcone-sub}.
    \begin{figure}
        \begin{subfigure}[t]{0.25\linewidth}
            \includegraphics[width=\linewidth]{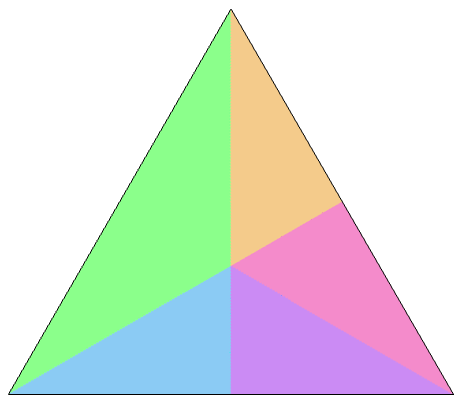}
            \subcaption{The nested set fan.}\label{fig:nested-set-fan}
        \end{subfigure}
        \hspace*{20pt}
        \begin{subfigure}[t]{0.25\linewidth}
            \includegraphics[width=\linewidth]{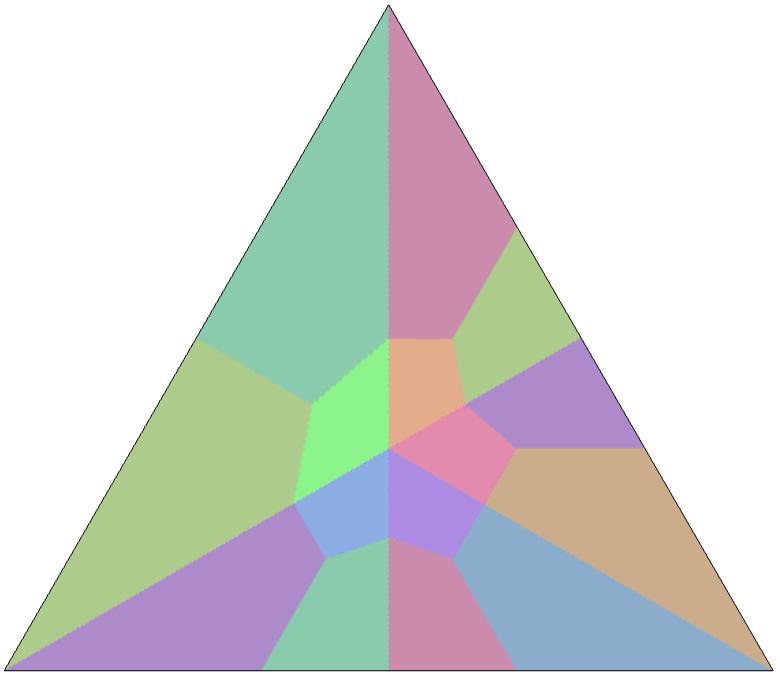}
            \subcaption{The minimal lightcone refinement.}\label{fig:min-lightcone-sub}
        \end{subfigure}
        \hspace*{20pt}
        \begin{subfigure}[t]{0.25\linewidth}
            \includegraphics[width=\linewidth]{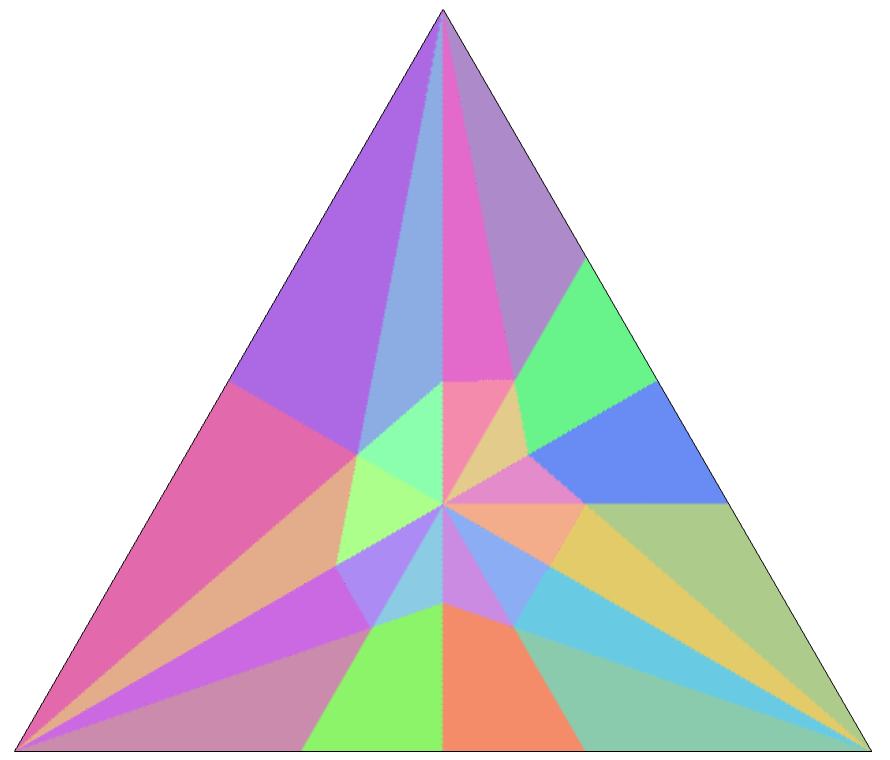}
            \subcaption{The maximal lightcone refinement.}\label{fig:max-lightcone-sub}
        \end{subfigure}
        \caption{The nested set fan and its subdivisions from \Cref{ex:123-bool-lc-refine}. All fans are contained in the positive orthant of $\mathbb R^3$ and we draw them projectively. Colors correspond to the gradient of the defining piecewise linear functions.}
    \end{figure}

\Cref{fig:min-lightcone-sub} shows the minimal lightcone refinement of $\Sigma_{\mathcal G}$. The five inner cones form the fan $p(\Delta_0^{-1}(\mathcal U_{\mathcal G}))$, visualizing \Cref{prop:universe-as-subfan}. Also note that indeed $\overline p(\Delta_0^{-1}(\mathcal U_{\mathcal G})) = \overline \Sigma_{\mathcal G}$.

\end{ex}

As already mentioned above, in the case of the boolean lattice, $\overline \Sigma_{\mathcal G}$ is a complete fan and in fact the normal fan of a generalized permutahedron \cite{Pos09}. The lightcone subdivisions are candidates for normal fans of blowups of these polytopes. They have the combinatorics of normal fans of \textit{acyclonesto-cosmohedra}, defined in \cite{ForGleKim25}.

\begin{ex}
    Continuing with \Cref{ex:123-bool-lc-refine}, let us consider the fan $\overline \Sigma_{\mathcal G}$. It is the normal fan of a pentagon, the $4$-associahedron.
    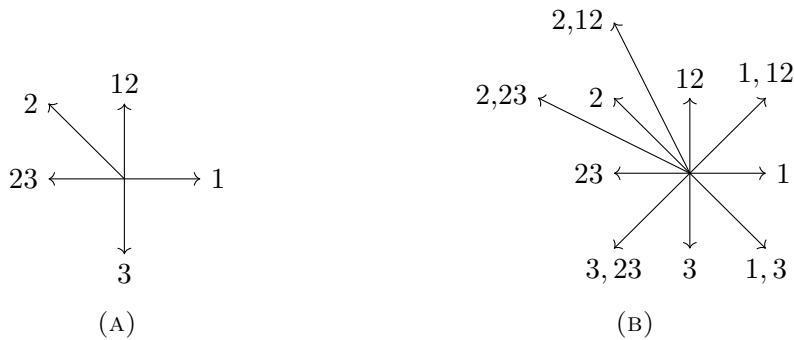
\begin{figure}[t]
        \centering
        \begin{subfigure}{0.3\linewidth}
            \centering
            \begin{tikzpicture}[->]
                \draw (0,0) -- (1,0) node[right] {$1$};
                \draw (0,0) -- (0,-1) node[below] {$3$};
                \draw (0,0) -- (0,1) node[above] {$12$};
                \draw (0,0) -- (-1,0) node[left] {$23$};
                \draw (0,0) -- (-1,1) node[above,left] {$2$};
            \end{tikzpicture}
            \caption{}
        \end{subfigure}
        \hspace{50pt}
        \begin{subfigure}{0.3\linewidth}
            \centering
            \begin{tikzpicture}[->]
                \draw (0,0) -- (1,0) node[right] {$1$};
                \draw (0,0) -- (0,-1) node[below] {$3$};
                \draw (0,0) -- (1,-1) node[right,below] {$1,3$};
                \draw (0,0) -- (-1,-1) node[left,below] {$3,23$};
                \draw (0,0) -- (0,1) node[above] {12};
                \draw (0,0) -- (1,1) node[right,above] {$1,12$};
                \draw (0,0) -- (-1,0) node[left] {23};
                \draw (0,0) -- (-1,1) node[above,left] {2};
                \draw (0,0) -- (-2,1) node[left] {2,23};
                \draw (0,0) -- (-1,2) node[left] {2,12};
            \end{tikzpicture}
            \caption{}
        \end{subfigure}
        \caption{The fan $\overline \Sigma_{\mathcal G}$ for $\mathcal G$ from \Cref{ex:123-bool-lc-refine} and its maximal lightcone refinement.}
    \end{figure}
    In the quotient, the maximal lightcone refinement that we computed in \Cref{ex:123-bool-lc-refine} just becomes a barycentric subdivision of each maximal cone of $\overline \Sigma_{\mathcal G}$. This is the normal fan of a decagon, the $4$-cosmohedron.

    The situation becomes more interesting in higher dimensions. If we take $L$ to be the boolean lattice on $\{1,2,3,4\}$ and the set of all intervals as $\mathcal G$, the nested set fan from \Cref{ex:123-bool-lc-refine}, shown in \Cref{fig:nested-set-fan}, becomes a subfan of the complete fan $\overline \Sigma_{\mathcal G}$ in $\mathbb R^{\{1,2,3\}} \cong \mathbb R^E / \langle u_{\mathrm{tot}}\rangle$. The whole fan $\overline \Sigma_{\mathcal G}$ is normal to the $5$-associahedron. The subdivision shown in \Cref{fig:max-lightcone-sub} extends to the maximal lightcone subdivision of $\overline \Sigma_{\mathcal G}$ which is normal to the $5$-cosmohedron. Note that this subdivision is \textit{not} simplicial, which corresponds to the fact that cosmohedra are not simple polytopes.
\end{ex}

%%%%%%%%%%%%%%%%%%%%%%
%%%%%%%%%%%%%%%%%%%%%%
%%%%%%%%%%%%%%%%%%%%%%
%%%%%%%%%%%%%%%%%%%%%%
\section{The Physical Case}\label{sec:physics}
Our results specialize to the case of the tree-level wavefunction of the universe, in physics. This corresponds to taking the building set $\mc{G}$ of (non-singleton) intervals in $[n]=\{1,2,\dots,n\}$, which we identify with chords of an $(n+1)$-gon. In this case, the wavefunction $\Psi(L,\mc{G})$ becomes the tree-level cosmological wavefunction $\Psi_{n+1}$, and the total-energy residue gives the tree-level colored scalar amplitude $A_{n+1}$.

The cosmological wavefunction $\Psi$ is a key step in the computation of cosmological correlators, which control the statistics of density fluctuations in the early universe, which is important for linking theories of the beginning of the universe to observations (e.g. \cite{maldacena03}). In \cite{ArkaniHamed17}, it was found that contributions to $\Psi$ are given by the canonical forms (or dually volumes) of \emph{cosmological polytopes}, for each graph. As emphasized there, this raises the question of how to understand the geometry of the full function: different graphs come with different variable sets, so it was not clear how the cosmological polytopes could be glued together. Our results in the present paper suggest that it is natural to view the (dual) cosmological polytopes as living in a common ambient space, $\RR^{2\mc{G}}$, where they are cut out by simple linear inequalities. In more recent work, \cite{ArkFigVaz24} introduced the \emph{cosmohedron}, a blow-up of the associahedron. The canonical form of the cosmohedron does not directly equal the wavefunction, but the wavefunction can be extracted from the canonical form by discarding certain terms. This raises the question of whether there is a convex geometry whose canonical form is precisely $\Psi_{n+1}$. Our paper does not give such a geometry, but does exhibit the cosmohedron as a shadow of the universe fan, whose Laplace transform is directly $\Psi_{n+1}$. Moreover, it is suggestive that the universe fan is cut out by simple linear functions, related to `causal diamond' pictures in the poset $\mc{G}$ (which is the \emph{kinematic mesh} of \cite{ArkFigVaz24}), which might offer clues for finding a stringy/worldsheet picture for cosmological observables.

For the convenience of readers interested in the physics application, we give in this section an extended overview of how the definitions and results of the paper apply to the physical case.

%%%%%%%%%%%%%%%%%%%%%%
\subsection{Polygons and energies}\label{sec:physics-intervals}
We take $\mc{G}_n$ to be the graphical building set of the path graph on $[n]$, i.e.
\[
\mc{G}_n \;:=\; \{\, I\subseteq[n] \mid I \text{ is a nonempty interval}, \ |I| \geq 2 \,\}.
\]
In particular, the maximal element in $\mc{G}_n$ is $\hat{1} = [n]$. $\mc{G}_n$ is a building set for a Boolean lattice, $L$: meet and join are just the intersection and union of sets.\footnote{We take $L$ to be the boolean lattice $2^E$ where $E$ is the set of edges of the path graph, i.e. $|E| = n-1$.} The intervals $I \in \mc{G}_n$ are naturally identified with chords of an $(n{+}1)$-gon, with $\hat{1}$ corresponding to a distinguished side. Label the sides cyclically by $1,2,\dots,n,\star$, where $\star$ denotes the distinguished side. A chord partitions $\{1,2,\dots,n,\star\}$ into two parts, and the part not containing $\star$ is an interval $I\subset[n]$ (with $|I| \geq 2$). This gives a bijection between chords (plus the side $\star$) and elements of $\mc{G}_n$. The remaining sides $1,\ldots, n$ are not elements of $\mc{G}_n$.

\smallskip
Two chords $I,J$ cross if and only if $I$ and $J$ overlap without containment, i.e.\ $I\cap J\neq\emptyset$ and neither $I\subseteq J$ nor $J\subseteq I$. So nested sets $N\subseteq \mc{G}_n$ are in bijection with collections of pairwise noncrossing chords, i.e.\ polyangulations of the $(n+1)$-gon, and maximal nested sets are triangulations. A causal region $R=\{I,J_1,\dots,J_k\}\subseteq N$ (as in \Cref{sec:universe}) has one maximal element $I$ and all other elements incomparable. In the chord model, such an $R$ determines a subpolygon of the $(n+1)\text{-gon}$ bounded by the chords $I, J_1,\ldots, J_k$ together with the sides $a$ of the big polygon such that $a \in I$ and $a \not \in J_i$. For example, the two chords $I = 1234$ and $J = 12$ bound an interior square with sides $12,3,4,1234$.

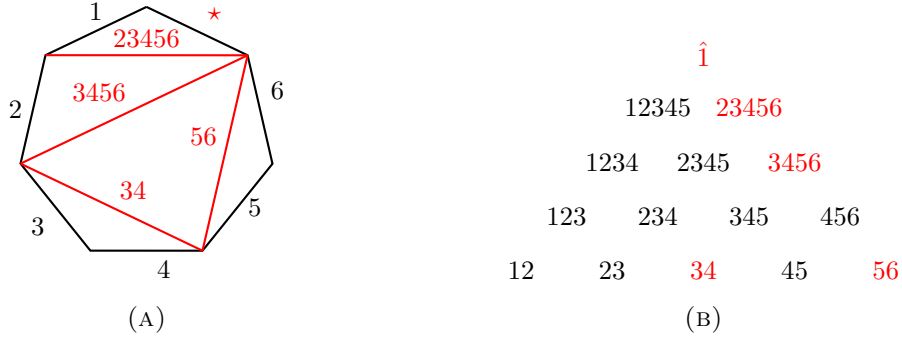
\begin{figure}
    \centering
      \begin{subfigure}{0.45\textwidth}
    \centering
\begin{tikzpicture}[scale=1.7, every node/.style={font=\small}]
  \foreach \i in {0,...,6}{
    \coordinate (P\i) at ({90+51.42*\i}:1);
  }
  \draw[thick]
    (P0) -- node[midway, above] {$1$} (P1) --
    node[midway, left] {$2$} (P2) --
    node[midway, below left] {$3$} (P3) --
    node[midway, below right] {$4$} (P4) --
    node[midway, right] {$5$} (P5) --
    node[midway, above right]{$6$} (P6) --
    node[red, midway, above right]{$\star$} (P0);
{\tiny
  \draw[thick, red] (P1) -- node[midway,above]{$23456$} (P6);
  \draw[thick, red] (P2) -- node[midway,above left]{$3456$} (P6);
  \draw[thick, red] (P2) -- node[midway,above right]{$34$} (P4);
  \draw[thick, red] (P4) -- node[midway,above left]{$56$} (P6);
}
\end{tikzpicture}
    \caption{}
    \label{fig:physics-polygon}
  \end{subfigure}
     \begin{subfigure}{0.45\textwidth}
    \centering
    \begin{tikzpicture}[scale=1.2,>=stealth]
  \def\h{0.6}
  % Row 1
  \coordinate (P00) at ( 0,        0        ); 
  % Row 2
  \coordinate (P10) at (-0.5,     -\h       );
  \coordinate (P11) at ( 0.5,     -\h       );
  % Row 3 
   \coordinate (P20) at (-1.0,     -2*\h     );
  \coordinate (P21) at ( 0.0,     -2*\h     );
  \coordinate (P22) at ( 1.0,     -2*\h     );
  % Row 4
  \coordinate (P30) at (-1.5,     -3*\h     );
  \coordinate (P31) at (-0.5,     -3*\h     );
  \coordinate (P32) at ( 0.5,     -3*\h     );
  \coordinate (P33) at ( 1.5,     -3*\h     );
  % Row 5
  \coordinate (P40) at (-2.0,     -4*\h     );
  \coordinate (P41) at (-1.0,     -4*\h     );
  \coordinate (P42) at ( 0.0,     -4*\h     );
  \coordinate (P43) at ( 1.0,     -4*\h     );
  \coordinate (P44) at ( 2.0,     -4*\h     );
  
{\small
\node[red] at (P00) {$\hat{1}$};
\node at (P10) {$12345$};
\node[red] at (P11) {$23456$};
\node at (P20) {$1234$};
\node at (P21) {$2345$};
\node[red] at (P22) {$3456$};
\node at (P30) {$123$};
\node at (P31) {$234$};
\node at (P32) {$345$};
\node at (P33) {$456$};
\node at (P40) {$12$};
\node at (P41) {$23$};
\node[red] at (P42) {$34$};
\node at (P43) {$45$};
\node[red] at (P44) {$56$};
}
\end{tikzpicture}
    \caption{}
    \label{fig:physics-mesh}
  \end{subfigure}
    \caption{(A) Chords in a heptagon labelled by subsets of the sides. (B) The set of all chords forms a poset $\mc{G}_6$, a building set for the Boolean lattice, sometimes called the `kinematic mesh'.}
    \label{fig:physics-causal}
\end{figure}

\smallskip
We now introduce the physical variables that appear in the wavefunction. For each side $i\in[n]$ introduce a physical energy variable $E_i$, and also introduce an energy $E_\star$ for the distinguished side $\star$. For each chord/interval $I\in \mc{G}_n$ introduce a variable $E_I$, and in particular we identify $E_{\hat{1}} \;=\; E_\star$. The wavefunction $\Psi(L,\mc{G}_n)$ defined in this paper is naturally written in the doubled chord variables $E_I^\pm$ (one pair for each chord $I\in\mc{G}_n$). There are $n(n-3)+1$ such variables (where we include $E_{\hat{1}}^+$). We pass to the physical variables by the substitution
\[
E_I^\pm \;\longmapsto\; E_I \pm \sum_{i\in I} E_i, \qquad I\in\mc{G}_n.
\]
For the top element $\hat{1}=[n]$ this gives $E_{\hat{1}}^+ \mapsto E_\star + E_1+E_2+\cdots+E_n$,
the \emph{total energy}, and we will write $E_{\hat{1}}^+ = E_{\mathrm{total}}$. Finally, the products $E_I^+E_I^-$ become Mandelstam variables under this substitution:
\[
E_I^+E_I^- \;\longmapsto\; E_I^2-\Bigl(\sum_{i\in I}E_i\Bigr)^2 = s_I,
\]
for each $I \neq \hat{1}$. Indeed, viewing the energy $E_i$ as the $0$-component of a 4-momentum vector $k_i = (E_i, \vec{k_i})$, and taking $E_I$ to be the norm of the 3-momentum $\sum_{i \in I} \vec{k_i}$, the Mandelstam variable is
\[
s_I = \big( \sum_{i \in I} k_i \big)^2 = - \big( \sum_{i\in I} E_i \big)^2 + \big\| \sum_{i \in I} \vec{k_i} \big\|^2.
\]
Identifying $E^+_I E^-_I$ with $s_I$ is important in the $E_{\mathrm{total}} = 0$ limit.

\smallskip
The substitution by physical energies also has a natural polygon interpretation. Recall that to each causal region $R=\{I, J_1,\dots,J_k\}$, we have an associated sum of variables
\[
E_R :=\ E_{I}^+ + \sum_{i=1}^k E_{J_i}^- .
\]
Substituting with physical energies, this $E_R$ becomes sum of energies around the perimeter of the corresponding subpolygon, which we identify with $R$:
\[
E_R \;\longmapsto\; \sum_{e\in \partial R} E_e.
\]
%Indeed, since $J_i \subset I$ (for each $i$), the substitution gives
%\[
%E_R \;\longmapsto\; E_I + \sum_{i=1}^k E_{J_i} + \sum_{\substack{a \in I\\ a \not\in \,\bigcup J_i}} E_a,
%\]
%which are precisely the sides of the subpolygon $R$.

\begin{rmk}
The same interval/chord building set discussed here can also be formulated as a building set for a lattice arising from the oriented matroid of the complete graph, as discussed for example in the final subsection of \cite{Lam24}.
\end{rmk}

\subsection{Wavefunction and factorization}\label{sec:physics-wavefunction}
We now recall the definition of the universe fan, $\mc{U}_{\mc{G}_n}$, but for this polygonal setting. The fan lives in $\RR_+^{\dbl \mc{G}_n}$, with two generators $I^\pm$ for every chord. Write a vector $w\in \RR_+^{\dbl\mc{G}_n}$ as
\[
w \;=\; p_{\hat{1}}\,\hat{1}^+ \;+\; \sum_{I\neq \hat{1}} p_I\, I^+ \;+\; \sum_{I\neq \hat{1}} m_I\, I^-, \qquad p_I,m_I\geq 0.
\]
Note that do not include a variable for $\hat{1}^-$ (or we restrict to the coordinate hyperplane $m_{\hat{1}} = 0$). For each chord $I\in \mc{G}_n$, consider the sum
\[
F_I \;=\; p_{\hat 1}+\sum_{I\subset J \subset \hat 1}(p_J-m_J)-m_I.
\]
The \emph{causal cone} (Definition~\ref{dfn:causal-cone}) is cut out by the positivity of the $F_I$:
\[
F_I\ge 0\qquad (I\in\mc{G}_n).
\]
The $F_I$ can be viewed as sums over \emph{causal diamonds} in the kinematic mesh: the intervals containing a chord $I$ form a diamond-shaped subposet (\Cref{fig:diamond}). A polyangulation $N$ restricts to a chain or ``causal path'' in each such diamond (\Cref{fig:diamondpath}). In particular, for each $I \in N$, the elements of $N$ strictly containing $I$ form a chain,
\[
I \subset J_1 \subset \ldots \subset J_k \subset \star.
\]
The restriction of $F_I$ to the $N$‑subspace ($p_I,m_I = 0$ for $I \not\in N$) is then a sum along this chain,
\[
\left.F_I\right|_N = - m_I + \sum_{i=1}^k (p_{J_i} - m_{J_i}) + p_\star.
\]
The fan $\mc{U}_{\mc{G}_n}$ is the intersection of the causal cone (given by $F_I \geq 0$) with union of all the $N$-subspaces. I.e. the maximal cones $U_N$ are indexed by the triangulations $N$ of the $(n+1)$-gon, and are the intersection with the coordinate hyperplane $p_I, m_I = 0$ for $I \not\in N$. Let us write $\Psi_{n+1} = \Psi(L, \mc{G}_n)$ for the wavefunction, given by the Laplace transform of $\mc{U}_{\mc{G}_n}$.

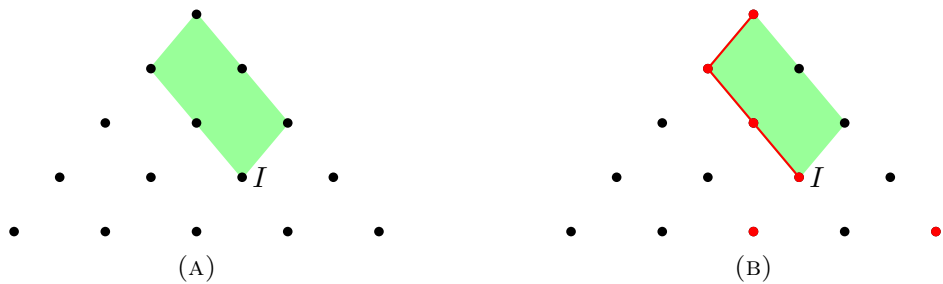
\begin{figure}
    \centering
\begin{subfigure}{0.45\textwidth}
\centering
\begin{tikzpicture}[scale=1.2,>=stealth]
  \def\h{0.6}
  % Row 1
  \coordinate (P00) at ( 0,        0        );
  % Row 2
  \coordinate (P10) at (-0.5,     -\h       );
  \coordinate (P11) at ( 0.5,     -\h       );
  % Row 3 
   \coordinate (P20) at (-1.0,     -2*\h     );
  \coordinate (P21) at ( 0.0,     -2*\h     );
  \coordinate (P22) at ( 1.0,     -2*\h     );
  % Row 4
  \coordinate (P30) at (-1.5,     -3*\h     );
  \coordinate (P31) at (-0.5,     -3*\h     );
  \coordinate (P32) at ( 0.5,     -3*\h     );
  \coordinate (P33) at ( 1.5,     -3*\h     );
  % Row 5
  \coordinate (P40) at (-2.0,     -4*\h     );
  \coordinate (P41) at (-1.0,     -4*\h     );
  \coordinate (P42) at ( 0.0,     -4*\h     );
  \coordinate (P43) at ( 1.0,     -4*\h     );
  \coordinate (P44) at ( 2.0,     -4*\h     );
  \filldraw[green, thick, opacity=0.4] (P00) -- (P10) -- (P32) -- (P22) -- (P00);
  \node at (P32)[right]{$I$};
  \foreach \p in {
    P00,
    P10,P11,
    P20,P21,P22,
    P30,P31,P32,P33,
    P40,P41,P42,P43,P44}
    \fill (\p) circle (1.5pt);
\end{tikzpicture}
\caption{}
\label{fig:diamond}
\end{subfigure}
\begin{subfigure}{0.45\textwidth}
\centering
\begin{tikzpicture}[scale=1.2,>=stealth]
  \def\h{0.6}
  % Row 1
  \coordinate (P00) at ( 0,        0        ); 
    % Row 2
  \coordinate (P10) at (-0.5,     -\h       );
  \coordinate (P11) at ( 0.5,     -\h       );
  % Row 3 
   \coordinate (P20) at (-1.0,     -2*\h     );
  \coordinate (P21) at ( 0.0,     -2*\h     );
  \coordinate (P22) at ( 1.0,     -2*\h     );
  % Row 4
  \coordinate (P30) at (-1.5,     -3*\h     );
  \coordinate (P31) at (-0.5,     -3*\h     );
  \coordinate (P32) at ( 0.5,     -3*\h     );
  \coordinate (P33) at ( 1.5,     -3*\h     );
  % Row 5
  \coordinate (P40) at (-2.0,     -4*\h     );
  \coordinate (P41) at (-1.0,     -4*\h     );
  \coordinate (P42) at ( 0.0,     -4*\h     );
  \coordinate (P43) at ( 1.0,     -4*\h     );
  \coordinate (P44) at ( 2.0,     -4*\h     );
  \filldraw[green, thick, opacity=0.4] (P00) -- (P10) -- (P32) -- (P22) -- (P00);
  \node at (P32)[right]{$I$};
  \foreach \p in {
    P00,
    P10,P11,
    P20,P21,P22,
    P30,P31,P32,P33,
    P40,P41,P42,P43,P44}
    \fill (\p) circle (1.5pt);
     \foreach \p in {
    P00,
    P10,
    P21,
    P32,
    P42,P44}
    \fill[red] (\p) circle (1.5pt);
      \draw[red, thick] (P00) -- (P10) -- (P21) -- (P32);
\end{tikzpicture}
\caption{}
\label{fig:diamondpath}
\end{subfigure}
    \caption{(A) The chords containing an interval $I$ form a causal diamond. (B) A triangulation/polyangulation $N$ (red) restricted to a causal diamond (green) forms a totally ordered chain, or causal path.}
    \label{fig:physics-diamond}
\end{figure}

\begin{ex}
Consider $n=3$. The kinematic mesh is $\mathcal{G}_4 = \{12, 23, \hat{1}\}$. The universe fan $\mc{U}_4$ lives in a five dimensional orthant $\RR_+^5$ generated by $12^+, 12^-, 23^+, 23^-, \hat{1}^+$, with coordinates
\[
p_{12}, m_{12}, p_{23}, m_{23}, p_{\hat{1}} \geq 0.
\]
The causal cone in $\RR_+^5$ is given by
\[
p_{\hat{1}} - m_{12} \geq 0,\qquad p_{\hat{1}} - m_{23} \geq 0,
\]
and it has 6 facets and 6 generators. The universe fan $\mc{U}_4$ then has two maximal cones in $\RR_+^5$, both of which are simplices:
\[
\{p_{\hat{1}} - m_{12} \geq 0, p_{23} = m_{23} = 0 \} \cup \{ p_{\hat{1}} - m_{23} \geq 0, p_{12} = m_{12} = 0 \}.
\]
The Laplace transform of these cones gives
\[
\Psi_4 = \frac{1}{E_{\hat{1}}^+} \frac{1}{E_{\hat{1}}^+ + E_{12}^-} \frac{1}{E_{12}^+} + \frac{1}{E_{\hat{1}}^+} \frac{1}{E_{\hat{1}}^+ + E_{23}^-} \frac{1}{E_{23}^+}.
\]
\end{ex}

\smallskip
A key property of $\Psi_{n+1}$ is that the total-energy residue recovers the \emph{tree-level amplitude}. This follows immediately from the causal path/diamond inequalities. Indeed, the total-energy variable $E_{\hat{1}}^+=E_{\mathrm{total}}$ is dual to the ray $\hat{1}^+$ in $\RR_+^{2\mc{G}_n}$. Taking the $E_{\mathrm{total}} = 0$ residue corresponds to projecting the link of $\hat{1}^+$ in the fan. 
Under this projection, the condition $F_I \geq 0$ is trivialized by taking $p_{\hat{1}}$ large. So the projection of a maximal cone $U_N$ is cut out by $p_I, m_I \geq 0$ (for $I \in N$, $I \neq \hat{1}$), which has Laplace transform
\[
\prod_{I \in N, I \neq \hat{1}} \frac{1}{E_I^+ E_I^-}. 
\]
Recalling the identification of $E_I^+ E_I^-$ with the Mandelstam variable $s_I$, we find (Corollary~\ref{cor:tot-energy-res})
\[
\Res_{E_{\mathrm{total}}=0}\Psi_{n+1} = \sum_{N\ \mathrm{triangulation}}
\ \prod_{I\in N}
\frac{1}{s_I} =: A_{n+1},
\]
where $A_{n+1}$ is the tree amplitude.

\smallskip
Similar arguments compute the other residues of $\Psi_{n+1}$ (\Cref{cor:wavefunction-res}). The poles of  $\Psi_{n+1}$ are $E_R=0$, for each subpolygon $R$. Its residues on these poles factorize into a product of two factors. In polygon pictures, one factor corresponds to the subpolygon $R$ itself, and the other to what remains after removing the interior of $R$ from the $(n+1)$-gon.

Let $R = \{ I, J_1, \ldots, J_m\}$, with $I$ the maximal chord. The chords of the subpolygon $R$ (together with its maximal side, $I$) forms a subset $\mc{G}^R_\text{in}$ of $\mc{G}_n$. We can identify $\mc{G}^{\dstream R}$ with $\mc{G}_m$, where $m+1$ is the number of sides of $R$. The chords that are not in the interior of $R$ form another subset $\mc{G}^R_\text{out}$ of $\mc{G}_n$. $\mc{G}^R_\text{out}$ is the set of chords in the polygons formed by removing $R$, together with $I$. Then
%\[
%\mc{G}^R_\text{in} = \mc{G}^{\dstream R} = \{ K \ | \ K \subseteq I, \text{ and } K \supset J_i \text{ or } K \cap J_i = \emptyset ~ \forall i \}.
%\]
%\[
%\mc{G}^R_\text{out} = \{I\} \cup \mc{G}^{\ustream R} = \{ K \ | \ K \supseteq I \text{ or } K\cap I = \emptyset \text{ or } K \subseteq J_i \}. 
%\]
\[
\frac{1}{E_R} \Res_{E_R=0} \Psi(\mathcal G_n) = A^{sh}( \mc{G}^R_\text{in} ) \cdot \Psi(\mc{G}^R_\text{out}).
\]
Here, $A^{\mathrm{sh}}$ is the $m+1$-point amplitude with variables labelled by the chords in $\mc{G}^R_\text{in}$, but with variables shifted by
\[
E_K^+ \mapsto E_K^+ + \sum_{J_i \subset K} E_{J_i}^-, \qquad E_K^- \mapsto E_K^- - \sum_{J_i \subset K} E_{J_i}^-.
\]
This makes sense from the point of view of the energy substitutions: the energy variables $E_{J_i}$ now become, under this shift, the energies of the sides $J_1,\ldots, J_m$ of the $R$ sub-polygon. $\Psi(\mc{G}^R_\text{out})$ is a wavefunction in the sense of this paper (associated to a lattice $L$ and building set $\mc{G}$), but it cannot be identified with $\Psi_{m+1}$ for some $m$. 

\begin{figure}
\centering
\begin{subfigure}[c]{0.48\textwidth}
  \centering
  \begin{subfigure}[c]{0.48\linewidth}
    \centering
 \begin{tikzpicture}[scale=1.2,>=stealth]
  \def\h{0.6}
  % Row 1
  \coordinate (P00) at ( 0,        0        );
  % Row 2
  \coordinate (P10) at (-0.5,     -\h       );
  \coordinate (P11) at ( 0.5,     -\h       );
  
  \node at (P00) {$3456$};
  \node at (P10) {$345$}; \node at (P10)[above=8pt, left=15pt]{$\mc{G}_{\text{in}}^R$:};
  \node at (P11) {$456$};
\end{tikzpicture}
  \end{subfigure}\hfill
  \begin{subfigure}[c]{0.48\linewidth}
    \centering
\begin{tikzpicture}[scale=1.6, every node/.style={font=\small}]
  \foreach \i in {0,...,6}{
    \coordinate (P\i) at ({90+51.42*\i}:1);
  }
  {\tiny
  \filldraw[opacity=0.1, green] (P2) -- (P6) -- (P5) -- (P3) -- (P2);
  \node[green!30!black] at (0,0)[below=1pt] {$R$};
  \draw[thick, red] (P2) -- node[midway,above=3pt]{$3456$} (P6);
  \draw[thick, blue] (P3) -- node[midway,below]{$45$} (P5);
  \draw[thick, black] (P2) -- (P3);
    \draw[thick, black] (P5) -- (P6);
 }
  \draw[gray, thick]
    (P0) -- node[midway, above] {$1$} (P1) --
    node[midway, left] {$2$} (P2) --
    node[black, midway, below left] {$3$} (P3) --
    node[midway, below right] {$4$} (P4) --
    node[midway, right] {$5$} (P5) --
    node[black, midway, above right]{$6$} (P6) --
    node[red, midway, above right]{$\star$} (P0);
\end{tikzpicture}
  \end{subfigure}
  \caption{}
\end{subfigure}\hfill
\begin{subfigure}[c]{0.48\textwidth}
  \centering
  \begin{subfigure}[c]{0.48\linewidth}
    \centering
    \begin{tikzpicture}[scale=1.6, every node/.style={font=\small}]
  \foreach \i in {0,...,6}{
    \coordinate (P\i) at ({90+51.42*\i}:1);
  }
    \filldraw[opacity=0.1,green] (P2) -- (P6) -- (P0) -- (P1) -- (P2);
    \filldraw[opacity=0.1,green] (P3) -- (P4) -- (P5) -- (P3);
  \draw[black, thick]
    (P0) -- node[midway, above] {$1$} (P1) --
    node[midway, left] {$2$} (P2) --
    node[gray, midway, below left] {$3$} (P3) --
    node[midway, below right] {$4$} (P4) --
    node[midway, right] {$5$} (P5) --
    node[gray, midway, above right]{$6$} (P6) --
    node[red, midway, above right]{$\star$} (P0);
      {\tiny
  \draw[thick, red] (P2) -- node[midway,below=3pt]{$3456$} (P6);
  \draw[thick, blue] (P3) -- node[midway,above]{$45$} (P5);
  \draw[thick, gray] (P2) -- (P3);
    \draw[thick, gray] (P5) -- (P6);
  }
\end{tikzpicture}
  \end{subfigure}\hfill
  \begin{subfigure}[c]{0.48\linewidth}
    \centering
\begin{tikzpicture}[scale=1,>=stealth]
  \def\h{0.6}
  % Row 1
  \coordinate (P00) at ( 0,        0        );
  % Row 2
  \coordinate (P10) at (-0.5,     -\h       );
  \coordinate (P11) at ( 0.5,     -\h       );
  % Row 3 
  \coordinate (P22) at ( 1.0,     -2*\h     );
  % Row 4
  \coordinate (P33) at ( 1.0,     -3*\h     );
 \node at (P00) {$\star$};
 \node at (P10) {$12$}; \node at (P10)[below=7pt, left=10pt]{$\mc{G}_{\text{out}}^R:$};
 \node at (P11) {$23456$};
 \node at (P22) {$3456$};
 \node at (P33) {$45$};
\end{tikzpicture}
  \end{subfigure}
  \caption{}
\end{subfigure}
\caption{The subpolygon $R$ bounded by chords $\{45,3456\}$ defines two subsets of the chords of the heptagon: the set of chords inside $R$ (left) and the set of chords not inside $R$ (right).}
    \label{fig:physics-fact}
\end{figure}

\begin{ex}[Factorization for $\Psi_5$]\label{ex:factorization}
Consider $n=4$. The wavefunction $\Psi_5$ has a pole at $E_R = E^+_{\hat{1}} + E_{12}^- = 0$, with $R = \{12, \hat{1}\}$ a square in the pentagon. The interior and exterior of $R$ correspond to the two new sets
\[
\mc{G}^R_\text{in} = \{123,34,\hat{1}\},\qquad \mc{G}^R_\text{out} = \{12,\hat{1}\}.
\]
The first term in the factorization formula, $A^{sh}( \mc{G}^R_\text{in} )$, is given by the amplitude
\[
A( \mc{G}^R_\text{in} ) =  \frac{1}{E_{123}^- E_{123}^+} + \frac{1}{E_{34}^- E_{34}^+},
\]
under the replacement $E_{123}^\pm \rightarrow E_{123}^\pm \pm E_{12}^-$. The second term is the wavefunction $\Psi(\mc{G}^R_\text{out})$. The set $\mc{G}^R_\text{out} = \{12,\hat{1}\}$ has just one nested set, namely $\{12,\hat{1}\}$, and so this wavefunction is given by the Laplace transform of the cone cut out by $\{p_{\hat{1}} - m_{12} \geq 0, p_{\hat{1}} \geq 0, p_{12} \geq 0\}$. This gives
\[
\Psi(\mc{G}^R_\text{out}) = \frac{1}{E_{\hat{1}}^+ E_{12}^+ (E_{\hat{1}}^+ + E_{12}^-)}.
\]
\end{ex}

\subsection{The Universe Fan and its refinements}\label{sec:physics-fan}
The generators of the universe fan $\mc{U}_{\mc{G}_n}$ are indexed by the subpolygons $R = \{I, J_1,\ldots, J_k\}$ of the $(n+1)$-gon and are given by the vectors (\Cref{prop:universe-rays})
\[
w_R = I^+ + \sum_{i=1}^k J_i^-.
\]
Moreover, the maximal cones $U_N$, for each triangulation $N$, are
\[
U_N \;=\; \cone\bigl(w_R \mid R\subseteq N \text{ a causal region}\bigr).
\]
In other words, $U_N$ is generated by all possible subpolygons that can be formed from the chords in the triangulation $N$. Note that $U_N$ is not a simplicial cone. In fact, it is the cone over the (dual) cosmological polytope of the dual graph of $N$.

\smallskip
However, Proposition~\ref{prop:max-lightcone-refine} gives a simplicial refinement of $\mc{U}_{\mc{G}_n}$, whose maximal cones are the simplices
\[
\cone(w_{R_1},\dots,w_{R_\ell}),
\]
indexed by complete nestings of subpolygons, $R_i$. Given this simplicial refinement, $\Psi_{n+1}$ decomposes term-by-term as a sum of Laplace transforms of simplicial cones, whose generators index complete nestings of sub-polygons. In other words, we recover the standard Russian doll formula \cite{ArkFigVaz24}
\[
\Psi_{n+1} \;=\;
\sum_{\substack{ N \\ \mathrm{triangulation} }} \sum_{\substack{ (R_1,\ldots,R_m)  \\ \mathrm{nested polygons} }} \ \prod_{i=1}^m \frac{1}{E_{R_i}},
\]
where for each triangulation $N$ we sum over all sets of complete nestings of sub-polygons $R_i \subseteq N$.

 \smallskip
The cones of the simplicial refinement are given by the linear domains of piecewise linear functions of the variables $(p_I, m_I)$. In fact, this is only one of a family of refinements of this type. For any polyangulation (i.e. nested set) $N$, and chord $I \in N$, let
\[
\lambda_{I} = \sum_{\substack{J\in N \\ I \subset J}} p_J + \sum_{\substack{J\in N \\ I \supset J \ | \ I \cap J = \emptyset }} m_J.
\]
Note that the two subsets appearing here are \emph{not} the subsets of chords that lie on either side of $I$. The subsets are natural to draw in the picture for $\mc{G}_n$, where they resemble lightcones. For any cone $U_N$, we can subdivide it by taking domains of linearity of
\[
\min(\lambda_{I} \ | \ I \in M)
\]
for some $M \subseteq N$, or mutual refinements for multiple subsets $M$. These refinements descend, after quotienting by the $\hat{1}^+$-direction and intersecting with the diagonal ($p_I = m_I$), to refinements of the associahedron fan, in particular recovering the cosmohedron fan of \cite{ArkFigVaz24} from the maximal lightcone refinement (\Cref{sec:cosmohedra}). We illustrate this in an example

\begin{ex}
Consider $n=4$, and fix a triangulation $N = \{12,34,\hat{1}\}$ of the pentagon. The cone $U_N$ is generated by all the subpolygons we can form in $N$:
\[
U_N = \langle \hat{1}^+, 12^+, 34^+, \hat{1}^+ + 12^- + 34^-, \hat{1}^+ + 12^-, \hat{1}^+ + 34^-\rangle.
\]
Equivalently, $U_N$ is cut out by the causal path conditions, $p_{\hat{1}} - m_{12} \geq 0$ and $p_{\hat{1}} - m_{34} \geq 0$.
In particular, note that $U_N$ is not a simplex. The functions $\lambda_I$ are
\[
\lambda_{12} = m_{34} + p_{\hat{1}},\qquad \lambda_{34} = m_{12} + p_{\hat{1}},\qquad \lambda_{\hat{1}} = m_{12} + m_{34}.
\]
We can subdivide $U_N$ by taking the domains of linearity of
\[
\alpha_N = \min(\lambda_{12}, \lambda_{34}, \lambda_{\hat{1}}) + \min(\lambda_{12}, \lambda_{34}) + \min(\lambda_{12}, \lambda_{\hat{1}}) + \min(\lambda_{34}, \lambda_{\hat{1}}) = 3 \lambda_{\hat{1}} + \min(\lambda_{12}, \lambda_{34}).
\]
This gives two simplices
\[
U_N = \langle \hat{1}^+, 12^+, 34^+, \hat{1}^+ + 12^- + 34^-, \hat{1}^+ + 12^-\rangle \cup \langle \hat{1}^+, 12^+, 34^+, \hat{1}^+ + 12^- + 34^-, \hat{1}^+ + 34^-\rangle
\]
Note that the generators of each of these simplices correspond to a complete nesting of subpolygons of the pentagon. Under the quotient by $\hat{1}^+$, these simplices project to (Figure \ref{fig:tetra})
\[
U_N / \langle \hat{1} \rangle = \langle 12^+, 34^+, 12^- + 34^-, 12^-\rangle \cup \langle 12^+, 34^+, 12^- + 34^-, 34^-\rangle
\]
The intersection of these simplices with the diagonal ($m_f = p_f$) gives two cones,
\[
\langle 12, 12+34\rangle,\qquad \langle 34,12+34\rangle,
\]
which is a refinement of the cone $\langle 12,34\rangle$ in the normal fan of the associahedron (\Cref{fig:deca}). Note that every maximal simplex in the refinement upstairs also becomes a maximal cone in the refinement of the associahedron: so the refinement captures the combinatorics of subpolygon nestings or Russian doll pictures.

\end{ex}

\begin{figure}
    \centering
\begin{subfigure}{0.45\textwidth}
\begin{tikzpicture}[scale=1.4]
  \coordinate (A) at (0,0);        % 12^+
  \coordinate (B) at (3,0.3);      % 12^-
  \coordinate (D) at (1.2,1.0);    % 123^-
  \coordinate (C) at (1.5,2.3);    % 123^+
  \draw (A) -- (B);
  \draw (A) -- (D);
  \draw (B) -- (D);
  \draw (A) -- (C);
  \draw (B) -- (C);
  \draw (D) -- (C);
  \fill (A) circle (2.2pt);
  \fill (B) circle (2.2pt);
  \fill (C) circle (2.2pt);
  \fill (D) circle (2.2pt);
  \node[below left=2pt]  at (A) {$12^{+}$};
  \node[below right=2pt] at (B) {$12^{-}$};
  \node[left=1pt]        at (D) {$34^{+}$};
  \node[above=3pt]       at (C) {$34^{-}$};
  \coordinate (M) at ($ (C)!0.5!(B) $);
  \fill[red] (M) circle (2.2pt);
  \node[red, above right=2pt] at (M) {$12^-+34^-$};
  \draw[red] (M) -- (A);
  \draw[red] (M) -- (D);
\end{tikzpicture}
    \caption{}
    \label{fig:tetra}
\end{subfigure}
\begin{subfigure}{0.45\textwidth}
\centering
 \begin{tikzpicture}[scale=0.8, >=stealth]
  \coordinate (O) at (0,0);
  \draw[thick,->] (O) -- (2,0) node[at end, right=2pt] {$12$};
  \draw[thick,->] (O) -- (0,2) node[at end, above=2pt] {$123$};
  \draw[thick,->] (O) -- (-1.6,1.6) node[at end, above left=2pt] {$23$};
  \draw[thick,->] (O) -- (-2,0) node[at end, left=2pt] {$234$};
  \draw[thick,->] (O) -- (0,-2) node[at end, below=2pt] {$34$};

  \draw[thick,red,->] (O) -- (1.6,1.6);
  \draw[thick,red,->] (O) -- (-1.0,2.0);
  \draw[thick,red,->] (O) -- (-2.0,1.0);
  \draw[thick,red,->] (O) -- (-1.6,-1.6);
  \draw[thick,red,->] (O) -- (1.6,-1.6) node[at end, below right=2pt]{$12+34$};
\end{tikzpicture}
\caption{}
\label{fig:deca}
\end{subfigure}
    \caption{The simplicial refinement of the universe fan induces a (non-simplicial) refinement of the associahedron fan, recovering the normal fan of the cosmohedron. (A) A cone $U_N$ in the $n=4$ universe fan, projected through $\hat{1}^+$, and its simplicial refinement (red). (B) Intersecting this cone on the diagonal ($p_I = m_I$) gives (after projecting further) a refinement of a cone in the associahedron fan.}
    \label{fig:physics-ref}
\end{figure}
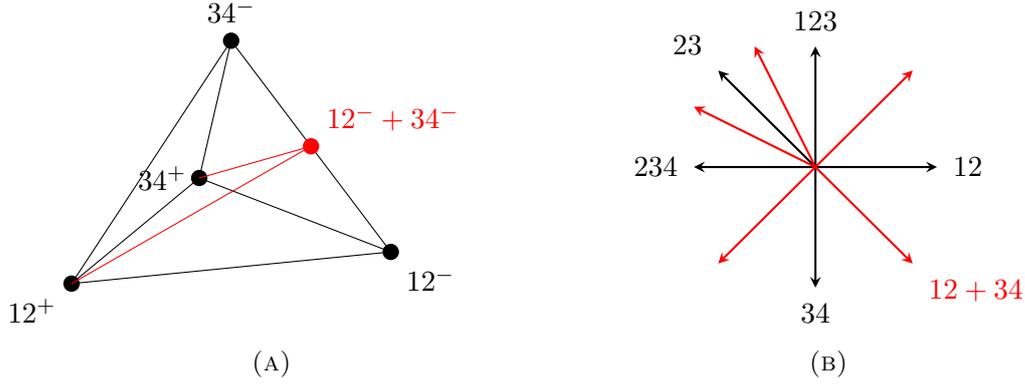

%%%%%%%%%%%%%%%%%%%%%%
%%%%%%%%%%%%%%%%%%%%%%
%%%%%%%%%%%%%%%%%%%%%%
%%%%%%%%%%%%%%%%%%%%%%
%%%%%%%%%%%%%%%%%%%%%%
\newpage
\appendix
%%%%%%%%%%%%%%%%%%%%%%
%%%%%%%%%%%%%%%%%%%%%%
%%%%%%%%%%%%%%%%%%%%%%
%%%%%%%%%%%%%%%%%%%%%%
%%%%%%%%%%%%%%%%%%%%%%
\section{Oriented Matroids and Matroid Amplitudes}\label{app:matroid}

We recall some fundamentals on oriented matroids and the construction of associated matroid amplitudes in \cite{Lam24}.
%%%%%%%%%%%%%%%%%%%%%%
\subsection{Oriented matroids}\label{sec:matroid:oriented}
\begin{dfn}
    An \textit{oriented matroid} is a pair $(E,\mathcal V)$ of a finite set $E$ and a collection $\mathcal V$ of maps $v: E \to \{-,0,+\}$, called vectors, such that
    \begin{enumerate}
        \item $\mathcal V$ contains $v \equiv 0$,
        \item if $v \in \mathcal V$ then $-v \in \mathcal V$,
        \item if $v, w \in \mathcal V$, then $v \circ w \in \mathcal V$, where
        $$(v \circ w)(e) := \begin{cases}
            v(e) & \text{ if $v(e) \in \{-,+\}$} \\
            w(e) & \text{ if $v(e) = 0$}
        \end{cases}$$
        \item if $v, w \in \mathcal V$ and $v(e) = - w(e)$ then there is some $z \in \mathcal V$ with $z(e) = 0$ and $z(a) = (v \circ w)(a)$ whenever $v(a) = w(a)$.
    \end{enumerate}

    The vectors of $M$ form a lattice $\mathcal L(M)$ by letting $v \leq w$ if $v(e) = w(e)$ whenever $w(e) \not= 0$. A \textit{tope} is a minimal vector in this lattice.

\end{dfn}

\begin{ex}
    The terminology is motivated by the case where $E$ is a set of hyperplanes $H_e$ in a vector space $V$, cut out by covectors $f_e \in V^\vee$ for $e \in E$. Every vector $v \in V$ defines a map $E \to \{-,0,+\}$ via $e \mapsto \mathrm{sgn} \ f_e(v)$. Then axiom $i)$ is just the condition that $0 \in V$ and $ii)$ says that $-v \in V$ for $v \in V$. $iii)$ and $iv)$ hold since $v + \lambda w \in V$ for $v \in V$, $w \in V$ and $\lambda \geq 0$, where for $iii)$ we choose $\lambda$ sufficiently small and for $iv)$ such that $f_e(v) + \lambda f_e(w) = 0$. \par
    The topes of this oriented matroid are the connected components $C$ of the hyperplane arrangement complement $V - \bigcup_e H_e$. Note that the closures $\bar{C}$ are polyhedra. The vectors of the matroid are precisely the faces of these polyhedra.
\end{ex}

\begin{ex}
    As a special case of the previous example, any directed graph $G = (V,E)$ defines an oriented matroid by taking the hyperplanes in $\mathbb R^V$ cut out by covectors $w^* - v^*$ for $(v \to w) \in E$. An element of $\mathbb R^V$ can be viewed as assigning potentials to the vertices of $G$, and then a vector of the associated oriented matroid records the signs of potential differences along the edges of $G$. It follows that the topes of the matroid correspond to acyclic orientations of the undirected graph underlying $G$ and more generally the vectors of the matroid are one-to-one with directed acyclic graphs on $V$ obtained via transitive reduction and contraction from such graphs.
\end{ex}

\begin{ex}
    Consider the graph
 \[
        \begin{tikzpicture}
            \node (1) at (0,1) {1};
            \node (2) at (0,0) {2};
            \node (3) at (1,0) {3};

            \draw[->] (1) -- (2);
            \draw[->] (2) -- (3);
            \draw[->] (2) -- (3);
            \draw[->] (3) -- (1);
        \end{tikzpicture}
\]
There are six acyclic orientations of the graph, corresponding to topes $(v(12),v(23),v(13))$:
$$(++-) \quad (+--) \quad (-++) \quad (-+-) \quad (+-+) \quad (--+)$$
\end{ex}

\begin{dfn}
    Let $M = (E, \mathcal V)$ be an oriented matroid. 
    For $v \in \mathcal V$, the set $f_v := \{ e \in E \ | \ v(e) = 0 \}$ is called a \textit{flat}. The (inclusion) poset
    $$L(M) := \{ f_v \ | \ v \in \mathcal V\}$$
    is called the \textit{lattice of flats}. It is a geometric lattice and in particular defines a matroid with ground set $E$, the \textit{underlying} matroid of $M$.
\end{dfn}

Note that the map $\mathcal L(M) \to L(M), \ v \mapsto f_v$ is a morphism of lattices.

\begin{ex}
    Let $M$ be the oriented matroid associated to a set $E$ of hyperplanes $\{ f_e = 0 \}$ in a vector space $V$. Then the flats are one-to-one with the intersections of hyperplanes. In particular, the lattice $L(M)$ is the intersection lattice of the hyperplane arrangement and the underlying matroid is just the matroid of the hyperplane arrangement.
\end{ex}

\begin{ex}
    Given a directed graph, the flats of the associated oriented matroid are disjoint unions of connected induced subgraphs. The corresponding matroid is the graphical matroid associated to the undirected graph.
\end{ex}

\begin{dfn}
    Let $(E, \mathcal V)$ be an oriented matroid and let $P$ be a tope. The lattice $\mathcal L(P)$ given by the interval $[P, \hat{1}]$ in $\mathcal L(M)$ is called the \textit{Las Vergnas face lattice} of $P$. The map $\mathcal L(M) \to L(M), \ v \mapsto f_v$ is injective on $\mathcal L(P)$ and in particular, we can identify $\mathcal L(P)$ with its image $L(P)$.
\end{dfn}

\begin{ex}
    Recall that if $M$ is the oriented matroid associated to a set $E$ of hyperplanes $\{ f_e = 0 \}$, then $P$ corresponds to a region $C$ of the hyperplane arrangement complement. $\mathcal L(P) = L(P)$ is then just the (inverse) face lattice of the polyhedron $\bar{C}$.\par

    If $M$ is associated to a directed graph $G$, then recall that a tope $P$ is the same as an acyclic orientation of the underlying undirected graph. The Las Vergnas face lattice is then the contraction lattice of the directed acyclic graph obtained by transitive reduction.
\end{ex}

%%%%%%%%%%%%%%%%%%%%%%
\subsection{Nested sets and fans}\label{app:matroid:nested}
Let $\mathcal M$ be an oriented matroid with underlying matroid $M$. We write $L(\mathcal M) = L(M)$ for its lattice of flats. Fix a tope $P$ of $\mathcal M$.

\begin{dfn}
    A \textit{building set} for $L(M)$ is a set $\mathcal G \subseteq L - \{\hat 0\}$ such that for $f \in L(M) - \{\hat 0\}$ the join induces an isomorphism of posets
    $$\prod_{g \in I} [\hat 0, g] \to [\hat 0, f]$$
    where $I \subseteq \mathcal G$ is the set of maximal elements $g$ with $g \leq f$ in $L(M)$.
\end{dfn}

\begin{dfn}
    A ($\mathcal G$-)\textit{nested set} is a subset $N \subseteq \mathcal G$ such that for incomparable $f_1, \ldots, f_k \in N$ we have $f_1 \lor \ldots \lor f_k \notin \mathcal G$. We denote the poset of nested sets by $\mathcal N(L(M),\mathcal G)$.\par
    The \textit{positive nested set complex} associated to the tope $P$ is the subset
    $$\mathcal N(L(M), \mathcal G)_P := \{ N \in \mathcal N(L(M),\mathcal G) \ | \ N \subseteq L(P)\}.$$
    Recall that $L(P) \subseteq L(M)$ is the Las Vergnas face lattice of $P$.
\end{dfn}

\begin{dfn}\label{def:positive-nested-fan}
    Define $\Sigma^{\mathcal G,P}$ to be the fan in $\mathbb R^{\mathcal G}$ obtained by taking the cones
    $$C_N := \cone(f \ | \ f \in N)$$
    for all $N \in \mathcal N(L(M), \mathcal G)_P$. The projection of this fan along $$\mathbb R^{\mathcal G} \to \mathbb R^E, \ f \mapsto \sum_{e\in f} e$$
    is again a fan, called the \textit{positive nested set fan} $\Sigma_{\mathcal G,P}$.
\end{dfn}

This projection does indeed define a fan by \Cref{prop:nested-projection} (and \Cref{lem:tope-building}).

%%%%%%%%%%%%%%%%%%%%%%
\subsection{Matroid Amplitudes}\label{app:matroid:laplace}
The \textit{matroid amplitude} of a tope $P$ is the Laplace transform of the associated fan, $\Sigma_{\mc{G},P}$. Recall that the Laplace transform of a function $f$ on $\mathbb R^n$ is the function
\begin{align*}
    \mathcal L(f): \Gamma &\to \mathbb R \\
        y &\mapsto \int_{\mathbb R^n_{>0}} f(x) \exp(-y \cdot x) \mathrm{d}x
\end{align*}
where $\Gamma \subseteq \mathbb R^n$ is the domain of convergence of the integral. We are interested in the case where the function $f$ is the characteristic function of a cone $C$ in $\mathbb R^n$.
\begin{lem}\label{lem:cone-laplace}
    Let $C = \cone(x_1,\ldots,x_n)$ be an $n$-dimensional simplicial cone in $\mathbb R^n$. The integral $$\int_{C} \exp(-x\cdot y) \mathrm{d}x$$ converges absolutely for $y$ in the interior of the dual cone $\{ y \in \mathbb R^n \ | \ x\cdot y \geq 0 \text{ for all $x \in C$} \}$. It is given by the rational function
    \begin{align*}
        \mathcal L(C): \mathbb R^n &\to \mathbb R \\
    y &\mapsto |\mathrm{det}(x_1,\ldots,x_n)| \cdot \prod_{i=1}^{n}  \frac{1}{x_i\cdot y}
    \end{align*} 
    which we call the Laplace transform of the cone $C$, ignoring the domain of convergence.
\end{lem}
\begin{proof}
    We have
    \begin{align*}
        \int_{C} \exp(-x(y)) \mathrm{d}x &= |\mathrm{det}(x_1,\ldots,x_n)| \cdot \int_{\mathbb R^n_+} \exp(-\lambda_1 x_1(y) - \ldots - \lambda_n x_n(y)) \mathrm{d} \lambda_1 \ldots \mathrm{d} \lambda_n \\
        &= \prod_{i=1}^n \int_0^\infty \exp(-\lambda_i x_i(y)) d\lambda_i
        = |\mathrm{det}(x_1,\ldots,x_n)| \cdot \prod_{i=1}^{n}  \frac{1}{x_i(y)}.
    \end{align*}
\end{proof}

More generally, the Laplace transform of a cone (that is, of its characteristic function) can be computed by first triangulating the cone and then applying \Cref{lem:cone-laplace}.

\begin{dfn}
    Let $\Sigma$ be a fan in $\mathbb R^n$ such that every of its cones $C$ is unimodular. In other words, if $V$ is the vector space spanned by $C$, then $V \cap \mathbb Z^n$ is spanned by parts of a $\mathbb Z$-basis of $\mathbb Z^n$. Define a measure on $V \cap \mathbb Z^n$ such that the unit cube has volume $1$. Then the Laplace transform of $C \subseteq V$ is well-defined with respect to that measure. We call the function
    $$\mathcal L(\Sigma) := \sum_{C} \mathcal L(C)$$
    where the sum is over the maximal cones of $\Sigma$, the \textit{Laplace transform} of $\Sigma$.
\end{dfn}

\begin{dfn}
    The \textit{matroid amplitude} $\mathcal A(P)$ of the tope $P$ is the Laplace transform $\mathcal L(\Sigma_{\mc{G},P})$ of the positive nested set fan. Note that it does not depend on the choice of building set, $\mc{G}$.
\end{dfn}

\begin{ex}
    Consider the oriented matroid $\mathcal M$ associated to the directed graph
    \[ \begin{tikzpicture}
        \node (1) at (0,0) {1};
        \node (2) at (1,0) {2};
        \node (3) at (1,-1) {3};

        \draw[->] (1) -- (2) node [midway,above, fill=white] {$e_{12}$};
        \draw[->] (2) -- (3) node [midway,right, fill=white] {$e_{23}$};
        \draw[->] (1) -- (3) node [near end,left=10pt, fill=white] {$e_{13}$};
    \end{tikzpicture}\]
    Its lattice of flats $L(\mathcal M)$ is
    \[ \begin{tikzpicture}
        \node (123) at (2,1) {$\{e_{12},e_{23},e_{13}\}$};
        \node (12) at (0,0) {$\{e_{12}\}$};
        \node (23) at (4,0) {$\{e_{23}\}$};
        \node (13) at (2,0) {$\{e_{13}\}$};
        \node (0) at (2,-1) {$\emptyset$};
        \draw (123) -- (12);
        \draw (123) -- (23);
        \draw (123) -- (13);
        \draw (13) -- (0);
        \draw (12) -- (0);
        \draw (23) -- (0);
    \end{tikzpicture}\]
    We choose the whole lattice as the building set $\mathcal G$. The graph itself defines a tope $P$ of $\mathcal M$ as it is acyclic. Let us compute the matroid amplitude associated to that tope. The Las Vergnas face lattice of $P$ is
    \[ \begin{tikzpicture}
        \node (123) at (2,1) {$\{e_{12},e_{23},e_{13}\}$};
        \node (12) at (1,0) {$\{e_{12}\}$};
        \node (23) at (3,0) {$\{e_{23}\}$};
        \node (0) at (2,-1) {$\emptyset$};
        \draw (123) -- (12);
        \draw (123) -- (23);
        \draw (12) -- (0);
        \draw (23) -- (0);
    \end{tikzpicture}\]
     Thus the associated positive nested set fan has two maximal cones,
     $$C_1 := \cone(e_{12}, e_{12} + e_{23} +e_{13}) \quad \text{and} \quad C_2 := \cone(e_{23}, e_{12} + e_{23} +e_{13}).$$
    and thus the matroid amplitude of $P$ is the rational function
    $$\mathcal A(P) := \mathcal L(\Sigma_{\mc{G},P})= \frac{1}{e_{12}^* \cdot (e_{12}^* + e_{23}^* + e_{13}^*)} + \frac{1}{e_{23}^* \cdot (e_{12}^* + e_{23}^* + e_{13}^*)}$$
    on $\mathbb R^{E}$.
\end{ex}

\printbibliography

\end{document}